\documentclass{article}

\parskip \medskipamount
\usepackage[T1]{fontenc}
\usepackage{geometry}
\usepackage{titling}
\usepackage{setspace}
\usepackage{xcolor}
\usepackage{authblk}
\usepackage{amsmath,amssymb,amsthm,amsfonts,bm,bbm}
\DeclareMathOperator*{\esssup}{ess\,sup}

\DeclareMathOperator*{\argmax}{arg\,max}

\DeclareMathOperator*{\Var}{Var}

\usepackage{mathrsfs,stmaryrd}
\usepackage{upgreek}
\usepackage{stmaryrd}
\SetSymbolFont{stmry}{bold}{U}{stmry}{m}{n}
\usepackage{indentfirst}
\usepackage{subdepth}
\usepackage{graphicx}
\usepackage{diagbox}
\usepackage{pdfpages}
\usepackage[latin1]{inputenc}
\usepackage{url}
\usepackage{float}
\usepackage{appendix}
\usepackage{etoolbox}
\usepackage{microtype}
\usepackage[centerlast]{caption}
\usepackage[plainpages=false,hyperfootnotes=false]{hyperref}
\hypersetup{colorlinks=true,urlcolor=blue,linkcolor=red,citecolor=blue}
\usepackage{cleveref}
\crefname{equation}{\hspace{-0.4em}}{\hspace{-0.4em}}
\crefrangeformat{equation}{(#3#1#4)--(#5#2#6)}
\usepackage{bookmark}
\usepackage{framed}
\usepackage{enumitem}
\usepackage{apacite}
\usepackage{natbib}

\newtheorem{theorem}{Theorem}[section]
\newtheorem{lemma}[theorem]{Lemma}

\newtheorem{remark}[theorem]{Remark}
\newtheorem{definition}[theorem]{Definition}
\newtheorem{assumption}[theorem]{Assumption}

\newtheorem{example}[theorem]{Example}
\renewenvironment{proof}{\noindent {\bf Proof.}}{\hfill $\Box$}
\allowdisplaybreaks[4]

\hypersetup{
  pdftitle={Equilibrium strategies for stochastic control problems with higher-order moments and applications to portfolio selection},
  pdfauthor={Y. Wang}
}

\title{Equilibrium strategies for stochastic control problems with higher-order moments and applications to portfolio selection
\thanks{Jingzhen Liu and Jiaqin Wei are co-corresponding authors.
        This work is supported by the National Natural Science Foundation of China [Grants 12401611, 12571520 and 12471447];
                                  Major Program of the Key Research Institute on Humanities and Social Science of China Ministry of Education [Grant 22JJD790091];
                                  the 111 Project [Grant B17050];
                              and CTBU [Grant 2355010].}}
\date{\vspace{-6ex}}
\author{
    Yike Wang\thanks{School of Finance, Chongqing Technology and Business University, Chongqing 400067, China.}, ~
    Jingzhen Liu\thanks{China Institute of Actuarial Science, Central University of Finance and Economics, Beijing 100081, China.}, ~
    Jiaqin Wei\thanks{Key Laboratory of Advanced Theory and Application in Statistics and Data Science-MOE, School of Statistics, East China Normal University, Shanghai 200062, China.}
}

\begin{document}
\maketitle

\begin{abstract}
In this paper we derive a novel characterization result for time-consistent stochastic control problems with higher-order moments, originally formulated by Wang et al. [SIAM J. Control. Optim., 63 (2025), 1560--1589],
and newly explore many solvable instances including a mean-variance-excess kurtosis portfolio selection problem.
By improving an asymptotic result of the variational process for the uniform boundedness and integrability properties, 
we obtain both the sufficiency and necessity of an equilibrium condition for an open-loop Nash equilibrium control (ONEC).
This condition is simply formulated by the diagonal processes of a flow of backward stochastic differential equations (BSDEs) whose data do not necessarily satisfy the usual square-integrability condition.
In particular, for linear controlled dynamics with deterministic parameters,
we show that the ONEC can be derived by solving a polynomial algebraic equation under a class of nonlinear objective functions.
Interestingly, the mean-variance equilibrium strategy is an ONEC for our general higher-order moment problem if and only if a homogeneity condition holds.
Additionally, in the case with random parameters, we characterize the ONEC by finitely many BSDEs with a recurrence relation. 
As an intuitive illustration, the solution to the mean-variance-skewness problems is given by a quadratic BSDE.
\end{abstract}

\noindent {\bf Keywords:} open-loop Nash equilibrium control, higher-order moment, backward stochastic differential equations, homogeneity condition, mean-variance-skewness problem

\vspace{3mm}

\noindent {\bf AMS2010 classification:} Primary: 93E20, 91G80; Secondary: 91B08, 49N90
\vspace{3mm}

\section{Introduction}
\label{sec: Introduction}

Since \cite{Markowitz-1952} pioneered the mean-variance analysis domain in the 1950s,
the mean maximization and variance minimization problem has attracted much attention in fields such as mathematical finance and control engineering,
and research on the mean-variance efficiency frontier and efficient portfolios has continued in recent years.
Owing to multi-objective convex optimization theory, the process of seeking an efficient portfolio can be reduced to solving a mean-variance utility optimization problem (see \cite[p. 337]{Yong-Zhou-1999}).
Therefore, mean-variance objective functions are becoming increasingly common in studies related to stochastic control and dynamic optimization problems.

On the one hand, to obtain theoretical developments and implement practical applications, an increasing number of characterizations of risks, such as skewness and kurtosis, should be considered. 
This is not only a natural extension of mean-variance analysis, but also flexibly models the asymmetry, heavy tails and outliers of the real-world distribution.
As a comprehensive review, \cite{Mandal-Thakur-2024} presented a detailed systematic literature analysis for the portfolio selection models with higher-order moments,
and collects several reasons to incorporate higher-order moments into the portfolio problem.
For example, \cite{Konno-Shirakawa-Yamazaki-1993,Konno-Suzuki-1995} inferred the importance of the skewness in portfolio optimization models by the third-order approximation of expected utility, 
and proposed different theoretical schemes to obtain a portfolio with a large skewness. 
\cite{Chunhachinda-Dandapani-Hamid-Prakash-1997} also showed by an empirical approach that the skewness in investor preferences significantly affects the construction of the optimal portfolio.
\cite{Maringer-Parpas-2009} incorporated both the skewness and kurtosis into a weighted mean-variance model.
The kurtosis aversion was considered therein, based on the hypothesis that the investors dislike extreme events with a high probability on either side. 
\cite{Kim-Fabozzi-Cheridito-Fox-2014} suggested solving the mean-variance-skewness-kurtosis optimization problem formulated as 
\begin{equation*}
\max_{\text{portfolio}} \text{mean} - \text{variance} + \text{skewness} - \text{kurtosis}.
\end{equation*}
This optimization problem was also considered in \cite{Chen-Zhong-Chen-2020} to examine the data-driven asset selection.
\cite{Liu-Jiang-An-Cheung-2020} conducted an empirical study and found that incorporating the kurtosis in the portfolio problem can lead to the most robust improvement in investment performance. 
Besides, a wide variety of mean-variance-skewness-kurtosis models and methods have been proposed to solve portfolio problems. 
See, e.g., \cite{Joro-Na-2006,Briec-Kerstens-Jokung-2007,Yu-Wang-Lai-2008,Zhai-Bai-Wu-2018,Chen-Zhou-2018,Zhou-Palomar-2021,Zhai-Bai-Hao-2021}.
In addition, provided that all information about the distribution of a contingent claim is included in the higher-order moments,
the higher-order moment problems might have a potential connection with the portfolio problems under some specific criteria, 
e.g., stochastic dominance constraints in \cite{Wang-Xia-2021}, nonlinear law-dependent preferences in \cite{Liang-Xia-Yuan-2025}, quantile maximization in \cite{He-Jiang-Kou-2025} and rank-dependent utilities in \cite{Wei-Xia-Zhao-2025}.

However, unlike the well-developed mean-variance model---which has been deeply explored with stochastic control theory---research on higher-order moment portfolio selection has primarily focused on algorithm design within discrete-time frameworks.
This predicament is admittedly attributable primarily to technical limitations.
For example, \cite{Konno-Suzuki-1995} mentioned that the skewness of a contingent claim is not a concave function and hence it is difficult to solve the static problem through standard computational approaches.
\cite{Boissaux-Schiltz-2010} derived the stochastic maximum principle for a dynamic mean-variance-skewness-kurtosis problem but failed to derive an explicit solution.
In pursuit of concrete breakthroughs in analytical solutions, researchers are shifting their focus from traditional maximization/minimization frameworks to alternative problem formulations.
Recently, \cite{Wang-Liu-Bensoussan-Yiu-Wei-2025} provided a game-theoretic solution to the higher-order problems within a continuous-time stochastic control framework and facilitated two financial applications:
(i) Investigating the higher-order moment problem is a heuristic approach for solving portfolio problems with fairly general nonlinear law-dependent preferences (see Example 4.6 therein).
(ii) Through a higher-moment problem, one can solve a portfolio problem with a fairly general penalty function of the deviation from the expectation of a contingent claim (see Remark 5.10 therein).
In the present paper, we follow the game-theoretic perspective---which is introduced along with time-consistency later---and derive some novel results for the higher-order moment problems.

On the other hand, the presence of variance, as well as other higher-order moments, in the objective function usually implies that the optimization problem is time-inconsistent.
The study of time-inconsistency can also be traced back to the 1950s; see \cite{Strotz-1955} for research on the consumption optimization problem with non-exponential intertemporal utility discounting.
The major negative effect of time-inconsistency is that the decision maker immediately changes their plan to realize a temporary maximum/minimum once the information is refreshed and finally implements a myopic strategy.
For example, one can follow the research of \cite{Zhou-Li-2000} to solve a mean-variance optimization problem with a new initial epoch and then find the difference between the new optimal strategy and the previous one. 
The usual approach for addressing time-inconsistency is to take the game-theoretic perspective and formulate the problem as a game between the incarnations of the decision maker at different time instants;
see, e.g.,  \cite[Section 2.3]{Bjork-Murgoci-Zhou-2014}.
An equilibrium point is treated as a time-consistent solution because any spike deviation from it will not be significantly better.
If a control process (resp. a feedback scheme) is sought, an equilibrium point is called an open-loop (resp. closed-loop) Nash equilibrium control.
Related works include \cite{Yong-2011,Yong-2012,Bjork-Murgoci-2014,Bjork-Murgoci-Zhou-2014,Bjork-Khapko-Murgoci-2017,Wei-2017,Wang-Zheng-2021} for closed-loop Nash equilibrium control (CNEC),
\cite{Hu-Jin-Zhou-2012,Hu-Jin-Zhou-2017,Sun-Guo-2019,Alia-2019,Wang-2019,Alia-2020} for open-loop Nash equilibrium control (ONEC) and \cite{Yong-2017,Wang-Liu-Bensoussan-Yiu-Wei-2025} for both types.

As a well-known result, for a classic mean-variance portfolio problem with deterministic parameters, the ONEC and CNEC are the same deterministic function; see also \cite{Hu-Jin-Zhou-2012,Bjork-Khapko-Murgoci-2017}.
The recent work \cite{Wang-Liu-Bensoussan-Yiu-Wei-2025} extends this result to a class of higher-order moment problems,
and shows that a deterministic function is a time-consistent solution under certain conditions, although it is difficult to find the maximizer for the same objective function. 
However, those conditions are so strong.
For example, the $\mathbb{L}^{ 2n-2 + \epsilon }$-integrability of the terminal state value was proposed, but the objective function only includes the moments up to the $n$-th order.
This technical condition is crucial for ensuring the square-integrability of the backward stochastic differential equations (BSDEs) which serve as adjoint equations in the maximum principle, yet it narrows the scope of admissible strategies.
In fact, it is $\mathbb{L}^{n}$-integrability that is an eligible requirement for a financial problem with the higher-order moments up to the $n$-th order.
Additionally, global Lipschitz continuity prohibits 
some nonlinear terms of higher-order moments, such as skewness and kurtosis formulated by standardized moments, from being included in the objective function.
This limits the potential application scope of this study.
Furthermore, for open-loop equilibria, \cite{Wang-Liu-Bensoussan-Yiu-Wei-2025} only shows the sufficiency of a complicated equilibrium condition, 
which contains the limit of the ratio of the conditional expectation of an integral to the variational factor.
In view of \cite{Hu-Jin-Zhou-2017} and \cite{Wang-2019}, the equilibrium condition is expected to be simplified by using the diagonal processes generated by a flow of BSDEs indexed by different initial epochs. 
The simplified condition may facilitate the application of our results. 

In this paper, we characterize and seek an ONEC for the objective function formulated by the sum of the classic mean-variance utility and a general function of higher-order central moments.
Considering the above mentioned improvements needed for financial practice, we make the corresponding contributions as follows.

\emph{First}, we relax the global Lipschitz continuity condition in the prior research to a local Lipschitz continuity condition.
Moreover, as the objective function includes the moments up to the $2n$-th order, we merely require the ``$\mathbb{L}^{2n}$-integrability'' of the terminal state value.
As a consequence, the terminal value, i.e., the data, in each adjoint BSDE is merely $\mathbb{L}^{ 1 + \delta }$-integrable for a small $\delta > 0$.
We strengthen \cite[Lemma 4.1]{Wang-Liu-Bensoussan-Yiu-Wei-2025} to \Cref{lem: actuarial estimate} with a weaker integrability condition and a stronger uniform estimate in the present paper, which serves the perturbation argument.

\emph{Second}, by exploiting the theory of $\mathbb{L}^{p}$ solutions for BSDEs (see \cite{Briand-Delyon-Hu-Pardoux-Stoica-2003,ElKaroui-Peng-Quenez-1997,Chen-2010}),
we show in \Cref{lem: integrabiity and uniqueness of diagonal process} that the diagonal process is also ``$\mathbb{L}^{ 1 + \delta }$-integrable''
and then demonstrate in \Cref{thm: maximum principle} that the equilibrium condition, which is explicitly and briefly expressed by the exact value of the diagonal process triplet, is sufficient and necessary for attaining open-loop equilibria.
Notably, distinct from \cite{Hu-Jin-Zhou-2017}, a stochastic version of the Lebesgue differentiation theorem (see Lemma 3.4 therein) is unnecessary in our approach.
We still present a stochastic Lebesgue differentiation theorem in \Cref{app: Stochastic Lebesgue differentiation theorem} to satisfy the ``$\mathbb{L}^{ 1 + \delta }$-integrability'' condition that is weaker than \cite[Lemma 3.4]{Hu-Jin-Zhou-2017}. 
This extension provides not only an alternative approach for deriving the necessity of the equilibrium condition for linear controlled stochastic differential equations (SDEs),
but may also be useful for various stochastic control problems without adequate integrability conditions. 

\emph{Third}, to illustrate specific instances and to facilitate applying our results to financial problems, we consider linear controlled SDEs and derive the corresponding ONEC.
Given the deterministic linearity coefficients, we find that the ONEC can be obtained by solving a polynomial algebraic equation under a class of nonlinear objective functions 
(see \Cref{rem: method of polynomial algebraic equation} and \Cref{ex: method of polynomial algebraic equation}), e.g., by heuristically formulating,
\begin{equation*}
\text{objective function} = \text{mean} - \text{variance} - \text{excess kurtosis} \times | \text{variance} |^{2}.
\end{equation*}
In particular, we show in \Cref{thm: MV solution is ONEC} that the homogeneity condition \cref{eq: homogeneity condition :eq} is sufficient and necessary for this coincidence for our general higher-order moment problems.
This results extends the implication that the mean-variance equilibrium strategy is also an ONEC for the mean-variance-standardized moments of objective functions (see \cite[Example 5.7 and Theorem 5.9]{Wang-Liu-Bensoussan-Yiu-Wei-2025}).

\emph{Finally}, we consider linear controlled SDEs with some random coefficients, in which case it is almost impossible to derive an explicit solution.
We derive through backward stochastic partial differential equations (BSPDEs) and characterize the ONEC by a sequence of BSDEs with a recurrence relation.
Afterward, we find that the ONEC for mean-variance-skewness problems satisfying a certain condition can be derived by solving a quadratic BSDE.
Then, we have the existence and uniqueness of the solution. See \Cref{ex: BSDEs characterization}.

The rest of this paper is organized as follows.
In \Cref{sec: Problem formulation}, we formulate our control problem.
In \Cref{sec: Mathematical preliminaries}, we provide some mathematical preliminaries, including the perturbation argument and the diagonal processes generated by a flow of BSDEs, for characterizing open-loop equilibria.
In \Cref{sec: Sufficient and necessary condition for ONECs}, we show the sufficiency and necessity of the equilibrium condition for open-loop equilibria.
In \Cref{sec: Closed-form solution}, we consider linear controlled SDEs to illustrate several applications to portfolio selection, and then derive the closed-form expression of an ONEC or provide its characterization.
This illustrates many applications to portfolio selection problems.
To end this paper, we present some concluding remarks in \Cref{sec: Concluding remark}.

\section{Problem formulation}
\label{sec: Problem formulation}

Let $T$ be a fixed finite time horizon and $( \Omega, \mathcal{F}, \mathbb{F}, \mathbb{P} )$ be a filtered probability space that satisfies the usual hypotheses, 
where $\mathbb{F} := \{ \mathcal{F}_{t} \}_{ t \in [ 0,T ] }$ is generated by the one-dimensional standard Brownian motion $\{ {W}_{t} \}_{ t \in [ 0,T ] }$. Denote by
$\mathbb{E}$ the expectation operator and $\mathbb{E}_{t} [ \cdot ] := \mathbb{E} [ \cdot | \mathcal{F}_{t} ]$ the expectation conditioned on $\mathcal{F}_{t}$.
For $p,q > 1$, we introduce the following space notation. 
\begin{itemize}
\item $\mathbb{L}_{ \mathcal{F}_{t} }^{p} ( \Omega )$ denotes the set of all $\mathcal{F}_{t}$-measurable random variables $f: \Omega \to \mathbb{R}$ with $\mathbb{E} [ | f |^{p} ] < \infty $.
\item $\mathbb{L}_{\mathbb{F}}^{p} ( 0,T; \mathbb{L}^{q} ( \Omega ) )$ denotes the set of all $\mathbb{F}$-progressively measurable processes $f: [ 0,T ] \times \Omega \to \mathbb{R}$ with
      $\mathbb{E} [ ( \int_{0}^{T} | f ( s, \cdot ) |^{p} ds )^{ \frac{q}{p} } ] < \infty $.
\item ${C}_{\mathbb{F}} ( 0,T; \mathbb{L}^{p} ( \Omega ) )$ denotes the set of all $\mathbb{F}$-progressively measurable and $\mathbb{P}$-a.s. sample-continuous processes $f: [ 0,T ] \times \Omega \to \mathbb{R}$ with
      $\mathbb{E} [ \sup_{ s \in [ 0,T ] } | f ( s, \cdot ) |^{p} ] < \infty $.
\item $\mathbb{L}_{\mathbb{F},loc}^{p} ( 0,T; \mathbb{L}^{q} ( \Omega ) )$ denotes the set of all $\mathbb{F}$-progressively-measurable processes $f: [ 0,T ) \times \Omega \to \mathbb{R}$
      with $\mathbb{E} [ ( \int_{0}^{\tau } | f ( s, \cdot ) |^{p} ds )^{ \frac{q}{p} } ] < \infty$ for any fixed $\tau \in ( 0,T )$.
\item $\mathbb{L}^{2} ( 0,T ) \cap \mathbb{L}^{\infty }_{loc} ( 0,T )$ denotes the set of all deterministic measurable functions $f: [ 0,T ] \to \mathbb{R}$ 
      such that $\int_{0}^{T} | {f}_{t} |^{2} dt < \infty $ and $\esssup_{ s \in [ 0, \tau ] } | {f}_{s} | < \infty $ for any fixed $\tau \in ( 0,T )$.
\end{itemize}

Hereafter, we suppress the statement of the argument $\omega $ as usual.
In addition, for our control problem with high-order moments up to the $2n$-th order,
we let ${\alpha }_{0} (y) := 1$, ${\alpha }_{2j-1} (y) := 0$ and ${\alpha }_{2j} (y) := (2j-1)!! {y}^{j}$ for all positive integers $j$,
where $(2j-1)!!$ is the double factorial of $2j-1$, and we write $\vec{\alpha } (y) := \big( {\alpha }_{2} (y), {\alpha }_{3} (y), \ldots, {\alpha }_{2n} (y) \big)$ for the sake of brevity.
Notably, to mitigate misunderstandings, we strongly suggest that readers keep in mind that the argument ${z}_{j}$ in the sequel corresponds to ${\alpha }_{j} (y)$ and the $j$-th central moment unless otherwise mentioned.

In this paper, we shall solve various mean-variance-skewness-kurtosis portfolio selection problem with the preferences including
\begin{itemize}
\item $\mathbb{E} [ {X}^{u}_{T} ] 
       - \frac{\gamma }{2} \mathbb{E} [ ( {X}^{u}_{T} - \mathbb{E} [ {X}^{u}_{T} ] )^{2} ]
       - {\varphi }_{3} \mathbb{E} [ ( {X}^{u}_{T} - \mathbb{E} [ {X}^{u}_{T} ] )^{3} ]
       - {\varphi }_{4} \mathbb{E} [ ( {X}^{u}_{T} - \mathbb{E} [ {X}^{u}_{T} ] )^{4} ]$,
\item $\mathbb{E} [ {X}^{u}_{T} ] - \frac{\gamma }{2} \mathbb{E} [ ( {X}^{u}_{T} - \mathbb{E} [ {X}^{u}_{T} ] )^{2} ]
       - {\varphi }_{3} \frac{ \mathbb{E} [ ( {X}^{u}_{T} - \mathbb{E} [ {X}^{u}_{T} ] )^{3} ] }{ | \mathbb{E} [ ( {X}^{u}_{T} - \mathbb{E} [ {X}^{u}_{T} ] )^{2} ] |^{\frac{3}{2}} }
       - {\varphi }_{4} \frac{ \mathbb{E} [ ( {X}^{u}_{T} - \mathbb{E} [ {X}^{u}_{T} ] )^{4} ] }{ | \mathbb{E} [ ( {X}^{u}_{T} - \mathbb{E} [ {X}^{u}_{T} ] )^{2} ] |^{2} }$,
\item $\mathbb{E} [ {X}^{u}_{T} ] - \frac{\gamma }{2} \mathbb{E} [ ( {X}^{u}_{T} - \mathbb{E} [ {X}^{u}_{T} ] )^{2} ]
       - {\varphi }_{4} | \mathbb{E} [ ( {X}^{u}_{T} - \mathbb{E} [ {X}^{u}_{T} ] )^{2} ] |^{2} \times
         \big( \frac{ \mathbb{E} [ ( {X}^{u}_{T} - \mathbb{E} [ {X}^{u}_{T} ] )^{4} ] }{ | \mathbb{E} [ ( {X}^{u}_{T} - \mathbb{E} [ {X}^{u}_{T} ] )^{2} ] |^{2} } - 3 \big)$,
\end{itemize}
where ${X}^{u}$ represents the wealth process corresponding to the investment amount process $u$. See \Cref{sec: Closed-form solution}.

For generality, we consider the following SDE for the state-control pair $( X,u )$:
\begin{equation}\label{eq: general controlled SDE :eq}
d {X}_{t} = b ( t, {X}_{t}, {u}_{t} ) dt + \sigma ( t, {X}_{t}, {u}_{t} ) d {W}_{t}, \quad \forall t \in [ 0,T ], \quad
  {X}_{0} = {x}_{0} ~ (\text{given a priori}),
\end{equation}
where $b, \sigma : [ 0,T ] \times \Omega \times \mathbb{R} \times \mathbb{R} \to \mathbb{R}$ are $\mathbb{F}$-progressively measurable in the first two arguments $( t, \omega )$ and uniformly Lipschitz continuous in the last two arguments $( x,u )$, 
i.e., there exists a constant $K > 0$ such that for all $t \in [ 0,T ]$ and $x,y,u,v \in \mathbb{R}$,
\begin{equation}\label{eq: Lipschitz continuity :eq}
| b ( t,x,u ) - b ( t,y,v ) | + | \sigma ( t,x,u ) - \sigma ( t,y,v ) | \le K ( | x - y | + | u - v | ), \quad \mathbb{P}\text{-}a.s.,
\end{equation}
and the $\mathbb{F}$-adapted process $u: [ 0,T ] \times \Omega \to \bar{\mathbb{R}}$ satisfies $\mathbb{E} [ ( \int_{0}^{T} | b ( t, 0, {u}_{t} ) | dt )^{2} + \int_{0}^{T} | \sigma ( t, 0, {u}_{t} ) |^{2} dt ] < \infty$.
In fact, by using the Picard iteration method as in the proof of \cite[Theorem 5.2.1]{Oksendal-2003}, 
one can conclude that \cref{eq: linear controlled SDE :eq} permits the unique solution denoted by ${X}^{u}$ in ${C}_{\mathbb{F}} ( 0,T; \mathbb{L}^{2} ( \Omega ) )$,
where the uniqueness is up to indistinguishability according to \cite[Lemma 3.2.10]{Cohen-Elliott-2015}.
In addition, we assume that $b ( t,x,u )$ and $\sigma ( t,x,u )$ are twice continuously differentiable in $x$ and that the partial derivatives $( {b}_{xx}, {\sigma }_{xx} )$ are uniformly bounded for all $( t, \omega, u )$. 

In terms of the objective function, for generality, we consider
\begin{align*}
J (u) & := \mathbb{E} [ {X}^{u}_{T} ] - \frac{\gamma }{2} \mathbb{E} \big[ ( {X}^{u}_{T} - \mathbb{E} [ {X}^{u}_{T} ] )^{2} \big] \\
      &~\quad - \varphi \big( \mathbb{E} \big[ ( {X}^{u}_{T} - \mathbb{E} [ {X}^{u}_{T} ] )^{2} \big], \mathbb{E} \big[ ( {X}^{u}_{T} - \mathbb{E} [ {X}^{u}_{T} ] )^{3} \big], 
                              \ldots, \mathbb{E} \big[ ( {X}^{u}_{T} - \mathbb{E} [ {X}^{u}_{T} ] )^{2n} \big] \big),
\end{align*}
which is a sum of the classic mean-variance utility ${J}_{MV} (u) = \mathbb{E} [ {X}^{u}_{T} ] - \frac{\gamma }{2} \Var [ {X}^{u}_{T} ]$ with $\gamma \in \mathbb{R}_{+}$
and a general evaluation for the higher-order central moments $\{ \mathbb{E} [ ( {X}^{u}_{T} - \mathbb{E} [ {X}^{u}_{T} ] )^{j} ] \}_{ j = 2,3,\ldots,2n }$.
This objective function acts as the total utility of an economic agent and is expected to be maximized only based on the initial information.
Hereafter, $\varphi ( {z}_{2}, {z}_{3}, \ldots, {z}_{2n} )$ is a continuously differentiable function that is Lipschitz continuous on every compact subset that does not contain any point with zero coordinate components.
In other words, $\varphi $ is allowed to be undefined or ill posed at points $( \ldots, 0, \ldots )$.
To evaluate a control $u$ for any initial epoch $t$,
we let ${J}^{t} (u) := - \infty $ if one of $( \varphi, {\varphi }_{2}, \ldots, {\varphi }_{2n} )$, where ${\varphi }_{j} := \frac{\partial \varphi }{\partial {z}_{j} }$, is not well defined at the point 
$( \mathbb{E}_{t} [ ( {X}^{u}_{T} - \mathbb{E}_{t} [ {X}^{u}_{T} ] )^{2} ], \ldots, \mathbb{E}_{t} [ ( {X}^{u}_{T} - \mathbb{E}_{t} [ {X}^{u}_{T} ] )^{2n} ] )$;
otherwise, we introduce the following objective function:
\begin{align}\label{eq: objective function :eq}
{J}^{t} (u) & := \mathbb{E}_{t} [ {X}^{u}_{T} ] - \frac{\gamma }{2} \mathbb{E}_{t} \big[ ( {X}^{u}_{T} - \mathbb{E}_{t} [ {X}^{u}_{T} ] )^{2} \big] \\
            &~\quad - \varphi \big( \mathbb{E}_{t} \big[ ( {X}^{u}_{T} - \mathbb{E}_{t} [ {X}^{u}_{T} ] )^{2} \big], \mathbb{E}_{t} \big[ ( {X}^{u}_{T} - \mathbb{E}_{t} [ {X}^{u}_{T} ] )^{3} \big],
                                    \ldots, \mathbb{E}_{t} \big[ ( {X}^{u}_{T} - \mathbb{E}_{t} [ {X}^{u}_{T} ] )^{2n} \big] \big), \notag
\end{align}
which acts as the agent's utility based on all the information up to the epoch $t$.
Notably, in this paper, we do not consider CNECs presented in, e.g., \cite{Bjork-Khapko-Murgoci-2017}, 
because it is not necessarily possible to define the equilibrium value function and the extended Hamiltonian-Jacobi-Bellman equation to fit the terminal value ${J}^{T} (u) = {X}^{u}_{T} - \varphi ( 0, \ldots, 0 )$ under our relaxed assumptions for $\varphi $.

Let $\mathcal{U}$ be the set of all $\mathbb{F}$-adapted control processes $u$ with ${X}^{u}_{T} \in \mathbb{L}_{ \mathcal{F}_{T} }^{2n} ( \Omega )$ and ${J}^{t} (u) > - \infty $ for all $t \in [ 0,T )$.
Notably, the integrability condition for $u \in \mathcal{U}$ is much weaker than that for ${X}^{u}_{T} \in \mathbb{L}_{ \mathcal{F}_{T} }^{ 4n-2 + \epsilon } ( \Omega )$ with some $\epsilon > 0$, as required in \cite{Wang-Liu-Bensoussan-Yiu-Wei-2025},
and this improvement indeed exploits \Cref{lem: actuarial estimate,lem: integrabiity and uniqueness of diagonal process}, which are the key results for characterizing the ONEC and are innovative relative to the moment estimates in the literature. 

As mentioned in \Cref{sec: Introduction}, the solutions for the original maximization problems formulated by
$\sup_{ u \in \mathcal{U} } {J}^{t} (u)$
with different starting time instants $t$ are inconsistent, unless $\varphi = - \frac{\gamma }{2} {z}_{2} + K$ for some constant $K$.
To address the time-inconsistency, we refer to \cite{Hu-Jin-Zhou-2012,Hu-Jin-Zhou-2017} and investigate the ONECs defined below.
Intuitively speaking, we aim to seek a control $\bar{u}$ that is not significantly worse than its spike variation $\bar{u}^{ t, \varepsilon, \zeta }$ given by $\bar{u}^{ t, \varepsilon, \zeta }_{s} = \bar{u}_{s} + \zeta {1}_{\{ s \in [ t, t + \varepsilon ) \}}$
in the sense of first-order derivative expansion, i.e., 
\begin{equation*}
{J}^{t} ( \bar{u} ) \ge {J}^{t} ( \bar{u}^{ t, \varepsilon, \zeta } ) + o ( \varepsilon ).
\end{equation*}
provided that
\begin{equation*}
\bigcap_{ t \in [ 0,T ) } \argmax_{ u \in \mathcal{U} } {J}^{t} (u) = \varnothing \quad and \quad 
\lim_{ \varepsilon \downarrow 0 } \big( {J}^{t} ( \bar{u} ) - {J}^{t} ( \bar{u}^{ t, \varepsilon, \zeta } ) \big) = 0.
\end{equation*}
Notably, for any $\bar{u} \in \mathcal{U}$, $t \in [ 0,T )$, $\varepsilon \in ( 0, T - t ] $ and $\zeta \in \mathbb{L}_{ \mathcal{F}_{t} }^{2n} ( \Omega )$, 
the spike variation $\bar{u}^{ t, \varepsilon, \zeta }$ produces ${X}^{ \bar{u}^{ t, \varepsilon, \zeta } }_{T} \in \mathbb{L}_{ \mathcal{F}_{T} }^{2n} ( \Omega )$ under the Lipschitz continuity condition \cref{eq: Lipschitz continuity :eq}.

\begin{definition}\label{def: open-loop Nash equilibrium control}
$\bar{u} \in \mathcal{U}$ is an ONEC, if
\begin{equation}\label{eq: open-loop Nash equilibrium control :eq}
0 \le \liminf_{ \varepsilon \downarrow 0 }
      \frac{ {J}^{t} ( \bar{u} ) - {J}^{t} ( \bar{u}^{ t, \varepsilon, \zeta } ) }{\varepsilon }, \quad \forall \zeta \in \mathbb{L}_{ \mathcal{F}_{t} }^{2n} ( \Omega ), ~ \mathbb{P}\text{-}a.s., ~ a.e. ~ t \in [ 0,T ).
\end{equation}
\end{definition}

\begin{remark}\label{rem: ill-posed cases}
The cases in which one of $( \varphi, {\varphi }_{2}, \ldots, {\varphi }_{2n} )$ is not well defined at the point
$( \mathbb{E}_{t} [ ( {X}^{u}_{T} - \mathbb{E}_{t} [ {X}^{u}_{T} ] )^{2} ], \ldots, \mathbb{E}_{t} [ ( {X}^{u}_{T} - \mathbb{E}_{t} [ {X}^{u}_{T} ] )^{2n} ] )$
are various and complicated.
For example, the mean-variance-standardized moment objective function
\begin{equation*}
{J}_{MVSM}^{t} (u) := \mathbb{E}_{t} [ {X}^{u}_{T} ] - \frac{\gamma }{2} \mathbb{E}_{t} \big[ ( {X}^{u}_{T} - \mathbb{E} [ {X}^{u}_{T} ] )^{2} \big]
                    - \sum_{j=3}^{n} {\gamma }_{j} \frac{ \mathbb{E}_{t} \big[ ( {X}^{u}_{T} - \mathbb{E}_{t} [ {X}^{u}_{T} ] )^{j} \big] }
                                                        { \big( \mathbb{E}_{t} \big[ ( {X}^{u}_{T} - \mathbb{E}_{t} [ {X}^{u}_{T} ] )^{2} \big] \big)^{ \frac{j}{2} } }
\end{equation*}
cannot be defined for $\mathbb{E}_{t} [ ( {X}^{u}_{T} - \mathbb{E}_{t} [ {X}^{u}_{T} ] )^{2} ] = 0$,
but \cite[Example 5.7 and Theorem 5.9]{Wang-Liu-Bensoussan-Yiu-Wei-2025} implies that this control problem has the same ONEC and CNEC as those of the mean-variance problem, provided that the controlled SDE is linear.
On the other hand, corresponding to the mean-variance-standard deviation objective function, i.e.,
\begin{equation*}
{J}_{MVSD}^{t} (u) := \mathbb{E}_{t} [ {X}^{u}_{T} ] - \frac{\gamma }{2} \mathbb{E}_{t} \big[ ( {X}^{u}_{T} - \mathbb{E}_{t} [ {X}^{u}_{T} ] )^{2} \big] 
                    - \big( \mathbb{E}_{t} \big[ ( {X}^{u}_{T} - \mathbb{E}_{t} [ {X}^{u}_{T} ] )^{2} \big] \big)^{ \frac{1}{2} },
\end{equation*}
with a slight abuse of notation, 
one obtains ${Y}^{ t, \bar{u} }_{T} = 1 - 0 \times \infty $ and $| {Z}^{ t, \bar{u} }_{T} | = \infty $ for the BSDEs \cref{eq: BSDEs :eq} with $\mathbb{E}_{t} [ ( {X}^{ \bar{u} }_{T} - \mathbb{E}_{t} [ {X}^{ \bar{u} }_{T} ] )^{2} ] = 0$.
However, under the linear SDE \cref{eq: linear controlled SDE :eq} with deterministic parameters, the control $\bar{u}$ given by ${I}_{t} {X}^{\bar{u}}_{t} + {D}_{t} \bar{u}_{t} + {F}_{t} = 0$, 
implying that ${X}^{ \bar{u} }_{T} = \mathbb{E}_{t} [ {X}^{ \bar{u} }_{T} ]$ up to modification on $[ 0,T ] \times \Omega $, produces
\begin{align*}
{J}^{t} ( \bar{u}^{ t, \varepsilon, \zeta } ) - {J}^{t} ( \bar{u} )
& = \zeta \int_{t}^{ t + \varepsilon } {e}^{ \int_{s}^{T} ( {A}_{v} - \frac{ {B}_{v} }{ {D}_{v} } {I}_{v} ) dv } {B}_{s} ds
  - \frac{\gamma }{2} {\zeta }^{2} \int_{t}^{ t + \varepsilon } {e}^{ 2 \int_{s}^{T} ( {A}_{v} - \frac{ {B}_{v} }{ {D}_{v} } {I}_{v} ) dv } | {D}_{s} |^{2} ds \\
& \quad - | \zeta | \bigg( \int_{t}^{ t + \varepsilon } {e}^{ 2 \int_{s}^{T} ( {A}_{v} - \frac{ {B}_{v} }{ {D}_{v} } {I}_{v} ) dv } | {D}_{s} |^{2} ds \bigg)^{ \frac{1}{2} } \\
& \le - | \zeta | \iota {e}^{ - T ( \sup |A| + \frac{1}{\iota } \sup |B| \sup |I| ) } {\varepsilon }^{ \frac{1}{2} } + O ( \varepsilon ).
\end{align*}
This results in $\frac{1}{\varepsilon } ( {J}^{t} ( \bar{u} ) - {J}^{t} ( \bar{u}^{ t, \varepsilon, \zeta } ) ) \to + \infty $ as $\varepsilon \downarrow 0$ for any $\zeta \ne 0$.
In other words, the given $\bar{u} \notin \mathcal{U}$ could also be an ONEC.
In fact, given \cref{eq: linear controlled SDE :eq}, once $| \varphi ( \vec{\alpha } (0) ) | < \infty $, one obtains
\begin{align*}
  {J}^{t} ( \bar{u}^{ t, \varepsilon, \zeta } ) - {J}^{t} ( \bar{u} )
& = \zeta \int_{t}^{ t + \varepsilon } {e}^{ \int_{s}^{T} ( {A}_{v} - \frac{ {B}_{v} }{ {D}_{v} } {I}_{v} ) dv } {B}_{s} ds 
  - \frac{\gamma }{2} {\zeta }^{2} \int_{t}^{ t + \varepsilon } {e}^{ 2 \int_{s}^{T} ( {A}_{v} - \frac{ {B}_{v} }{ {D}_{v} } {I}_{v} ) dv } | {D}_{s} |^{2} ds \\
& \quad + \varphi \big( \vec{\alpha } (0) \big) - \varphi \bigg( \vec{\alpha } \Big( {\zeta }^{2} \int_{t}^{ t + \varepsilon } {e}^{ 2 \int_{s}^{T} ( {A}_{v} - \frac{ {B}_{v} }{ {D}_{v} } {I}_{v} ) dv } | {D}_{s} |^{2} ds \Big) \bigg)
\end{align*}
for any $t \in [ 0,T )$ such that ${I}_{s} {X}^{\bar{u}}_{s} + {D}_{s} \bar{u}_{s} + {F}_{s} = 0$ $\mathbb{P}$-a.s. for a.e. $s \in [ t,T ]$.
Given that the fairly general $\varphi $ might be ill-posed around the origin, one cannot obtain a universal asymptotic estimate for ${J}^{t} ( \bar{u}^{ t, \varepsilon, \zeta } ) - {J}^{t} ( \bar{u} )$ akin to \Cref{lem: actuarial estimate}.
\end{remark}

\section{Mathematical preliminaries}
\label{sec: Mathematical preliminaries}

\subsection{Perturbation argument}

Suppose that $\bar{u} \in \mathcal{U}$ is fixed.
For notational simplicity, we write that 
\begin{align*}
& f (s) := f ( s, {X}^{ \bar{u} }_{s}, \bar{u}_{s} ), \quad {\delta }^{\zeta } f (s) := f ( s, {X}^{ \bar{u} }_{s}, \bar{u}_{s} + \zeta ) - f (s), \quad \forall f = b, {b}_{x}, {b}_{xx}, \sigma, {\sigma }_{x}, {\sigma }_{xx}; \\
& {f}^{t,u} := f \Big( \mathbb{E}_{t} \big[ ( {X}^{u}_{T} - \mathbb{E}_{t} [ {X}^{u}_{T} ] )^{2} \big], \ldots, 
                       \mathbb{E}_{t} \big[ ( {X}^{u}_{T} - \mathbb{E}_{t} [ {X}^{u}_{T} ] )^{2n} \big] \Big), \quad \forall f = \varphi, {\varphi }_{j}.
\end{align*} 
Then, for the following first-order and second-order variational equations,
\begin{equation}\label{eq: variational equation :eq}
\left\{ \begin{aligned}
d {y}^{ t, \varepsilon, \zeta }_{s} & = {b}_{x} (s) {y}^{ t, \varepsilon, \zeta }_{s} ds
                                      + \big( {\sigma }_{x} (s) {y}^{ t, \varepsilon, \zeta }_{s} + {1}_{\{ s \in [ t, t + \varepsilon ) \}} {\delta }^{\zeta } \sigma (s) \big) d {W}_{s}, \quad \forall s \in [ t,T ], \\
  {y}^{ t, \varepsilon, \zeta }_{t} & = 0; \\
d {z}^{ t, \varepsilon, \zeta }_{s} & = \Big( {b}_{x} (s) {z}^{ t, \varepsilon, \zeta }_{s}
                                      + \frac{1}{2} {b}_{xx} (s) ( {y}^{ t, \varepsilon, \zeta }_{s} )^{2}
                                      + {1}_{\{ s \in [ t, t + \varepsilon ) \}} {\delta }^{\zeta } b (s) \Big) ds \\
                                    & \quad + \Big( {\sigma }_{x} (s) {z}^{ t, \varepsilon, \zeta }_{s}
                                                  + \frac{1}{2} {\sigma }_{xx} (s) ( {y}^{ t, \varepsilon, \zeta }_{s} )^{2}
                                                  + {1}_{\{ s \in [ t, t + \varepsilon ) \}} {\delta }^{\zeta } {\sigma }_{x} (s) {y}^{ t, \varepsilon, \zeta }_{s} \Big) d {W}_{s}, \quad \forall s \in [ t,T ], \\
  {z}^{ t, \varepsilon, \zeta }_{t} & = 0,
\end{aligned} \right.
\end{equation}
one can conclude that for any $k \ge 1$, there exists a $( t, \zeta )$-independent constant $K \ge 0$ such that
\begin{equation*}
\left\{ \begin{aligned}
& \limsup_{ \varepsilon \downarrow 0 } \frac{1}{ {\varepsilon }^{k} } 
  \mathbb{E}_{t} \bigg[ \sup_{ s \in [ t,T ] } | {y}^{ t, \varepsilon, \zeta }_{s} |^{2k} \bigg] \le K | \zeta |^{2k}, \quad
  \limsup_{ \varepsilon \downarrow 0 } \frac{1}{ {\varepsilon }^{2k} } 
  \mathbb{E}_{t} \bigg[ \sup_{ s \in [ t,T ] } | {z}^{ t, \varepsilon, \zeta }_{s} |^{2k} \bigg] \le K | \zeta \vee 1 |^{4k}, \\
& \limsup_{ \varepsilon \downarrow 0 } \frac{1}{ {\varepsilon }^{2k} } 
  \mathbb{E}_{t} \bigg[ \sup_{ s \in [ t,T ] } | {X}^{ \bar{u}^{ t, \varepsilon, \zeta } }_{s} - {X}^{ \bar{u} }_{s} - {y}^{ t, \varepsilon, \zeta }_{s} - {z}^{ t, \varepsilon, \zeta }_{s} |^{2k} \bigg] = 0,
\end{aligned} \right.
\end{equation*}
$\mathbb{P}$-a.s. for all $t \in [ 0,T )$ and $\zeta \in \mathbb{L}_{ \mathcal{F}_{t} }^{2n} ( \Omega )$.
For brevity, using Landau symbols, one usually writes
\begin{equation}
\left\{ \begin{aligned}
& \mathbb{E}_{t} \bigg[ \sup_{ s \in [ t,T ] } | {y}^{ t, \varepsilon, \zeta }_{s} |^{2k} \bigg] = O ( {\varepsilon }^{k} ), \quad
  \mathbb{E}_{t} \bigg[ \sup_{ s \in [ t,T ] } | {z}^{ t, \varepsilon, \zeta }_{s} |^{2k} \bigg] = O ( {\varepsilon }^{2k} ), \\
& \mathbb{E}_{t} \bigg[ \sup_{ s \in [ t,T ] } | {X}^{ \bar{u}^{ t, \varepsilon, \zeta } }_{s} - {X}^{ \bar{u} }_{s} - {y}^{ t, \varepsilon, \zeta }_{s} - {z}^{ t, \varepsilon, \zeta }_{s} |^{2k} \bigg] = o ( {\varepsilon }^{2k} ).
\end{aligned} \right.
\label{eq:moment estimates :eq}
\end{equation}

In comparison with the classic moment estimate results reported in \cite[Theorem 3.4.4]{Yong-Zhou-1999}, \cref{eq:moment estimates :eq} is slightly stronger, 
because the supremum is taken in the conditional expectation and $K$ is a constant fixed a priori.
Interested readers can refer to \Cref{app: Standard moment estimate results} for a detailed justification.
Furthermore, owing to the inadequate integrability of ${X}^{u}$ (e.g., ${X}^{u}_{T}$ does not necessarily belong to $\mathbb{L}_{ \mathcal{F}_{T} }^{4n-2} ( \Omega )$),  
we must strengthen \cite[Lemma 4.1]{Wang-Liu-Bensoussan-Yiu-Wei-2025} with a weaker integrability condition. 
The main consequential result is also isolated in the following lemma.

\begin{lemma}\label{lem: actuarial estimate}
Suppose that $\xi \in \mathbb{L}_{ \mathcal{F}_{T} }^{q} ( \Omega )$ with some fixed $q > 1$ and that $\xi $ is independent of $( \varepsilon, \zeta )$. Then, 
\begin{equation*}
\sup_{ \tau \in [ t,T ] } | \mathbb{E}_{t} [ \xi {y}^{ t, \varepsilon, \zeta }_{\tau } ] |^{2} = o ( \varepsilon ).
\end{equation*}
Furthermore, for any given $\bar{u} \in \mathcal{U}$, $t \in [ 0,T )$, $\zeta \in \mathbb{L}^{2n}_{ \mathcal{F}_{t} } ( \Omega )$ and a sufficiently small $\varepsilon > 0$,
\begin{align*}
  {J}^{t} ( \bar{u}^{ t, \varepsilon, \zeta } )
& = {J}^{t} ( \bar{u} )
  + \mathbb{E}_{t} \big[ \big( 1 + \gamma ( {X}^{ \bar{u} }_{T} - \mathbb{E}_{t} [ {X}^{ \bar{u} }_{T} ] ) \big) ( {y}^{ t, \varepsilon, \zeta }_{T} + {z}^{ t, \varepsilon, \zeta }_{T} ) \big]
  - \frac{\gamma }{2} \mathbb{E}_{t} \big[ ( {y}^{ t, \varepsilon, \zeta }_{T} )^{2} \big] \\
& \quad - \sum_{j=2}^{2n} j \mathbb{E}_{t} \Big[ \Big( ( {X}^{ \bar{u} }_{T} - \mathbb{E}_{t} [ {X}^{ \bar{u} }_{T} ] )^{j-1} - \mathbb{E}_{t} \big[ ( {X}^{ \bar{u} }_{T} - \mathbb{E}_{t} [ {X}^{ \bar{u} }_{T} ] )^{j-1} \big] \Big) 
                                                 ( {y}^{ t, \varepsilon, \zeta }_{T} + {z}^{ t, \varepsilon, \zeta }_{T} ) \Big] {\varphi }_{j}^{ t, \bar{u} } \\
& \quad - \frac{1}{2} \sum_{j=2}^{2n} j (j-1) \mathbb{E}_{t} \Big[ ( {X}^{ \bar{u} }_{T} - \mathbb{E}_{t} [ {X}^{ \bar{u} }_{T} ] )^{j-2} ( {y}^{ t, \varepsilon, \zeta }_{T} )^{2} \Big] {\varphi }_{j}^{ t, \bar{u} } 
  + o ( \varepsilon ).
\end{align*}
\end{lemma}

\begin{proof}
See \Cref{pf-lem: actuarial estimate}.
\end{proof}

\subsection{BSDEs and diagonal processes}
\label{sec: BSDEs and diagonal processes}

For $u \in \mathcal{U}$, let us introduce a flow of linear BSDEs that are indexed by $t \in [ 0,T )$ as follows:
\begin{equation}\label{eq: BSDEs :eq}
\left\{ \begin{aligned}
d {Y}^{t,u}_{s} & = - \big( {b}_{x} (s) {Y}^{t,u}_{s} + {\sigma }_{x} (s) \mathcal{Y}^{t,u}_{s} \big) ds + \mathcal{Y}^{t,u}_{s} d {W}_{s}, ~ \forall s \in [ 0,T ], \\
  {Y}^{t,u}_{T} & = 1 - \gamma ( {X}^{u}_{T} - \mathbb{E}_{t} [ {X}^{u}_{T} ] ) \\
                & \quad - \sum_{j=2}^{2n} j \Big( ( {X}^{u}_{T} - \mathbb{E}_{t} [ {X}^{u}_{T} ] )^{j-1} - \mathbb{E}_{t} \big[ ( {X}^{u}_{T} - \mathbb{E}_{t} [ {X}^{u}_{T} ] )^{j-1} \big] \Big) {\varphi }_{j}^{t,u}, \\
d {Z}^{t,u}_{s} & = - \big( 2 {b}_{x} (s) {Z}^{t,u}_{s} + | {\sigma }_{x} (s) |^{2} {Z}^{t,u}_{s} + 2 {\sigma }_{x} (s) \mathcal{Z}^{t,u}_{s} \\
                & \qquad~ + {b}_{xx} (s) {Y}^{t,u}_{s} + {\sigma }_{xx} (s) \mathcal{Y}^{t,u}_{s} \big) ds + \mathcal{Z}^{t,u}_{s} d {W}_{s}, ~ \forall s \in [ 0,T ], \\
  {Z}^{t,u}_{T} & = - \gamma
                    - \sum_{j=2}^{2n} j (j-1) ( {X}^{u}_{T} - \mathbb{E}_{t} [ {X}^{u}_{T} ] )^{j-2} {\varphi }_{j}^{t,u}.
\end{aligned} \right. 
\end{equation}
Considering the well-posedness of \cref{eq: BSDEs :eq}, we employ the following integrability assumption, which automatically holds if $\varphi $ is Lipschitz continuous.

\begin{assumption}\label{ass: uniform integrabiity condition with continuity}
For the given $u \in \mathcal{U}$, there exists a sufficiently small $\delta \in ( 0, \frac{1}{2n-1} )$ such that
${\varphi }_{j}^{ \cdot, u } \in \mathbb{L}_{\mathbb{F},loc}^{ \frac{ 2 ( 1 + \delta ) }{ 1 - \delta } } ( 0,T; \mathbb{L}^{ \frac{ 2 n ( 1 + \delta ) }{ 2 n - ( ( j - 1 ) \vee n ) ( 1 + \delta ) } } ( \Omega ) )$ for all $j$.
\end{assumption}

Notably, $\frac{ 2 n ( 1 + \delta ) }{ 1 - ( 2n-1 ) \delta } \ge \frac{ 2 n ( 1 + \delta ) }{ 2 n - ( ( j - 1 ) \vee n ) ( 1 + \delta ) } \ge \frac{ 2 ( 1 + \delta ) }{ 1 - \delta } > 2$ holds for all $( j, \delta )$,
and readers could consider the simpler assumption that ${\varphi }_{j}^{ \cdot, u } \in \mathbb{L}_{\mathbb{F},loc}^{ 2 ( 1 + \rho ) } ( 0,T; \mathbb{L}^{ 2n ( 1 + \rho ) } ( \Omega ) )$ with a small $\rho > 0$ uniformly for all $j$,
which is slightly stronger than \Cref{ass: uniform integrabiity condition with continuity}.
In short, we indeed need ${Y}^{t,u}_{T}, {Z}^{t,u}_{T} \in \mathbb{L}_{\mathcal{F}_{T}}^{ 1 + \delta } ( \Omega )$ for a.e. $t \in [ 0,T )$ with a small $t$-independent $\delta > 0$,
which can be derived from $\int_{0}^{\tau } \mathbb{E} [ | {Y}^{t,u}_{T} |^{ 1 + \delta } + | {Z}^{t,u}_{T} |^{ 1 + \delta } ] dt < \infty $ for all $\tau \in [ 0,T )$.
For a.e. $t \in [ 0,T )$ and any $u \in \mathcal{U}$ satisfying \Cref{ass: uniform integrabiity condition with continuity},
\cref{eq: BSDEs :eq} permits the unique solution $( {Y}^{t,u}, \mathcal{Y}^{t,u}, {Z}^{t,u}, \mathcal{Z}^{t,u} ) \in
{C}_{\mathbb{F}} ( 0,T; \mathbb{L}^{ 1 + \delta } ( \Omega ) ) \times \mathbb{L}_{\mathbb{F}}^{2} ( 0,T; \mathbb{L}^{ 1 + \delta } ( \Omega ) ) \times 
{C}_{\mathbb{F}} ( 0,T; \mathbb{L}^{ 1 + \delta } ( \Omega ) ) \times \mathbb{L}_{\mathbb{F}}^{2} ( 0,T; \mathbb{L}^{ 1 + \delta } ( \Omega ) )$
according to \cite{Briand-Delyon-Hu-Pardoux-Stoica-2003} for the $\mathbb{L}^{p}$-solutions of general BSDEs.
Related applications and theoretical improvements of those BSDEs can be found in \cite{ElKaroui-Peng-Quenez-1997,Chen-2010}, etc.

However, the existence, uniqueness and integrability of the diagonal process $\{ \mathcal{Y}^{s,u}_{s} \}_{ s \in [ 0,T ) }$ should be carefully justified to prove the necessity of \Cref{thm: maximum principle}.
To mitigate misunderstandings, we should note that it is $t=s$, instead of $s=t$, that is plugged into $\mathcal{Y}^{t,u}_{s}$ when defining the diagonal process.
The feasibility of this definition is due to the fact that the variables $( t,s )$ can be separated in the expressions \cref{eq: polynomial expression of calY :eq} of $\mathcal{Y}^{t,u}_{s}$,
where $t$ is arbitrarily fixed in $[ 0,T )$ and the domain of $s$ is independent of $t$.
Then, we have the following \Cref{lem: integrabiity and uniqueness of diagonal process} to guarantee the uniqueness and integrability of the diagonal processes 
$\{ ( {Y}^{t,u}_{t}, \mathcal{Y}^{t,u}_{t}, {Z}^{t,u}_{t} ) \}_{ t \in [ 0,T ) }$. 

\begin{lemma}\label{lem: integrabiity and uniqueness of diagonal process}
Suppose that $u \in \mathcal{U}$ satisfies \Cref{ass: uniform integrabiity condition with continuity} and that $\rho \in ( 0, \delta )$ is an arbitrarily fixed constant.
Then, the flow of the BSDEs \cref{eq: BSDEs :eq} generates a diagonal process triplet:
\begin{equation*}
\{ ( {Y}^{t,u}_{t}, \mathcal{Y}^{t,u}_{t}, {Z}^{t,u}_{t} ) \}_{ t \in [ 0,T ) }
\in    \mathbb{L}_{\mathbb{F},loc}^{ \frac{ 2 ( 1 + \delta ) }{ 1 - \delta } } \big( 0,T; \mathbb{L}^{ 1 + \delta } ( \Omega ) \big) \times \mathbb{L}_{\mathbb{F},loc}^{ 1 + \delta } \big( 0,T; \mathbb{L}^{ 1 + \delta - \rho } ( \Omega ) \big)
\times \mathbb{L}_{\mathbb{F},loc}^{ 1 + \delta } \big( 0,T; \mathbb{L}^{ 1 + \delta - \rho } ( \Omega ) \big).
\end{equation*}
Moreover, if $\{ ( {Y}^{t,u}, \mathcal{Y}^{t,u}, {Z}^{t,u}, \mathcal{Z}^{t,u} ) \}_{ t \in [ 0,T ) }$
and $\{ ( \hat{Y}^{t,u}, \hat{\mathcal{Y}}^{t,u}, \hat{Z}^{t,u}, \hat{\mathcal{Z}}^{t,u} ) \}_{ t \in [ 0,T ) }$ fulfill \cref{eq: BSDEs :eq}, then
\begin{equation*}
( {Y}^{t,u}_{t}, \mathcal{Y}^{t,u}_{t}, {Z}^{t,u}_{t} ) = ( \hat{Y}^{t,u}_{t}, \hat{\mathcal{Y}}^{t,u}_{t}, \hat{Z}^{t,u}_{t} ),
\quad \mathbb{P}\text{-}a.s., ~ a.e. ~ t \in [ 0,T ).
\end{equation*}
\end{lemma}

\begin{proof}
See \Cref{pf-lem: integrabiity and uniqueness of diagonal process}.
\end{proof}

\section{Sufficient and necessary condition for ONECs}
\label{sec: Sufficient and necessary condition for ONECs}

Following the duality analysis as described in \cite[Section 3.4.3, pp. 134-137]{Yong-Zhou-1999}, namely,
plugging the BSDEs \cref{eq: BSDEs :eq} back into the asymptotic expansion of ${J}^{t} ( \bar{u}^{ t, \varepsilon, \zeta } )$ given by \Cref{lem: actuarial estimate},
in conjunction with using It\^o's rule and the moment estimates in \cref{eq:moment estimates :eq}, we obtain
\begin{align}\label{eq: spike perturbation :eq}
    {J}^{t} ( \bar{u}^{ t, \varepsilon, \zeta } ) - {J}^{t} ( \bar{u} )
& = \mathbb{E}_{t} \Big[ {Y}^{ t, \bar{u} }_{T} ( {y}^{ t, \varepsilon, \zeta }_{T} + {z}^{ t, \varepsilon, \zeta }_{T} ) + \frac{1}{2} {Z}^{ t, \bar{u} }_{T} ( {y}^{ t, \varepsilon, \zeta }_{T} )^{2} \Big] + o ( \varepsilon ) \\
& = \mathbb{E}_{t} \bigg[ \int_{t}^{ t + \varepsilon } \Big( {Y}^{ t, \bar{u} }_{s} {\delta }^{\zeta } b (s) + \mathcal{Y}^{ t, \bar{u} }_{s} {\delta }^{\zeta } \sigma (s) 
                                                           + \frac{1}{2} {Z}^{ t, \bar{u} }_{s} | {\delta }^{\zeta } \sigma (s) |^{2} \Big) ds \bigg]
  + o ( \varepsilon ), \quad \forall \bar{u} \in \mathcal{U}. \notag
\end{align}
In general, since $( {Y}^{ t, \bar{u} }, \mathcal{Y}^{ t, \bar{u} }, {Z}^{ t, \bar{u} } )$ depends on $t$ and $\mathcal{Y}^{ t, \bar{u} }$ is not necessarily continuous,
one cannot directly conclude that $\bar{u}$ is an ONEC if or only if 
${Y}^{ t, \bar{u} }_{t} {\delta }^{\zeta } b (t) + \mathcal{Y}^{ t, \bar{u} }_{t} {\delta }^{\zeta } \sigma (t) + \frac{1}{2} {Z}^{ t, \bar{u} }_{t} | {\delta }^{\zeta } \sigma (t) |^{2} \le 0$ $\mathbb{P}$-a.s. 
for a.e. $t \in [ 0,T )$ and any constant $\zeta \in \mathbb{R}$.
\cite{Wang-Liu-Bensoussan-Yiu-Wei-2025} merely arrives at the sufficiency of the following equilibrium condition for an ONEC $\bar{u} \in \mathcal{U}$:
\begin{equation*}
\limsup_{ s \downarrow t } \mathbb{E}_{t} \Big[ {Y}^{ t, \bar{u} }_{s} {\delta }^{\zeta } b (s) + \mathcal{Y}^{ t, \bar{u} }_{s} {\delta }^{\zeta } \sigma (s) + \frac{1}{2} {Z}^{ t, \bar{u} }_{s} | {\delta }^{\zeta } \sigma (s) |^{2} \Big] \le 0,
\quad \forall \zeta \in \mathbb{L}_{ \mathcal{F}_{t} }^{2n} ( \Omega ), ~ \mathbb{P}\text{-}a.s., ~ a.e. ~ t \in [ 0,T ).
\end{equation*}
The presence of limits renders it of little practical value.
Fortunately, in the present paper, 
we can use the expressions for $( {Y}^{ t, \bar{u} }_{s}, \mathcal{Y}^{ t, \bar{u} }_{s}, {Z}^{ t, \bar{u} }_{s} )$ derived in the proof of \Cref{lem: integrabiity and uniqueness of diagonal process} 
(see \cref{eq: polynomial expression of Y :eq}, \cref{eq: polynomial expression of calY :eq} and \cref{eq: polynomial expression of Z :eq}) to simplify the equilibrium condition
and obtain both the sufficiency and necessity for an ONEC $\bar{u} \in \mathcal{U}$.
Notably, \cite[Proposition 3.3]{Hu-Jin-Zhou-2017} no longer holds for our problem due to the presence of the non-separable terms $( {X}^{u}_{T} )^{j-1} {\varphi }^{t,u}_{j}$ in the terminal condition of ${Y}^{t,u}$.

\begin{theorem}\label{thm: maximum principle}
Suppose that $\bar{u} \in \mathcal{U}$ satisfies \Cref{ass: uniform integrabiity condition with continuity}.
Then, 
\begin{align}\label{eq: spike perturbation modified :eq}
  {J}^{t} ( \bar{u}^{ t, \varepsilon, \zeta } ) - {J}^{t} ( \bar{u} )
= \mathbb{E}_{t} \bigg[ \int_{t}^{ t + \varepsilon } \Big( {Y}^{ s, \bar{u} }_{s} {\delta }^{\zeta } b (s) + \mathcal{Y}^{ s, \bar{u} }_{s} {\delta }^{\zeta } \sigma (s) 
                                                         + \frac{1}{2} {Z}^{ s, \bar{u} }_{s} | {\delta }^{\zeta } \sigma (s) |^{2} \Big) ds \bigg]
+ o ( \varepsilon ), \\
\forall \zeta \in \mathbb{L}_{ \mathcal{F}_{t} }^{2n} ( \Omega ), ~ \mathbb{P}\text{-}a.s., ~ a.e. ~ t \in [ 0,T ). \notag
\end{align}
Furthermore, $\bar{u}$ is an ONEC if and only if
\begin{equation}\label{eq: equilibrium condition :eq}
{Y}^{ t, \bar{u} }_{t} {\delta }^{\zeta } b (t) + \mathcal{Y}^{ t, \bar{u} }_{t} {\delta }^{\zeta } \sigma (t) + \frac{1}{2} {Z}^{ t, \bar{u} }_{t} | {\delta }^{\zeta } \sigma (t) |^{2} \le 0, 
\quad \forall \zeta \in \mathbb{R}, ~ \mathbb{P}\text{-}a.s., ~ a.e. ~ t \in [ 0,T ),
\end{equation}
i.e., the following maximum principle holds:
\begin{equation*}
0 = \max_{ \zeta \in \mathbb{R} } \Big\{ {Y}^{ t, \bar{u} }_{t} {\delta }^{\zeta } b (t) + \mathcal{Y}^{ t, \bar{u} }_{t} {\delta }^{\zeta } \sigma (t) + \frac{1}{2} {Z}^{ t, \bar{u} }_{t} | {\delta }^{\zeta } \sigma (t) |^{2} \Big\}, 
\quad \mathbb{P}\text{-}a.s., ~ a.e. ~ t \in [ 0,T ).
\end{equation*}
\end{theorem}

\begin{proof}
See \Cref{pf-thm: maximum principle}.
\end{proof}

\begin{remark}\label{rem: equilibrium condition reduced}
If there exists a $\zeta $-independent process $R$ such that ${\delta }^{\zeta } b (t) = {R}_{t} {\delta }^{\zeta } \sigma (t)$ $\mathbb{P}$-a.s. for a.e. $t \in [ 0,T )$
and the range $\{ {\delta }^{\zeta } \sigma (t): \zeta \in \mathbb{R} \}$ covers some open intervals containing the origin,
then \cref{eq: equilibrium condition :eq} as a quadratic inequality of ${\delta }^{\zeta } \sigma (t)$ reduces to
\begin{equation}\label{eq: equilibrium condition reduced :eq}
{Y}^{ t, \bar{u} }_{t} {R}_{t} + \mathcal{Y}^{ t, \bar{u} }_{t} = 0, \quad {Z}^{ t, \bar{u} }_{t} \le 0, 
\quad \mathbb{P}\text{-}a.s., ~ a.e. ~ t \in [ 0,T ),
\end{equation}
\end{remark}

Even if the controlled SDE is linear and all partial derivatives ${\varphi }_{j}$ are constant,
it remains challenging to verify the hypothesis that ${Y}^{ t, \bar{u} }_{s}$ is a polynomial of 
$( {X}^{\bar{u}}_{t}, {X}^{\bar{u}}_{s}, \mathbb{E}_{t} [ {X}^{\bar{u}}_{s} ], \mathbb{E}^{(1)}_{t} [ {X}^{\bar{u}}_{s} ], \ldots, \mathbb{E}^{(1)}_{t} [ ( {X}^{\bar{u}}_{s} )^{2n-1} ] )$ in the spirit of \cite[Section 4]{Hu-Jin-Zhou-2012}.
This is due to the presence of the polynomial of $\mathbb{E}_{t} [ {X}^{\bar{u}}_{T} ]$ and the non-separable ${X}^{\bar{u}}_{T} \mathbb{E}_{t} [ {X}^{\bar{u}}_{T} ]$ in the expression of ${Y}^{ t, \bar{u} }_{T}$.
As a consequence, the method in \cite[Section 4]{Hu-Jin-Zhou-2012} for proving the uniqueness of an ONEC might not be applicable to our problem in general.

\section{Applications to portfolio selection}
\label{sec: Closed-form solution}

In this section, we consider the following linear controlled SDE:
\begin{equation}\label{eq: linear controlled SDE :eq}
d {X}^{u}_{t} = ( {A}_{t} {X}^{u}_{t} + {B}_{t} {u}_{t} + {C}_{t} ) dt + ( {I}_{t} {X}^{u}_{t} + {D}_{t} {u}_{t} + {F}_{t} ) d {W}_{t}, \quad \forall t \in [ 0,T ], \quad
  {X}_{0} = {x}_{0},
\end{equation}
where $( A, B, C, I, D, F )$ are essentially bounded $\mathbb{F}$-predictable processes with $\esssup_{ [ 0,T ] \times \Omega } | {D}^{-1} | < \infty $.
This type of equation frequently appears and has been widely studied in financial applications, including mean-variance portfolio selection scenarios.
Readers can find economic insights into the parameters in, e.g., \cite[Remark 5.1]{Wang-Liu-Bensoussan-Yiu-Wei-2025}, when the white noise is changed to an affine function of ${X}^{u}_{t}$.
More precisely, for a typical financial market (see, e.g., \cite[Section 1.1]{ElKaroui-Peng-Quenez-1997}),
\cref{eq: linear controlled SDE :eq} formulates the dynamics of the wealth process ${X}^{u}$ corresponding to the investment amount process $u$
with the interest rate process $A$, the risk premium process ${D}^{-1}B$ and the volatility process $D$.
In addition, ${C}_{t}$ represents a given instantaneous income or consumption rate, while $( {I}_{t} {X}^{u}_{t} + {F}_{t} ) d {W}_{t}$ is the white noise of the infinitesimal increment $d {X}^{u}_{t}$.

For ease of reference, we list some concerned portfolio selection problems with the first few higher-order moments, which will be solved for an equilibrium strategy in this section.
\begin{itemize}
\item Mean-variance-skewness problem (see \Cref{ex: method of undetermined coefficients,ex: method of polynomial algebraic equation,ex: BSDEs characterization}) with
      \begin{equation*}
      {J}^{t} (u) = \mathbb{E}_{t} [ {X}^{u}_{T} ] 
                  - \frac{\gamma }{2} \mathbb{E}_{t} \big[ ( {X}^{u}_{T} - \mathbb{E}_{t} [ {X}^{u}_{T} ] )^{2} \big]
                  - {\varphi }_{3} \mathbb{E}_{t} \big[ ( {X}^{u}_{T} - \mathbb{E}_{t} [ {X}^{u}_{T} ] )^{3} \big].
      \end{equation*}
\item Mean-variance-skewness-kurtosis problem (see \Cref{ex: method of undetermined coefficients}) with
      \begin{equation*}
      {J}^{t} (u) = \mathbb{E}_{t} [ {X}^{u}_{T} ] 
                  - \frac{\gamma }{2} \mathbb{E}_{t} \big[ ( {X}^{u}_{T} - \mathbb{E}_{t} [ {X}^{u}_{T} ] )^{2} \big]
                  - {\varphi }_{3} \mathbb{E}_{t} \big[ ( {X}^{u}_{T} - \mathbb{E}_{t} [ {X}^{u}_{T} ] )^{3} \big]
                  - {\varphi }_{4} \mathbb{E}_{t} \big[ ( {X}^{u}_{T} - \mathbb{E}_{t} [ {X}^{u}_{T} ] )^{4} \big].
      \end{equation*}
\item (Nonlinear) mean-variance-excess kurtosis problem (see \Cref{ex: method of polynomial algebraic equation}) with
      \begin{equation*}
      {J}^{t} (u) = \mathbb{E}_{t} [ {X}^{u}_{T} ] - \frac{\gamma }{2} \mathbb{E}_{t} \big[ ( {X}^{u}_{T} - \mathbb{E}_{t} [ {X}^{u}_{T} ] )^{2} \big]
                  - \frac{K}{4} \big| \mathbb{E}_{t} \big[ ( {X}^{u}_{T} - \mathbb{E}_{t} [ {X}^{u}_{T} ] )^{2} \big] \big|^{2} 
                    \bigg( \frac{ \mathbb{E}_{t} \big[ ( {X}^{u}_{T} - \mathbb{E}_{t} [ {X}^{u}_{T} ] )^{4} \big] }{ | \mathbb{E}_{t} [ ( {X}^{u}_{T} - \mathbb{E}_{t} [ {X}^{u}_{T} ] )^{2} ] |^{2} } - 3 \bigg).
      \end{equation*}  
\item Mean-variance-standardized moment problem (see \Cref{ex: MV solution is ONEC}), including the mean-variance-skewness-kurtosis problem with
      \begin{equation*}
      {J}^{t} (u) = \mathbb{E}_{t} [ {X}^{u}_{T} ] - \frac{\gamma }{2} \mathbb{E}_{t} \big[ ( {X}^{u}_{T} - \mathbb{E}_{t} [ {X}^{u}_{T} ] )^{2} \big]
                  - {\varphi }_{3} \frac{ \mathbb{E}_{t} \big[ ( {X}^{u}_{T} - \mathbb{E}_{t} [ {X}^{u}_{T} ] )^{3} \big] }{ | \mathbb{E}_{t} [ ( {X}^{u}_{T} - \mathbb{E}_{t} [ {X}^{u}_{T} ] )^{2} ] |^{\frac{3}{2}} }
                  - {\varphi }_{4} \frac{ \mathbb{E}_{t} \big[ ( {X}^{u}_{T} - \mathbb{E}_{t} [ {X}^{u}_{T} ] )^{4} \big] }{ | \mathbb{E}_{t} [ ( {X}^{u}_{T} - \mathbb{E}_{t} [ {X}^{u}_{T} ] )^{2} ] |^{2} }.
      \end{equation*} 
\item Mean-variance-standard deviation problem (see \Cref{ex: BSDEs characterization} for random parameters) with
      \begin{equation*}
      {J}^{t} (u) = \mathbb{E}_{t} [ {X}^{u}_{T} ] - \frac{\gamma }{2} \mathbb{E}_{t} \big[ ( {X}^{u}_{T} - \mathbb{E}_{t} [ {X}^{u}_{T} ] )^{2} \big] 
                  - \big( \mathbb{E}_{t} \big[ ( {X}^{u}_{T} - \mathbb{E}_{t} [ {X}^{u}_{T} ] )^{2} \big] \big)^{ \frac{1}{2} }.
      \end{equation*}
\end{itemize}

In fact, under the wealth dynamics \cref{eq: linear controlled SDE :eq}, it suffices to determine the terminal value ${X}^{ \bar{u} }_{T}$ to find the ONEC $\bar{u}$. 
Thus, given ${X}^{ \bar{u} }_{T} = \xi \in \mathbb{L}_{ \mathcal{F}_{T} }^{2n} ( \Omega )$, 
one can derive $\bar{u}$ as the hedging strategy by solving the following BSDE:
\begin{equation}\label{eq: BSDE ending at xi :eq}
- d {X}_{t} = - \bigg( \Big( {A}_{t} - \frac{ {B}_{t} }{ {D}_{t} } {I}_{t} \Big) {X}_{t} + \frac{ {B}_{t} }{ {D}_{t} } \mathcal{X}_{t} + \Big( {C}_{t} - \frac{ {B}_{t} }{ {D}_{t} } {F}_{t} \Big) \bigg) dt - \mathcal{X}_{t} d {W}_{t},
\quad s.t. \quad {X}_{T} = \xi,
\end{equation}
and then arrive at $\bar{u}_{t} = \frac{1}{ {D}_{t} } ( \mathcal{X}_{t} - {I}_{t} {X}_{t} - {F}_{t} )$.
By using the martingale representation theorem, with a slight abuse of notation, 
our control problem is reduced to finding an appropriate $\mathfrak{X} \in \mathbb{L}_{\mathbb{F}}^{2} ( 0,T; \mathbb{L}^{2n} ( \Omega ) )$ for $\xi = \mathbb{E} [ \xi ] + \int_{0}^{T} \mathfrak{X}_{s} d {W}_{s}$ such that
\begin{equation}\label{eq: BSDE in linear case :eq}
\left\{ \begin{aligned}
d {Y}^{t}_{s} & = - ( {A}_{s} {Y}^{t}_{s} + {I}_{s} \mathcal{Y}^{t}_{s} ) ds + \mathcal{Y}^{t}_{s} d {W}_{s}, ~ \forall s \in [ 0,T ], \\
  {Y}^{t}_{T} & = 1 - \gamma ( {\chi }_{T} - {\chi }_{t} )
                    - \sum_{j=2}^{2n} j \big( ( {\chi }_{T} - {\chi }_{t} )^{j-1} - \mathbb{E}_{t} [ ( {\chi }_{T} - {\chi }_{t} )^{j-1} ] \big) {\varphi }_{j}^{t}; \\
  {\chi }_{s} & = \int_{0}^{s} \mathfrak{X}_{v} d {W}_{v}, \quad {Y}^{t}_{t} {B}_{t} + \mathcal{Y}^{t}_{t} {D}_{t} = 0, 
\end{aligned} \right. 
\end{equation}
where ${\varphi }_{j}^{t} := {\varphi }_{j} ( \mathbb{E}_{t} [ ( {\chi }_{T} - {\chi }_{t} )^{2} ], \ldots, \mathbb{E}_{t} [ ( {\chi }_{T} - {\chi }_{t} )^{2n} ] )$, provided that ${Z}^{ t, \bar{u} }_{t} \le 0$.
Hereafter, for the sake of brevity, we omit the statements ``a.e. $t \in [ 0,T ]$'' and ``$\mathbb{P}$-a.s.'' unless otherwise mentioned.
In addition, we should emphasize that the ONEC $\bar{u}$ is uniquely determined given $\mathfrak{X}$, and vice versa.

\begin{remark}\label{rem: risk neutral measure}
The terminal value class $\{ \xi + K \}_{ K \in \mathbb{R} }$ corresponds to the same $\mathfrak{X}$, 
so one needs to determine the ``arbitrary constant'' $K$ to fit the initial condition ${X}^{ \bar{u} }_{0} = {x}_{0}$ given $\mathfrak{X}$, referring to the risk-neutral pricing technic.
That is, for $\xi = K + \int_{0}^{T} \mathfrak{X}_{s} d {W}_{s}$, one can establish the following linear equation to derive $K$:
\begin{equation*}
{x}_{0} = \tilde{\mathbb{E}} \bigg[ {e}^{ - \int_{0}^{T} ( {A}_{v} - \frac{ {B}_{v} }{ {D}_{v} } {I}_{v} ) dv } \bigg( K + \int_{0}^{T} \mathfrak{X}_{s} d {W}_{s} \bigg)
                                  - \int_{0}^{T} {e}^{ - \int_{0}^{s} ( {A}_{v} - \frac{ {B}_{v} }{ {D}_{v} } {I}_{v} ) dv } \Big( {C}_{s} - \frac{ {B}_{s} }{ {D}_{s} } {F}_{s} \Big) ds \bigg],
\end{equation*}
where $\tilde{\mathbb{E}}$ is the expectation operator under the equivalent martingale measure $\tilde{\mathbb{P}}$ given by
\begin{equation*}
\frac{ d \tilde{\mathbb{P}} }{ d \mathbb{P} } \Big|_{ \mathcal{F}_{T} } = {e}^{ - \int_{0}^{T} \frac{ {B}_{v} }{ {D}_{v} } d {W}_{v} - \frac{1}{2} \int_{0}^{T} | \frac{ {B}_{v} }{ {D}_{v} } |^{2} dv }.
\end{equation*}
Then, one can derive $( \mathcal{X}, \bar{u} )$ by differentiating the both sides of 
\begin{equation}\label{eq: conditional expectation expression of X :eq}
{X}_{t} = \tilde{\mathbb{E}}_{t} \bigg[ {e}^{ - \int_{t}^{T} ( {A}_{v} - \frac{ {B}_{v} }{ {D}_{v} } {I}_{v} ) dv } \bigg( K + \int_{0}^{T} \mathfrak{X}_{s} d {W}_{s} \bigg)
                                      - \int_{t}^{T} {e}^{ - \int_{t}^{s} ( {A}_{v} - \frac{ {B}_{v} }{ {D}_{v} } {I}_{v} ) dv } \Big( {C}_{s} - \frac{ {B}_{s} }{ {D}_{s} } {F}_{s} \Big) ds \bigg].
\end{equation}
\end{remark}

\subsection{Closed-form solution for some cases with deterministic parameters}

Throughout this subsection, we assume that $( A,B,C,I,D,F )$ are deterministic.
It follows from \cref{eq: BSDE ending at xi :eq} and \cref{eq: conditional expectation expression of X :eq} that
\begin{equation*}
d \tilde{\mathbb{E}}_{t} \bigg[ \int_{0}^{T} \mathfrak{X}_{s} \frac{ {B}_{s} }{ {D}_{s} } ds \bigg] 
= \Big( \mathfrak{X}_{t} - {e}^{ \int_{t}^{T} ( {A}_{v} - \frac{ {B}_{v} }{ {D}_{v} } {I}_{v} ) dv } \big( {I}_{t} {X}^{\bar{u}}_{t} + {D}_{t} \bar{u}_{t} + {F}_{t} \big) \Big) \Big( d {W}_{t} + \frac{ {B}_{t} }{ {D}_{t} } dt \Big).
\end{equation*}
Clearly, a deterministic function $\mathfrak{X}$ could make the left-hand side of this SDE vanish, which leads to linear feedback control:
\begin{equation*}
\bar{u}_{t} = \frac{ {e}^{ - \int_{t}^{T} ( {A}_{v} - \frac{ {B}_{v} }{ {D}_{v} } {I}_{v} ) dv } \mathfrak{X}_{t} - {I}_{t} {X}^{\bar{u}}_{t} - {F}_{t} }{ {D}_{t} }.
\end{equation*}
Proceeding with the hypothesis $\mathfrak{X} \in \mathbb{L}^{2} ( 0,T ) \cap \mathbb{L}^{\infty }_{loc} ( 0,T )$, 
which means that ${\chi }_{T} - {\chi }_{s}$ has a normal distribution with a mean of $0$ (resp. $\int_{s}^{T} \mathfrak{X}_{v} {I}_{v} dv$) and a variance of $\int_{s}^{T} | \mathfrak{X}_{v} |^{2} dv$ 
conditioned on $\mathcal{F}_{s}$ under $\mathbb{P}$ (resp. $\mathbb{P}^{(1)}$ given by 
$\frac{ d \mathbb{P}^{(1)} }{ d \mathbb{P} } |_{ \mathcal{F}_{T} } = {e}^{ \int_{0}^{T} {I}_{v} d {W}_{v} - \frac{1}{2} \int_{0}^{T} | {I}_{v} |^{2} dv }$ as in \Cref{pf-lem: integrabiity and uniqueness of diagonal process}),
from \cref{eq: BSDEs solution :eq} we have 
\begin{align*}
  {e}^{ - \int_{s}^{T} {A}_{v} dv } {Y}^{t}_{s} 
& = 1 - \gamma \mathbb{E}^{(1)}_{s} [ {\chi }_{T} - {\chi }_{s} ] - \gamma ( {\chi }_{s} - {\chi }_{t} ) \\
& \quad - \sum_{j=2}^{2n} j \sum_{k=0}^{j-1} \binom{j-1}{k} \mathbb{E}^{(1)}_{s} [ ( {\chi }_{T} - {\chi }_{s} )^{j-1-k} ] ( {\chi }_{s} - {\chi }_{t} )^{k}  
                          {\varphi }_{j} \bigg( \vec{\alpha } \Big( \int_{t}^{T} | \mathfrak{X}_{v} |^{2} dv \Big) \bigg) \notag \\
& \quad + \sum_{j=2}^{2n} j \mathbb{E}_{t} [ ( {\chi }_{T} - {\chi }_{s} )^{j-1} ] {\varphi }_{j} \bigg( \vec{\alpha } \Big( \int_{t}^{T} | \mathfrak{X}_{v} |^{2} dv \Big) \bigg), \\
  {e}^{ - \int_{s}^{T} {A}_{v} dv } \mathcal{Y}^{t}_{s} 
& = - \gamma \mathfrak{X}_{s} - \mathfrak{X}_{s} \sum_{j=2}^{2n} j (j-1) \mathbb{E}^{(1)}_{s} [ ( {\chi }_{T} - {\chi }_{t} )^{j-2} ] {\varphi }_{j} \bigg( \vec{\alpha } \Big( \int_{t}^{T} | \mathfrak{X}_{v} |^{2} dv \Big) \bigg).
\end{align*}
Here $\mathbb{E}^{(1)}_{s} [ \cdot ] = \mathbb{E}^{(1)} [ \cdot | \mathcal{F}_{s} ]$ is the conditional expectation under $\mathbb{P}^{(1)}$.
Additionally, the abovementioned normal distribution of ${\chi }_{T} - {\chi }_{s}$ makes \Cref{ass: uniform integrabiity condition with continuity} hold.
It then follows from the equilibrium condition $0 = {Y}^{t}_{t} {B}_{t} + \mathcal{Y}^{t}_{t} {D}_{t}$ that
\begin{align}\label{eq: integral equation :eq}
& \frac{ {B}_{t} }{ {D}_{t} } \bigg( 1 - \gamma \int_{t}^{T} \mathfrak{X}_{v} {I}_{v} dv
                                       - \sum_{j=2}^{2n} j \sum_{k=1}^{j-1} \binom{j-1}{k} \Big( \int_{t}^{T} \mathfrak{X}_{v} {I}_{v} dv \Big)^{k} 
                                                                            {\alpha }_{j-1-k} {\varphi }_{j} ( \vec{\alpha } ) \bigg) \\
& = \mathfrak{X}_{t} \bigg( \gamma + \sum_{j=2}^{2n} j (j-1) \sum_{k=0}^{j-2} \binom{j-2}{k} \Big( \int_{t}^{T} \mathfrak{X}_{v} {I}_{v} dv \Big)^{k} 
                                                                              {\alpha }_{j-2-k} {\varphi }_{j} ( \vec{\alpha } ) \bigg), \notag
\end{align}
where $\vec{\alpha } = \vec{\alpha } ( \int_{t}^{T} | \mathfrak{X}_{v} |^{2} dv )$.
Therefore, the desired closed-form expression of $\mathfrak{X}_{t}$ can be derived by solving this integral equation.
Of course, we need to check whether ${Z}^{ t, \bar{u} }_{t} \le 0$ holds; i.e.,
\begin{equation}\label{eq: integral inequality :eq}
0 \le \gamma + \sum_{j=2}^{2n} j (j-1) \sum_{k=0}^{j-2} \binom{j-2}{k} \Big( 2 \int_{t}^{T} \mathfrak{X}_{v} {I}_{v} dv \Big)^{k} {\alpha }_{j-2-k} {\varphi }_{j} ( \vec{\alpha } ).
\end{equation}

\subsubsection{Application I: linear objective function}

Following the spirit of \cite[Section 4]{Hu-Jin-Zhou-2012}, we have the following theorem for the semi-closed expression of $( \mathfrak{X}, {Y}^{t}, \mathcal{Y}^{t} )$ given the continuous parameters and linear function $\varphi $.
Notably, the continuity assumption only serves to rigorously establish ordinary differential equations (ODEs) and does not affect the validity of the derived expressions.

\begin{theorem}\label{thm: method of undetermined coefficients}
Suppose that the parameters $( A,B,I,D )$ are continuous and that $\varphi $ is a linear function (i.e., ${\varphi }_{j}^{t} \equiv {\varphi }_{j} \in \mathbb{R}$ for all $t \in [ 0,T ]$).
If the following dynamic system for $( \Phi, {\phi }_{j}, {\psi }_{j} )$ permits a solution with ${\phi }_{2,s} \ne 0$ for all $s \in [ 0,T ]$, i.e.,
\begin{equation}\label{eq: ODEs :eq}
\left\{ \begin{aligned}
& 0 = \Big( {A}_{s} - \frac{ {B}_{s} }{ {D}_{s} } {I}_{s} \Big) {\Phi }_{s} + {\Phi }_{s}' - \Big| \frac{ {B}_{s} {\Phi }_{s} }{ {D}_{s} {\phi }_{2,s} } \Big|^{2} ( {\phi }_{3,s} - {\psi }_{3,s} ) {1}_{\{ n \ge 2 \}}, \\
& 0 = {A}_{s} {\phi }_{j,s} + {\phi }_{j,s}' 
    + j \frac{ {B}_{s} }{ {D}_{s} } {I}_{s} {\Phi }_{s} \frac{ {\phi }_{j+1,s} }{ {\phi }_{2,s} } {1}_{\{ j \le 2n-1 \}} 
    + \binom{j+1}{2} \Big| \frac{ {B}_{s} {\Phi }_{s} }{ {D}_{s} {\phi }_{2,s} } \Big|^{2} {\phi }_{j+2,s} {1}_{\{ j \le 2n-2 \}}, \\
& 0 = {A}_{s} {\psi }_{j,s} + {\psi }_{j,s}' + \binom{j+1}{2} \Big| \frac{ {B}_{s} {\Phi }_{s} }{ {D}_{s} {\phi }_{2,s} } \Big|^{2} {\psi }_{j+2,s} {1}_{\{ j \le 2n-2 \}}, \\
& {\Phi }_{T} = 1, \quad {\phi }_{j,T} = j {\varphi }_{j} + \gamma {1}_{\{ j = 2 \}}, \quad {\psi }_{j,T} = j {\varphi }_{j},
\end{aligned} \right.
\end{equation}
then, the control problem formulated by \cref{eq: BSDE in linear case :eq} yields the following solution:
\begin{equation*}
\left\{ \begin{aligned}
& {Y}^{t}_{s} = {\Phi }_{s} - \sum_{j=2}^{2n} \big( {\phi }_{j,s} ( {\chi }_{s} - {\chi }_{t} )^{j-1} - {\psi }_{j,s} \mathbb{E}_{t} [ ( {\chi }_{s} - {\chi }_{t} )^{j-1} ] \big), \\
& \mathcal{Y}^{t}_{s} = - \mathfrak{X}_{s} \sum_{j=2}^{2n} (j-1) {\phi }_{j,s} ( {\chi }_{s} - {\chi }_{t} )^{j-2}, \\
& \mathfrak{X}_{s} = \frac{ {B}_{s} {\Phi }_{s} }{ {D}_{s} {\phi }_{2,s} }.
\end{aligned} \right.
\end{equation*}
\end{theorem}

\begin{proof}
See \Cref{pf-thm: method of undetermined coefficients}.
\end{proof}

\begin{example}\label{ex: method of undetermined coefficients}
If $n=1$, namely, for a mean-variance problem, it is easy to derive the corresponding closed-form solution:
\begin{equation*}
{\Phi }_{t} = {e}^{ \int_{t}^{T} ( {A}_{v} - \frac{ {B}_{v} }{ {D}_{v} } {I}_{v} ) dv }, \quad
{\phi }_{2,t} = ( \gamma + 2 {\varphi }_{2} ) {e}^{ \int_{t}^{T} {A}_{v} dv }, \quad
{\psi }_{2,t} = 2 {\varphi }_{2} {e}^{ \int_{t}^{T} {A}_{v} dv },
\end{equation*}
implying that
\begin{equation*}
\mathfrak{X}_{t} = \frac{ {B}_{s} {\Phi }_{s} }{ {D}_{s} {\phi }_{2,s} } = \frac{ {B}_{t} }{ ( \gamma + 2 {\varphi }_{2} ) {D}_{t} } {e}^{ - \int_{t}^{T} \frac{ {B}_{v} }{ {D}_{v} } {I}_{v} dv, }
\end{equation*}
and the linear feedback control is
\begin{equation*}
\bar{u}_{t} = \frac{ {B}_{t} }{ ( \gamma + 2 {\varphi }_{2} ) | {D}_{t} |^{2} } {e}^{ - \int_{t}^{T} {A}_{v} dv } - \frac{ {I}_{t} }{ {D}_{t} } {X}^{\bar{u}}_{t} - \frac{ {F}_{t} }{ {D}_{t} }.
\end{equation*}
Otherwise, it becomes increasingly difficult to derive explicit expressions for $\Phi, {\phi }_{j}, {\psi }_{j}$ as $n$ increases.
For example, when $n = 2$ and ${\varphi }_{4} = 0$, namely, for a mean-variance-skewness problem, 
one obtains the following results in sequence:
\begin{align*}
& {\phi }_{4}, {\psi }_{4} \equiv 0, \quad 
  {\phi }_{3,t} = {\psi }_{3,t} = 3 {\varphi }_{3} {e}^{ \int_{t}^{T} {A}_{v} dv }, \quad 
  {\psi }_{2,t} = 2 {\varphi }_{2} {e}^{ \int_{t}^{T} {A}_{v} dv }, \quad 
  {\Phi }_{t} = {e}^{ \int_{t}^{T} ( {A}_{v} - \frac{ {B}_{v} }{ {D}_{v} } {I}_{v} ) dv } \\
& | {\phi }_{2,t} |^{2} = {e}^{ 2 \int_{t}^{T} {A}_{v} dv } \bigg( ( \gamma + 2 {\varphi }_{2} )^{2} + 12 {\varphi }_{3} \Big( 1 - {e}^{ - \int_{t}^{T} \frac{ {B}_{v} }{ {D}_{v} } {I}_{v} dv } \Big) \bigg).
\end{align*}
Provided that $\gamma + 2 {\varphi }_{2} \ne 0$ and $( \gamma + 2 {\varphi }_{2} )^{2} + 12 {\varphi }_{3} ( 1 - {e}^{ - \int_{t}^{T} \frac{ {B}_{v} }{ {D}_{v} } {I}_{v} dv } ) > 0$ for all $t \in [ 0,T ]$,
it follows that
\begin{equation}\label{eq: mean-variance-skewness solution :eq}
\mathfrak{X}_{t} = \frac{ {B}_{t} }{ ( \gamma + 2 {\varphi }_{2} ) {D}_{t} } {e}^{ - \int_{t}^{T} \frac{ {B}_{v} }{ {D}_{v} } {I}_{v} dv } 
                   \bigg( 1 + \frac{ 12 {\varphi }_{3} }{ ( \gamma + 2 {\varphi }_{2} )^{2} } \Big( 1 - {e}^{ - \int_{t}^{T} \frac{ {B}_{v} }{ {D}_{v} } {I}_{v} dv } \Big) \bigg)^{ - \frac{1}{2} },
\end{equation}
and this leads to linear feedback control:
\begin{equation}\label{eq: mean-variance-skewness ONEC :eq}
\bar{u}_{t} = \frac{ {B}_{t} }{ ( \gamma + 2 {\varphi }_{2} ) | {D}_{t} |^{2} } {e}^{ - \int_{t}^{T} {A}_{v} dv } 
              \bigg| 1 + \frac{ 12 {\varphi }_{3} }{ ( \gamma + 2 {\varphi }_{2} )^{2} } \Big( 1 - {e}^{ - \int_{t}^{T} \frac{ {B}_{v} }{ {D}_{v} } {I}_{v} dv } \Big) \bigg|^{ - \frac{1}{2} }  
            - \frac{ {I}_{t} }{ {D}_{t} } {X}^{\bar{u}}_{t} - \frac{ {F}_{t} }{ {D}_{t} }.
\end{equation}

However, if ${\varphi }_{4} \ne 0$, namely, for a mean-variance-skewness-kurtosis problem, 
then we have ${\phi }_{4,t} = {\psi }_{4,t} = 4 {\varphi }_{4} {e}^{ \int_{t}^{T} {A}_{v} dv }$ and ${\psi }_{3,t} = 3 {\varphi }_{3} {e}^{ \int_{t}^{T} {A}_{v} dv }$, 
and we arrive at fully coupled nonlinear differential equations for the quadruplet $( \Phi, {\phi }_{2}, {\phi }_{3}, {\psi }_{2} )$:
\begin{equation*}
\left\{ \begin{aligned}
& 0 = \Big( {A}_{s} - \frac{ {B}_{s} }{ {D}_{s} } {I}_{s} \Big) {\Phi }_{s} + {\Phi }_{s}' - \Big| \frac{ {B}_{s} {\Phi }_{s} }{ {D}_{s} {\phi }_{2,s} } \Big|^{2} \Big( {\phi }_{3,s} - 3 {\varphi }_{3} {e}^{ \int_{t}^{T} {A}_{v} dv } \Big), \\
& 0 = {A}_{s} {\phi }_{2,s} + {\phi }_{2,s}' 
    + 2 \frac{ {B}_{s} }{ {D}_{s} } {I}_{s} \frac{ {\Phi }_{s} }{ {\phi }_{2,s} } {\phi }_{3,s}
    + 12 {\varphi }_{4} {e}^{ \int_{s}^{T} {A}_{v} dv } \cdot \Big| \frac{ {B}_{s} {\Phi }_{s} }{ {D}_{s} {\phi }_{2,s} } \Big|^{2}, \\
& 0 = {A}_{s} {\phi }_{3,s} + {\phi }_{3,s}' 
    + 12 {\varphi }_{4} {e}^{ \int_{s}^{T} {A}_{v} dv } \cdot \frac{ {B}_{s} }{ {D}_{s} } {I}_{s} \frac{ {\Phi }_{s} }{ {\phi }_{2,s} }, \\
& 0 = {A}_{s} {\psi }_{2,s} + {\psi }_{2,s}' + 12 {\varphi }_{4} {e}^{ \int_{s}^{T} {A}_{v} dv } \cdot \Big| \frac{ {B}_{s} {\Phi }_{s} }{ {D}_{s} {\phi }_{2,s} } \Big|^{2}, \\
& {\Phi }_{T} = 1, \quad {\phi }_{2,T} = 2 {\varphi }_{2} + \gamma, \quad {\phi }_{3,T} = 3 {\varphi }_{3}, \quad {\psi }_{2,T} = 2 {\varphi }_{j}.
\end{aligned} \right.
\end{equation*}
To the best of our knowledge, it is challenging to solve these equations unless $B I \equiv 0$.
Provided that $B I \equiv 0$ and $2 {\varphi }_{2} + \gamma > 0$, we obtain the following results in sequence:
\begin{align*}
& {\phi }_{3,t} = 3 {\varphi }_{3} {e}^{ \int_{t}^{T} {A}_{v} dv }, \quad 
  {\Phi }_{t} = {e}^{ \int_{t}^{T} {A}_{v} dv }, \quad
  {\phi }_{2,t} = {e}^{ \int_{t}^{T} {A}_{v} dv } \bigg( ( 2 {\varphi }_{2} + \gamma )^{3} + 36 {\varphi }_{4} \int_{t}^{T} \Big| \frac{ {B}_{s} }{ {D}_{s} } \Big|^{2} ds \bigg)^{ \frac{1}{3} }, \\
& \mathfrak{X}_{s} = \frac{ {B}_{s} }{ ( 2 {\varphi }_{2} + \gamma ) {D}_{s} } 
                     \bigg( 1 + \frac{ 36 {\varphi }_{4} }{ ( 2 {\varphi }_{2} + \gamma )^{3} } \int_{t}^{T} \Big| \frac{ {B}_{s} }{ {D}_{s} } \Big|^{2} ds \bigg)^{ - \frac{1}{3} }, \\
& \bar{u}_{t} = \frac{ {B}_{s} }{ ( 2 {\varphi }_{2} + \gamma ) | {D}_{s} |^{2} } {e}^{ - \int_{t}^{T} {A}_{v} dv }
                     \bigg( 1 + \frac{ 36 {\varphi }_{4} }{ ( 2 {\varphi }_{2} + \gamma )^{3} } \int_{t}^{T} \Big| \frac{ {B}_{s} }{ {D}_{s} } \Big|^{2} ds \bigg)^{ - \frac{1}{3} }
              - \frac{ {I}_{t} }{ {D}_{t} } {X}^{\bar{u}}_{t} - \frac{ {F}_{t} }{ {D}_{t} }.
\end{align*}
\end{example}

\subsubsection{Application II: nonlinear objective function with a polynomial condition}

In this subsection, we focus on the cases in which an analytical expression for an ONEC can be derived by solving a polynomial algebraic equation such as \cite[(5.14), therein]{Wang-Liu-Bensoussan-Yiu-Wei-2025}, 
even if $\varphi $ is not linear.
For this purpose, we write ${p}_{t} = \int_{t}^{T} \mathfrak{X}_{v} {I}_{v} dv$ and ${q}_{t} = \int_{t}^{T} | \mathfrak{X}_{v} |^{2} dv$ and introduce the function
\begin{align*}
H ( p,q ) & = 1 - \gamma p - \sum_{j=2}^{2n} \sum_{k=2}^{j} \binom{j}{k} k {p}^{k-1} {\alpha }_{j-k} (q) {\varphi }_{j} \big( \vec{\alpha } (q) \big) \\
          & = 1 - \gamma p - \sum_{k=2}^{2n} k {p}^{k-1} \sum_{j=k}^{2n} \binom{j}{k} {\alpha }_{j-k} (q) {\varphi }_{j} \big( \vec{\alpha } (q) \big).
\end{align*}

\begin{theorem}\label{thm: method of polynomial algebraic equation}
Suppose that there exists a constant sequence $\{ {c}_{k} \}_{k = 2,3, \ldots, 2n}$ such that
\begin{equation}\label{eq: binomial inversion :eq}
{\varphi }_{N-2m} \big( \vec{\alpha } (q) \big) = \sum_{k=0}^{m} \binom{N-2k}{N-2m} (-1)^{m-k} {c}_{N-2k} {\alpha }_{2m-2k} (q), \quad \forall m \le \frac{N}{2} - 1, ~ N = 2n, 2n-1.
\end{equation}
Then, \cref{eq: integral equation :eq} and \cref{eq: integral inequality :eq} are equivalent to
\begin{equation*}
H ( {p}_{t}, \cdot ) \equiv 1 - \gamma {p}_{t} - \sum_{k=2}^{2n} k {c}_{k} ( {p}_{t} )^{k-1} = {e}^{ - \int_{t}^{T} \frac{ {B}_{v} }{ {D}_{v} } {I}_{v} dv }, \quad 0 \ge {H}_{p} ( 2 {p}_{t}, {q}_{t} ).
\end{equation*}
\end{theorem}

\begin{proof}
See \Cref{pf-thm: method of polynomial algebraic equation}.
\end{proof}

\begin{remark}\label{rem: method of polynomial algebraic equation}
Given ${H}_{q} \equiv 0$, or equivalently, \cref{eq: binomial inversion :eq}, we obtain the following differentiation result:
\begin{align*}
  \frac{d}{dq} \varphi \big( \vec{\alpha } (q) \big)
& = \sum_{m=0}^{n-1} (n-m) {q}^{n-m-1} {\alpha }_{2n-2m} (1) \sum_{k=0}^{m} \binom{2n-2k}{2n-2m} {\alpha }_{2m-2k} (1) (-q)^{m-k} {c}_{2n-2k} \\
& = \sum_{k=0}^{n-1} (n-k) (-q)^{n-1-k} {\alpha }_{2n-2k} (1) {c}_{2n-2k} \sum_{m=k}^{n-1} \binom{n-1-k}{n-1-m} (-1)^{n-1-m} \\
& = {c}_{2}.
\end{align*}
Exploiting these results, we can provide a class of $\varphi $ satisfying \cref{eq: binomial inversion :eq} as follows:
\begin{equation*}
\varphi ( {z}_{2}, {z}_{3}, \ldots, {z}_{2n} ) = G \big( {z}_{2}, {z}_{3} - {\alpha }_{3} ( {z}_{2} ), \ldots, {z}_{2n} - {\alpha }_{2n} ( {z}_{2} ) \big),
\end{equation*}
where the continuously differentiable function $G ( {z}_{2}, {z}_{3}, \ldots, {z}_{4} )$ satisfies $G ( {z}_{2}, 0,0, \ldots, 0 ) = {c}_{2} {z}_{2}$ and 
\begin{equation*}
{G}_{N-2m} ( {z}_{2}, 0,0, \ldots, 0 ) = \sum_{k=0}^{m} \binom{N-2k}{N-2m} (-1)^{m-k} {\alpha }_{2m-2k} ( {z}_{2} ) {c}_{N-2k}, \quad m < \frac{N}{2} - 1, ~ N = 2n, 2n-1.
\end{equation*} 
This implies that
$\varphi ( \mathbb{E} \big[ ( \tilde{x} - \mathbb{E} [ \tilde{x} ] )^{2} \big], \ldots, \mathbb{E} \big[ ( \tilde{x} - \mathbb{E} [ \tilde{x} ] )^{2n} \big] ) = {c}_{2} \Var [ \tilde{x} ]$ for any normal random variable $\tilde{x}$.
For example, $G$ could be an affine function of $( {z}_{3}, \ldots, {z}_{2n} )$, and then,
\begin{align*}
\varphi & = {c}_{2} {z}_{2} + \sum_{m=0}^{n-2} \big( {z}_{2n-2m} - {\alpha }_{2n-2m} ( {z}_{2} ) \big) \sum_{k=0}^{m} \binom{2n-2k}{2n-2m} (-1)^{m-k} {\alpha }_{2m-2k} ( {z}_{2} ) {c}_{2n-2k} \\
        & \quad             + \sum_{m=0}^{n-2} {z}_{2n-1-2m} \sum_{k=0}^{m} \binom{2n-1-2k}{2n-1-2m} (-1)^{m-k} {\alpha }_{2m-2k} ( {z}_{2} ) {c}_{2n-1-2k} \\
        & = {c}_{2} {z}_{2} + \sum_{m=0}^{n} \sum_{k=0}^{m} \binom{2n-2k}{2n-2m} (-1)^{m-k} {\alpha }_{2m-2k} (1) {c}_{2n-2k} {z}_{2}^{m-k} {z}_{2n-2m} \\
        & \quad             + \sum_{m=0}^{n-2} \sum_{k=0}^{m} \binom{2n-1-2k}{2n-1-2m} (-1)^{m-k} {\alpha }_{2m-2k} (1) {c}_{2n-1-2k} {z}_{2}^{m-k} {z}_{2n-1-2m},
\end{align*}
where the last equality holds under the conventions ${z}_{0} = 1$ and ${c}_{0} = 0$.
However, although the problem with above specified $\varphi $ can be reduced to solving the polynomial algebraic equation $H ( {p}_{t}, \cdot ) \equiv {e}^{ - \int_{t}^{T} \frac{ {B}_{v} }{ {D}_{v} } {I}_{v} dv }$ in this case, 
we cannot provide a universal closed-form expression in radical for $\mathfrak{X}$ with a general $\{ {c}_{k} \}_{ k = 2, 3, \ldots, 2n }$ according to the Abel--Ruffini theorem and Galois theory.
\end{remark}

\begin{example}\label{ex: method of polynomial algebraic equation}
For $n = 2$, given $\varphi ( {z}_{2}, {z}_{3}, {z}_{4} ) = {\varphi }_{2} {z}_{2} + {\varphi }_{3} {z}_{3}$ with ${\varphi }_{3} \ne 0$, as investigated in \Cref{ex: method of undetermined coefficients}, 
we have $H ( p,q ) = 1 - ( \gamma + 2 {\varphi }_{2} ) p - 3 {\varphi }_{3} {p}^{2}$. Hence,
\begin{equation*}
0 = 1 - {e}^{ - \int_{t}^{T} \frac{ {B}_{v} }{ {D}_{v} } {I}_{v} dv } - ( \gamma + 2 {\varphi }_{2} ) \int_{t}^{T} \mathfrak{X}_{v} {I}_{v} dv - 3 {\varphi }_{3} \Big( \int_{t}^{T} \mathfrak{X}_{v} {I}_{v} dv \Big)^{2}.
\end{equation*}
Provided that $\gamma + 2 {\varphi }_{2} \ne 0$ and $( \gamma + 2 {\varphi }_{2} )^{2} + 12 {\varphi }_{3} ( 1 - {e}^{ - \int_{t}^{T} \frac{ {B}_{v} }{ {D}_{v} } {I}_{v} dv } ) > 0$ for all $t \in [ 0,T ]$, 
to satisfy the terminal condition $\int_{t}^{T} \mathfrak{X}_{v} {I}_{v} dv \to 0$ as $t \uparrow T$, one obtains
\begin{equation*}
\int_{t}^{T} \mathfrak{X}_{v} {I}_{v} dv = - \frac{ \gamma + 2 {\varphi }_{2} }{ 6 {\varphi }_{3} } 
                                             \bigg( 1 - \bigg| 1 + \frac{ 12 {\varphi }_{3} }{ ( \gamma + 2 {\varphi }_{2} )^{2} } \Big( 1 - {e}^{ - \int_{t}^{T} \frac{ {B}_{v} }{ {D}_{v} } {I}_{v} dv } \Big) \bigg|^{ \frac{1}{2} } \bigg),
\end{equation*}
which leads to \cref{eq: mean-variance-skewness solution :eq} on $\{ t: {I}_{t} \ne 0 \}$. 
Notably, ${H}_{p} ( 2 {p}_{t}, {q}_{t} ) = - ( \gamma + 2 {\varphi }_{2} ) - 12 {\varphi }_{3} \int_{t}^{T} \mathfrak{X}_{v} {I}_{v} dv \le 0$ if and only if
\begin{equation*}
\gamma + 2 {\varphi }_{2} \ge 0, \quad \frac{3}{4} ( \gamma + 2 {\varphi }_{2} )^{2} + 12 {\varphi }_{3} ( 1 - {e}^{ - \int_{t}^{T} \frac{ {B}_{v} }{ {D}_{v} } {I}_{v} dv } ) \ge 0, \quad \forall t \in [ 0,T ];
\end{equation*}
thus, \cref{eq: mean-variance-skewness ONEC :eq} gives an ONEC only if these conditions hold. 

Apart from that, given that $\varphi ( {z}_{2}, {z}_{3}, {z}_{4} ) = \frac{K}{4} ( {z}_{4} - 3 | {z}_{2} |^{2} ) + {c}_{2} {z}_{2} \equiv {c}_{2} {z}_{2} + \frac{K}{4} | {z}_{2} |^{2} ( {z}_{4} {z}_{2}^{-2} - 3 )$ 
with $K > 0$ and ${c}_{2} \ge - \frac{\gamma }{2}$, where the last factor ${z}_{4} {z}_{2}^{-2} - 3$ corresponds to the excess kurtosis, we have 
\begin{equation*}
H ( p,q ) = 1 - ( \gamma + 2 {c}_{2} ) p - K {p}^{3}, \quad {H}_{p} = - ( \gamma + 2 {c}_{2} + 3 K {p}^{2} ) \le 0,
\end{equation*} 
and hence
\begin{equation*}
0 = 1 - {e}^{ - \int_{t}^{T} \frac{ {B}_{v} }{ {D}_{v} } {I}_{v} dv } - ( \gamma + 2 {c}_{2} ) \int_{t}^{T} \mathfrak{X}_{v} {I}_{v} dv - K \Big( \int_{t}^{T} \mathfrak{X}_{v} {I}_{v} dv \Big)^{3}.
\end{equation*}
By Cardano's formula, in conjunction with the terminal condition $\int_{t}^{T} \mathfrak{X}_{v} {I}_{v} dv \to 0$ as $t \uparrow T$, we have
\begin{align*}
  \int_{t}^{T} \mathfrak{X}_{v} {I}_{v} dv 
& = \Bigg( \frac{1}{2K} \Big( 1 - {e}^{ - \int_{t}^{T} \frac{ {B}_{v} }{ {D}_{v} } {I}_{v} dv } \Big)
         + \bigg( \frac{1}{ 4 {K}^{2} } \Big( 1 - {e}^{ - \int_{t}^{T} \frac{ {B}_{v} }{ {D}_{v} } {I}_{v} dv } \Big)^{2}
                + \frac{ ( \gamma + 2 {c}_{2} )^{3} }{ 27 {K}^{3} } \bigg)^{ \frac{1}{2} } \Bigg)^{ \frac{1}{3} } \\
& \quad + \Bigg( \frac{1}{2K} \Big( 1 - {e}^{ - \int_{t}^{T} \frac{ {B}_{v} }{ {D}_{v} } {I}_{v} dv } \Big)
               - \bigg( \frac{1}{ 4 {K}^{2} } \Big( 1 - {e}^{ - \int_{t}^{T} \frac{ {B}_{v} }{ {D}_{v} } {I}_{v} dv } \Big)^{2}
                      + \frac{ ( \gamma + 2 {c}_{2} )^{3} }{ 27 {K}^{3} } \bigg)^{ \frac{1}{2} } \Bigg)^{ \frac{1}{3} }.
\end{align*}
In particular, if ${c}_{2} = - \frac{\gamma }{2}$ and $I \ne 0$, then we obtain
\begin{equation*}
\int_{t}^{T} \mathfrak{X}_{v} {I}_{v} dv = {K}^{ - \frac{1}{3} } \Big( 1 - {e}^{ - \int_{t}^{T} \frac{ {B}_{v} }{ {D}_{v} } {I}_{v} dv } \Big)^{ \frac{1}{3} }, \quad
\mathfrak{X}_{t} = \frac{ {B}_{t} }{ 3 {K}^{ \frac{1}{3} } {D}_{t} } {e}^{ - \int_{t}^{T} \frac{ {B}_{v} }{ {D}_{v} } {I}_{v} dv } \Big( 1 - {e}^{ - \int_{t}^{T} \frac{ {B}_{v} }{ {D}_{v} } {I}_{v} dv } \Big)^{ - \frac{2}{3} }.
\end{equation*}
Similarly, for $\varphi = {c}_{2n} \sum_{m=0}^{n} \binom{2n}{2m} (-1)^{m} {\alpha }_{2m} (1) {z}_{2}^{m} {z}_{2n-2m} - \frac{\gamma }{2} {z}_{2}$ and ${c}_{2n} = \frac{K}{2n} > 0$, namely,
the following specified mean-variance-standardized moment objective function: 
\begin{equation*}
{J}^{t} (u) := \mathbb{E}_{t} [ {X}^{u}_{T} ]
             - \frac{K}{2n} \big| \mathbb{E}_{t} \big[ ( {X}^{u}_{T} - \mathbb{E}_{t} [ {X}^{u}_{T} ] )^{2} \big] \big|^{n}
               \sum_{m=0}^{n} \binom{2n}{2m} (-1)^{m} {\alpha }_{2m} (1) 
                              \frac{ \mathbb{E}_{t} \big[ ( {X}^{u}_{T} - \mathbb{E}_{t} [ {X}^{u}_{T} ] )^{2n-2m} \big] }{ \big| \mathbb{E}_{t} \big[ ( {X}^{u}_{T} - \mathbb{E}_{t} [ {X}^{u}_{T} ] )^{2} \big] \big|^{n-m} },
\end{equation*}
leading to ${J}^{t} (u) = \mathbb{E}_{t} [ {X}^{u}_{T} ]$ for any normally distributed ${X}^{u}_{T}$,
we have $H ( p,q ) = 1 - K {p}^{2n-1}$ with ${H}_{p} = - (2n-1) K {p}^{2n-2} \le 0$. Hence,
\begin{equation*}
\int_{t}^{T} \mathfrak{X}_{v} {I}_{v} dv = {K}^{ - \frac{1}{2n-1} } \Big( 1 - {e}^{ - \int_{t}^{T} \frac{ {B}_{v} }{ {D}_{v} } {I}_{v} dv } \Big)^{ \frac{1}{2n-1} }.
\end{equation*}
Provided that $I \ne 0$, we obtain
\begin{equation*}
\mathfrak{X}_{t} = \frac{ {B}_{t} }{ (2n-1) {K}^{ \frac{1}{2n-1} } {D}_{t} } {e}^{ - \int_{t}^{T} \frac{ {B}_{v} }{ {D}_{v} } {I}_{v} dv } \Big( 1 - {e}^{ - \int_{t}^{T} \frac{ {B}_{v} }{ {D}_{v} } {I}_{v} dv } \Big)^{ - \frac{2n-2}{2n-1} },
\end{equation*}
which leads to the following linear feedback representation for an ONEC:
\begin{equation*}
\bar{u}_{t} = \frac{ {B}_{t} }{ (2n-1) {K}^{ \frac{1}{2n-1} } | {D}_{t} |^{2} } {e}^{ - \int_{t}^{T} {A}_{v} dv } \Big( 1 - {e}^{ - \int_{t}^{T} \frac{ {B}_{v} }{ {D}_{v} } {I}_{v} dv } \Big)^{ - \frac{2n-2}{2n-1} }  
            - \frac{ {I}_{t} }{ {D}_{t} } {X}^{\bar{u}}_{t} - \frac{ {F}_{t} }{ {D}_{t} }.
\end{equation*}
\end{example}

\subsubsection{Application III: nonlinear objective function with a homogeneity condition}

In this subsection, we derive an interesting result for nonlinear functions $\varphi $, provided that $I \equiv 0$; see \Cref{thm: MV solution is ONEC}.
The result indicates the condition, under which the ONECs for the mean-variance portfolio problem and the higher-order moment problem coincide,
given the wealth dynamics described in, e.g., \cite[Section 6.8]{Yong-Zhou-1999}, \cite[Section 5]{Hu-Jin-Zhou-2012} and \cite{Bjork-Murgoci-Zhou-2014}.

Given $I \equiv 0$, the system formulated by \cref{eq: integral equation :eq} and \cref{eq: integral inequality :eq} for determining $\mathfrak{X}$ reduces to
\begin{equation}\label{eq: equilibrium condition for I=0 :eq}
\left\{ \begin{aligned}
& \frac{ {B}_{t} }{ {D}_{t} } = \mathfrak{X}_{t} \Bigg( \gamma + \sum_{j=1}^{n} 2j (2j-1) {\alpha }_{2j-2} \Big( \int_{t}^{T} | \mathfrak{X}_{v} |^{2} dv \Big) 
                                                                                {\varphi }_{2j} \bigg( \vec{\alpha } \Big( \int_{t}^{T} | \mathfrak{X}_{v} |^{2} dv \Big) \bigg) \Bigg), \\
& 0 \le \gamma + \sum_{j=1}^{n} 2j (2j-1) {\alpha }_{2j-2} \Big( \int_{t}^{T} | \mathfrak{X}_{v} |^{2} dv \Big) {\varphi }_{2j} \bigg( \vec{\alpha } \Big( \int_{t}^{T} | \mathfrak{X}_{v} |^{2} dv \Big) \bigg),
\quad a.e. ~ t \in [ 0,T ),
\end{aligned} \right.
\end{equation}
which exactly coincides with the result reported in \cite[Theorem 5.9]{Wang-Liu-Bensoussan-Yiu-Wei-2025}.
Moreover, if $\frac{d}{dy} \varphi ( \vec{\alpha } (y) )$ is square-integrable on some interval $( 0, \rho ]$ with
$\int_{0}^{T} | \frac{ {B}_{s} }{ {D}_{s} } |^{2} ds \le \int_{0}^{\rho } | \gamma + 2 \frac{d}{dy} \varphi ( \vec{\alpha } (y) ) |^{2} dy < \infty $, 
then, by virtue of the following identity for $y \in \mathbb{R}_{+}$:
\begin{equation}\label{eq: derivative of varphi-alpha :eq}
  \sum_{j=1}^{n} j (2j-1) {\alpha }_{2j-2} (y) {\varphi }_{2j} \big( \vec{\alpha } (y) \big)
= \sum_{j=1}^{n} {\alpha }_{2j}' (y) {\varphi }_{2j} \big( \vec{\alpha } (y) \big) 
= \frac{d}{dy} \varphi \big( \vec{\alpha } (y) \big),
\end{equation}
one can find a solution $\mathfrak{X}$ for \cref{eq: equilibrium condition for I=0 :eq} via the following steps:
\begin{itemize}
\item solving the algebraic equations $\int_{t}^{T} | \frac{ {B}_{s} }{ {D}_{s} } |^{2} ds = \int_{0}^{q} | \gamma + 2 \frac{d}{dy} \varphi ( \vec{\alpha } (y) ) |^{2} dy$ for the unique solution $q = \bar{q}_{t}$
\item substituting $\int_{t}^{T} | \mathfrak{X}_{s} |^{2} ds = \bar{q}_{t}$ into \cref{eq: equilibrium condition for I=0 :eq} to arrive at
\begin{equation*}
\mathfrak{X}_{t} = \frac{ {B}_{t} }{ {D}_{t} } \bigg( \gamma + 2 \frac{d}{dy} \varphi \big( \vec{\alpha } (y) \big) \big|_{ y = \bar{q}_{t} } \bigg)^{-1},
\end{equation*}
\end{itemize}
provided that $\inf_{ y \in ( 0, \bar{x}_{0} ] } \frac{d}{dy} \varphi \big( \vec{\alpha } (y) \big) \ge - \frac{\gamma }{2}$.
Indeed, the square-integrability condition for $\frac{d}{dy} \varphi ( \vec{\alpha } (y) )$ seems much more natural than the uniform boundedness condition employed in \cite[Theorem 5.4]{Wang-Liu-Bensoussan-Yiu-Wei-2025},
and this algorithm differs from the backward iteration method for differential/integral equations. 

\begin{theorem}\label{thm: MV solution is ONEC}
Suppose that $I \equiv 0$.
Then, the control $\bar{u}$ given by
\begin{equation}\label{eq: MV time-consistent solution :eq}
\bar{u}_{t} = \frac{ {B}_{t} }{ \gamma | {D}_{t} |^{2} } {e}^{ - \int_{t}^{T} {A}_{v} dv } - \frac{ {F}_{t} }{ {D}_{t} }, \quad \mathbb{P}\text{-}a.s., ~ a.e. ~ t \in [ 0,T ),
\end{equation}
is an ONEC for \cref{eq: objective function :eq} if and only if 
\begin{equation}\label{eq: homogeneity condition :eq}
\frac{d}{dy} \varphi \big( \vec{\alpha } (y) \big) = 0, \quad \forall y \in \bigg( 0, \frac{1}{ {\gamma }^{2} } \int_{0}^{T} \Big| \frac{ {B}_{s} }{ {D}_{s} } \Big|^{2} ds \bigg).
\end{equation}
\end{theorem}

\begin{proof}
See \Cref{pf-thm: MV solution is ONEC}.
\end{proof}

\begin{example}\label{ex: MV solution is ONEC}
For the mean-variance-standardized moment problem formulated in \cite[Example 5.7]{Wang-Liu-Bensoussan-Yiu-Wei-2025} (see also \Cref{rem: ill-posed cases} in the present paper),
the objective function ${J}_{MVSM}$ satisfies \cref{eq: homogeneity condition :eq} since the corresponding $\varphi $ is a linear combination of ${z}_{j} {z}_{2}^{ - \frac{j}{2} }$
such that ${\alpha }_{j} (y) ( {\alpha }_{2} (y) )^{ - \frac{j}{2} }$ is independent of $y$.
In the same manner, one can find infinitely many objective functions that satisfy \cref{eq: homogeneity condition :eq}
by letting $\varphi $ be a linear combination of $\prod_{i=1}^{m} {z}_{ 2 {j}_{i} }^{ {k}_{i} }$ with ${j}_{i} \in \mathbb{N}_{+}$, ${k}_{i} \in \mathbb{R}$ and $\sum_{i=1}^{m} {j}_{i} {k}_{i} = 0$, which leads to
\begin{equation*}
\prod_{i=1}^{m} \big( {\alpha }_{ 2 {j}_{i} } (y) \big)^{ {k}_{i} } \equiv \prod_{i=1}^{m} \big( ( 2 {j}_{i} - 1 ) !! \big)^{ {k}_{i} }.
\end{equation*}
This is also the reason that we name \cref{eq: homogeneity condition :eq} the homogeneity condition.
For example, $\varphi = {\varphi }_{3} {z}_{3} {z}_{2}^{ - \frac{3}{2} } + {\varphi }_{4} {z}_{4} {z}_{2}^{-2}$ with any ${\varphi }_{3}, {\varphi }_{4} \in \mathbb{R}$ satisfies this homogeneity condition, 
implying that the mean-variance portfolio strategy is also an ONEC for the mean-variance-skewness-kurtosis problem with
\begin{equation*}
{J}^{t} (u) = \mathbb{E}_{t} [ {X}^{u}_{T} ] - \frac{\gamma }{2} \mathbb{E}_{t} \big[ ( {X}^{u}_{T} - \mathbb{E}_{t} [ {X}^{u}_{T} ] )^{2} \big]
            - {\varphi }_{3} \frac{ \mathbb{E}_{t} \big[ ( {X}^{u}_{T} - \mathbb{E}_{t} [ {X}^{u}_{T} ] )^{3} \big] }{ | \mathbb{E}_{t} [ ( {X}^{u}_{T} - \mathbb{E}_{t} [ {X}^{u}_{T} ] )^{2} ] |^{\frac{3}{2}} }
            - {\varphi }_{4} \frac{ \mathbb{E}_{t} \big[ ( {X}^{u}_{T} - \mathbb{E}_{t} [ {X}^{u}_{T} ] )^{4} \big] }{ | \mathbb{E}_{t} [ ( {X}^{u}_{T} - \mathbb{E}_{t} [ {X}^{u}_{T} ] )^{2} ] |^{2} }.
\end{equation*} 
In addition, $\varphi $ could be the monomial ${z}_{8} {z}_{4}^{-1} {z}_{2}^{-2}$, and then the corresponding objective function
\begin{align*}
J (u) & = \mathbb{E} [ {X}^{u}_{T} ] - \frac{\gamma }{2} \mathbb{E} \big[ ( {X}^{u}_{T} - \mathbb{E} [ {X}^{u}_{T} ] )^{2} \big]
        - \frac{ \mathbb{E} \big[ ( {X}^{u}_{T} - \mathbb{E} [ {X}^{u}_{T} ] )^{8} \big] }
               { \mathbb{E} \big[ ( {X}^{u}_{T} - \mathbb{E} [ {X}^{u}_{T} ] )^{4} \big] \big( \mathbb{E} \big[ ( {X}^{u}_{T} - \mathbb{E} [ {X}^{u}_{T} ] )^{2} \big] \big)^{2} } \\
      & = \mathbb{E} [ {X}^{u}_{T} ] - \frac{\gamma }{2} \mathbb{E} \big[ ( {X}^{u}_{T} - \mathbb{E} [ {X}^{u}_{T} ] )^{2} \big]
        - \frac{ \mathbb{E} \big[ ( {X}^{u}_{T} - \mathbb{E} [ {X}^{u}_{T} ] )^{8} \big] / \big( \mathbb{E} \big[ ( {X}^{u}_{T} - \mathbb{E} [ {X}^{u}_{T} ] )^{2} \big] \big)^{4} }
               { \mathbb{E} \big[ ( {X}^{u}_{T} - \mathbb{E} [ {X}^{u}_{T} ] )^{4} \big] / \big( \mathbb{E} \big[ ( {X}^{u}_{T} - \mathbb{E} [ {X}^{u}_{T} ] )^{2} \big] \big)^{2} }; 
\end{align*}
that is, $J$ is the mean-variance utility minus the ratio of the eighth standardized moment to the kurtosis.
\end{example}

\subsection{Characterization for the case with random parameters}

In general, $\mathfrak{X}$ is not necessarily deterministic, and hence, \cref{eq: BSDE in linear case :eq} might suggest non-Markovian feedback control.
Instead, we provide an equivalent condition for an ONEC by employing stochastic Feynman-Kac representations, namely, BSPDEs for some conditional expectations.
Unlike the equivalent condition \cref{eq: equilibrium condition reduced :eq} formulated through \cref{eq: BSDEs :eq}, 
the upcoming equivalent condition formulated through \cref{eq: BSDEs auxiliary :eq} explicitly involves the control $\mathfrak{X}$ for \cref{eq: BSDE in linear case :eq}.
For brevity, we introduce the convention that ${f}^{i,j} = 0$ for $j < 0$ and any mapping $f$ indexed by $( i,j )$.

\begin{theorem}\label{thm: BSDEs characterization}
Suppose that $\mathfrak{X} \in \mathbb{L}_{\mathbb{F}}^{2} ( 0,T; \mathbb{L}^{2n} ( \Omega ) )$,
the corresponding $\bar{u} \in \mathbb{L}_{\mathbb{F}}^{2} ( 0,T; \mathbb{L}^{2n} ( \Omega ) ) \cap \mathcal{U}$ satisfies \Cref{ass: uniform integrabiity condition with continuity},
and $( \mathbb{M}^{i,j}, \mathfrak{M}^{i,j} )$, which is indexed by $i = 0,1,2$ and $j = 0,1, \ldots, 2n-i$, fulfills the following BSDE:
\begin{equation}\label{eq: BSDEs auxiliary :eq}
\left\{ \begin{aligned}
- d \mathbb{M}^{i,j}_{t} 
& = \bigg\{ \frac{1}{2} j (j-1) \mathbb{M}^{i,j-2}_{t} | \mathfrak{X}_{t} |^{2} + j \mathfrak{M}^{i,j-1}_{t} \mathfrak{X}_{t} + i j \mathbb{M}^{i,j-1}_{t} \mathfrak{X}_{t} {I}_{t} \\
& \qquad  + i \mathfrak{M}^{i,j}_{t} {I}_{t} + i \mathbb{M}^{i,j}_{t} \Big( {A}_{t} + \frac{i-1}{2} | {I}_{t} |^{2} \Big) \bigg\} dt
  - \mathfrak{M}^{i,j}_{t} d {W}_{t}, \\
    \mathbb{M}^{i,j}_{T} & = {1}_{\{ j =0 \}}.
\end{aligned} \right.
\end{equation}
Then,
\begin{equation*}
\mathbb{M}^{i,j}_{t} = \mathbb{E}^{(i)}_{t} \bigg[ {e}^{ i \int_{t}^{T} ( {A}_{v} + \frac{i-1}{2} | {I}_{v} |^{2} ) dv } \bigg( \int_{t}^{T} \mathfrak{X}_{s} d {W}_{s} \bigg)^{j} \bigg].
\end{equation*}
Moreover, $\bar{u}$ is an ONEC if and only if the following statements hold for a.e. $t \in [ 0,T ]$, $\mathbb{P}$-a.s.:
\begin{align}\label{eq: equilibrium condition of first-order :eq}
& \mathfrak{X}_{t} \bigg( \gamma \mathbb{M}^{ 1,0 }_{t} + \sum_{j=2}^{2n} j ( j - 1 ) \mathbb{M}^{ 1, j-2 }_{t} {\varphi }_{j} ( \mathbb{M}^{ 0,2 }_{t}, \ldots, \mathbb{M}^{ 0,2n }_{t} ) \bigg) \\
\notag
& = \Big( \mathbb{M}^{ 1,0 }_{t} \frac{ {B}_{t} }{ {D}_{t} } + \mathfrak{M}^{ 1,0 }_{t} \Big)
  - \gamma \Big( \mathbb{M}^{ 1,1 }_{t} \frac{ {B}_{t} }{ {D}_{t} } + \mathfrak{M}^{ 1,1 }_{t} \Big) \\
\notag
& \quad - \sum_{j=2}^{2n} j \Big( \frac{ {B}_{t} }{ {D}_{t} } \big( \mathbb{M}^{ 1, j-1 }_{t} - \mathbb{M}^{ 0, j-1 }_{t} \mathbb{M}^{ 1,0 }_{t} \big)
                                + \mathfrak{M}^{ 1, j-1 }_{t} - \mathbb{M}^{ 0, j-1 }_{t} \mathfrak{M}^{ 1,0 }_{t} \Big) {\varphi }_{j} ( \mathbb{M}^{ 0,2 }_{t}, \ldots, \mathbb{M}^{ 0,2n }_{t} ), \\
\label{eq: equilibrium condition of second-order :eq}
& 0 \le \gamma \mathbb{M}^{ 2,0 }_{t} + \sum_{j=2}^{2n} j ( j - 1 ) \mathbb{M}^{ 2, j-2 }_{t} {\varphi }_{j} ( \mathbb{M}^{ 0,2 }_{t}, \ldots, \mathbb{M}^{ 0,2n }_{t} ).
\end{align}
\end{theorem}

\begin{proof}
See \Cref{pf-thm: BSDEs characterization}.
\end{proof}

In particular, when $A$ is deterministic and $I \equiv 0$, one obtains $\mathbb{M}^{i,j}_{t} = {e}^{ i \int_{t}^{T} {A}_{v} dv } \mathbb{M}^{0,j}_{t}$ and $\mathfrak{M}^{i,j}_{t} = {e}^{ i \int_{t}^{T} {A}_{v} dv } \mathfrak{M}^{0,j}_{t}$.
Moreover, \cref{eq: equilibrium condition of first-order :eq} and \cref{eq: equilibrium condition of second-order :eq} reduce to
\begin{equation*}
\left\{ \begin{aligned}
& \mathfrak{X}_{t} \bigg( \gamma + \sum_{j=2}^{2n} j ( j - 1 ) \mathbb{M}^{ 0, j-2 }_{t} {\varphi }_{j} ( \mathbb{M}^{ 0,2 }_{t}, \ldots, \mathbb{M}^{ 0,2n }_{t} ) \bigg)
= \frac{ {B}_{t} }{ {D}_{t} } - \sum_{j=2}^{2n} j \mathfrak{M}^{ 0, j-1 }_{t} {\varphi }_{j} ( \mathbb{M}^{ 0,2 }_{t}, \ldots, \mathbb{M}^{ 0,2n }_{t} ), \\
& 0 \le \gamma + \sum_{j=2}^{2n} j ( j - 1 ) \mathbb{M}^{ 0, j-2 }_{t} {\varphi }_{j} ( \mathbb{M}^{ 0,2 }_{t}, \ldots, \mathbb{M}^{ 0,2n }_{t} ),
\end{aligned} \right.
\end{equation*}
while \cref{eq: BSDEs auxiliary :eq} reduces to
\begin{equation*}
- d \mathbb{M}^{0,j}_{t}
= \bigg\{ \frac{1}{2} j (j-1) \mathbb{M}^{0,j-2}_{t} | \mathfrak{X}_{t} |^{2} + j \mathfrak{M}^{0,j-1}_{t} \mathfrak{X}_{t} \bigg\} dt
- \mathfrak{M}^{0,j}_{t} d {W}_{t}, \quad
\mathbb{M}^{0,j}_{T} = {1}_{\{ j =0 \}}.
\end{equation*}
However, although we have established an indirect characterization for an ONEC via $\mathfrak{X}$ and a sequence of linear BSDEs and makes it easier than before to check whether a given $\mathfrak{X}$ corresponds to an ONEC,
it is still difficult to derive an ``equilibrium strategy'' $\mathfrak{X}$ from the abovementioned equilibrium condition and BSDEs in general.

\begin{example}\label{ex: BSDEs characterization}
First, we introduce the MV benchmark by setting $\varphi \equiv 0$, implying that $\mathfrak{X}_{t} = \frac{ {B}_{t} }{ \gamma {D}_{t} }$.
Then, we suppose that $A$ is deterministic and that $I \equiv 0$, leading to $\mathbb{M}^{ 0,0 } \equiv 1$ and $\mathfrak{M}^{ 0,0 } \equiv \mathbb{M}^{ 0,1 } \equiv \mathfrak{M}^{ 0,1 } \equiv 0$.
For $\varphi = \sqrt{ {z}_{2} }$, i.e., for the mean-variance-standard deviation problem mentioned in \Cref{rem: ill-posed cases}, heuristically, we have
\begin{equation*}
\mathfrak{X}_{t} \Big( 1 + \gamma \sqrt{ \mathbb{M}^{ 0,2 }_{t} } \Big) = \frac{ {B}_{t} }{ {D}_{t} } \sqrt{ \mathbb{M}^{ 0,2 }_{t} },
\end{equation*}
where the additional term (in comparison with the MV benchmark) $\mathbb{M}^{ 0,2 }$ arises from the Lipschitz BSDE:
\begin{equation*}
- d \mathbb{M}^{0,2}_{t} = \frac{ | {B}_{t} |^{2} \mathbb{M}^{ 0,2 }_{t} }{ | {D}_{t} |^{2} \big| 1 + \gamma \sqrt{ \mathbb{M}^{ 0,2 }_{t} } \big|^{2} } dt - \mathfrak{M}^{0,2}_{t} d {W}_{t}, \quad \mathbb{M}^{0,2}_{T} = 0.
\end{equation*}
Notably, $\mathbb{M}^{0,2} \equiv \mathfrak{M}^{0,2} \equiv 0$ is the unique solution to this BSDE. 
Therefore, $\mathfrak{X} \equiv 0$, leading to 
\begin{equation*}
{X}^{\bar{u}}_{T} = \xi = \mathbb{E} [ \xi ] 
                  = {x}_{0} {e}^{ \int_{0}^{T} {A}_{v} dv } 
                  + \int_{0}^{T} {e}^{ \int_{s}^{T} {A}_{v} dv } \tilde{\mathbb{E}} \Big[ \Big( {C}_{s} - \frac{ {B}_{s} }{ {D}_{s} } {F}_{s} \Big) \Big] ds
\end{equation*}
(see also \Cref{rem: risk neutral measure}) and $\bar{u}_{t} = - \frac{1}{ {D}_{t} } ( {I}_{t} \mathbb{E} [ \xi ] + {F}_{t} )$.
This extends the corresponding results in \Cref{rem: ill-posed cases} to the random parameter case.

Additionally, for $\varphi = K {z}_{3}$ with $K \in \mathbb{R}$, i.e. for the mean-variance-skewness problem, we have
\begin{equation*}
\gamma \mathfrak{X}_{t} = \frac{ {B}_{t} }{ {D}_{t} } - 3K \mathfrak{M}^{ 0,2 }_{t},
\end{equation*}
where the additional term (in comparison with the MV benchmark) $\mathfrak{M}^{ 0,2 }$ arises from the quadratic BSDE:
\begin{equation*}
- d \mathbb{M}^{0,2}_{t} = \frac{1}{ {\gamma }^{2} } \bigg| \frac{ {B}_{t} }{ {D}_{t} } - 3K \mathfrak{M}^{ 0,2 }_{t} \bigg|^{2} dt - \mathfrak{M}^{0,2}_{t} d {W}_{t}, 
\quad \mathbb{M}^{0,2}_{T} = 0.
\end{equation*}
Interested readers can refer to \cite[Section A.9.1]{Cohen-Elliott-2015} and the references therein for the related theories.
In particular, \cite[Lemma A.9.13]{Cohen-Elliott-2015} guarantees the existence and uniqueness of $\mathfrak{M}^{ 0,2 }$ and implies that $\{ \int_{0}^{t} \mathfrak{M}^{0,2}_{s} d {W}_{s} \}_{ t \in [ 0,T ] }$ is a martingale with bounded mean oscillation.
As a result, the unique ONEC can be derived according to \Cref{rem: risk neutral measure} and \cref{eq: BSDE ending at xi :eq} with
\begin{align*}
{X}^{\bar{u}}_{T} = \xi & = {x}_{0} {e}^{ \int_{0}^{T} {A}_{v} dv } 
                          + \int_{0}^{T} {e}^{ \int_{s}^{T} {A}_{v} dv } \tilde{\mathbb{E}} \Big[ \Big( {C}_{s} - \frac{ {B}_{s} }{ {D}_{s} } {F}_{s} \Big) \Big] ds \\
                  & \quad + \frac{1}{\gamma } \mathbb{E} \bigg[ \int_{0}^{T} \Big( \frac{ {B}_{s} }{ {D}_{s} } - 3K \mathfrak{M}^{ 0,2 }_{s} \Big) \frac{ {B}_{s} }{ {D}_{s} } ds \bigg]
                          + \frac{1}{\gamma } \int_{0}^{T} \Big( \frac{ {B}_{s} }{ {D}_{s} } - 3K \mathfrak{M}^{ 0,2 }_{s} \Big) d {W}_{s}.
\end{align*}
\end{example}

\section{Concluding remarks}
\label{sec: Concluding remark}

We studied time-inconsistent stochastic control problems with higher-order central moments on the basis of a fairly general controlled SDE.
We relaxed the assumptions employed in the existing studies, including the Lipschitz continuity and integrability conditions, 
and provided both the sufficiency and necessity of the equilibrium condition for an ONEC.
To arrive at the equilibrium condition, we studied a flow of linear BSDEs indexed by different initial epochs.
Exploiting the theory of BSDEs with $\mathbb{L}^{p}$-integrable data, we demonstrated the integrability of the diagonal processes generated by the flow of BSDEs.
Afterward, we established the equivalent condition for an ONEC, which is explicitly formulated by the diagonal processes and includes no limit.

In particular, we considered the case with linearly controlled SDEs to illustrate the applications to portfolio selection.
On the one hand, given deterministic linearity coefficients, we investigated the solvability of the problem under several scenarios, including the following.
\begin{itemize}
\item The objective function is linear in all higher-order central moments, in which case the solution can be derived by solving an ODE system.
\item The objective function is nonlinear but satisfies a certain condition, in which case the solution can be derived by solving a polynomial algebraic equation.
\item The objective function satisfies the so-called homogeneity condition, in which case the solution is the same as that of the mean-variance equilibrium strategy.
\end{itemize}
On the other hand, given the random linearity coefficients, we provided a characterization formulated by finitely many recursive BSDEs. 
In this approach, we obtain the analytical solution for mean-variance-skewness portfolio selection problems.
Since the solution is formulated by a quadratic BSDE, its existence and uniqueness immediately arises.

\appendix

\section{Proof of lemmas and theorems}
\label{app: Proof of lemmas and theorems}

\subsection{Proof of \texorpdfstring{\Cref{lem: actuarial estimate}}{Lemma 3.1}}
\label{pf-lem: actuarial estimate}

Following the same line of proof as that employed for \cite[Lemma 4.1]{Wang-Liu-Bensoussan-Yiu-Wei-2025}, let us introduce the linear SDE 
$d \mathcal{E}_{s} = \mathcal{E}_{s} | {\sigma }_{x} (s) |^{2} ds - \mathcal{E}_{s} ( {b}_{x} (s) ds + {\sigma }_{x} (s) d {W}_{s} )$ starting at $\mathcal{E}_{t} = 1$, which results in
\begin{equation*}
\mathcal{E}_{\tau } {y}^{ t, \varepsilon, \zeta }_{\tau } 
= \int_{t}^{\tau } {1}_{\{ s \in [ t, t + \varepsilon ) \}} \mathcal{E}_{s} {\delta }^{\zeta } \sigma (s) \big( d {W}_{s} - {\sigma }_{x} (s) ds \big)
= \mathcal{E}_{ \tau \wedge ( t + \varepsilon ) } {y}^{ t, \varepsilon, \zeta }_{ \tau \wedge ( t + \varepsilon ) }.
\end{equation*}
Since ${b}_{x}$ and ${\sigma }_{x}$ are uniformly bounded, using Doob's maximal inequality, we can show that
\begin{equation*}
\mathbb{E}_{t} \bigg[ \sup_{ s \in [ t,T ] } \Big| \frac{ \mathcal{E}_{s} }{ \mathcal{E}_{\tau } } \Big|^{p} \bigg] < \infty, \quad \forall \tau \in [ t,T ], ~ |p| > 1.
\end{equation*}
For any fixed $\tau \in [ t,T ]$, since $\frac{ \mathbb{E}_{\tau } [ \xi ] }{ \mathcal{E}_{\tau } } \in \mathbb{L}_{ \mathcal{F}_{\tau } }^{p} ( \Omega )$ with some $p \in ( 1,q )$,
we refer to \cite{Lin-1995} for the $\mathbb{L}^{p}$-martingale representation theorem and arrive at 
$\frac{ \mathbb{E}_{\tau } [ \xi ] }{ \mathcal{E}_{\tau } } = \mathbb{E}_{t} [ \frac{\xi }{ \mathcal{E}_{\tau } } ] + \int_{t}^{\tau } {\gamma }^{\tau }_{s} d {W}_{s}$ 
for some ${\gamma }^{\tau } \in \mathbb{L}_{\mathbb{F}}^{2} ( t, \tau; \mathbb{L}^{p} ( \Omega ) )$.
Consequently, there exists a $( t, \varepsilon, \zeta )$-independent constant $K > 0$ such that 
\begin{align*}
    \sup_{ \tau \in [ t,T ] } | \mathbb{E}_{t} [ \xi {y}^{ t, \varepsilon, \zeta }_{\tau } ] |
& = \sup_{ \tau \in [ t,T ] } \bigg| \mathbb{E}_{t} \bigg[ \mathbb{E}_{ \tau \wedge ( t + \varepsilon ) } \Big[ \frac{\xi }{ \mathcal{E}_{\tau } } \Big] \cdot 
                                                           \big( \mathcal{E}_{ \tau \wedge ( t + \varepsilon ) } {y}^{ t, \varepsilon, \zeta }_{ \tau \wedge ( t + \varepsilon ) } \big) \bigg] \bigg| \\
& = \sup_{ \tau \in [ t,T ] } \bigg| \mathbb{E}_{t} \bigg[ \int_{t}^{ \tau \wedge ( t + \varepsilon ) } {\gamma }^{\tau }_{s} \mathcal{E}_{s} {\delta }^{\zeta } \sigma (s) ds 
                                                         - \int_{t}^{ \tau \wedge ( t + \varepsilon ) } \mathbb{E}_{s} \Big[ \frac{\xi }{ \mathcal{E}_{\tau } } \Big] \mathcal{E}_{s} {\delta }^{\zeta } \sigma (s) {\sigma }_{x} (s) ds \bigg] \bigg| \\
& \le K | \zeta | \bigg( \bigg| \sup_{ \tau \in [ t,T ] } \mathbb{E}_{t} \bigg[ \Big( \int_{t}^{ \tau \wedge ( t + \varepsilon ) } | {\gamma }^{\tau }_{s} |^{2} ds \Big)^{ \frac{p}{2} } \bigg] \bigg|^{ \frac{1}{p} } 
                         \bigg| \mathbb{E}_{t} \bigg[ \sup_{ s \in [ t,T ] } | \mathcal{E}_{s} |^{ \frac{p}{p-1} } \bigg] \bigg|^{ \frac{p-1}{p} } {\varepsilon }^{ \frac{1}{2} } \\
& \qquad \qquad        + \mathbb{E}_{t} \bigg[ \sup_{ \tau \in [ t,T ] } \Big| \frac{\xi }{ \mathcal{E}_{\tau } } \Big| \cdot \sup_{ s \in [ t,T ] } | \mathcal{E}_{s} | \bigg] \varepsilon \bigg).
\end{align*}

Now, we show that $\mathbb{E}_{t} [ ( \int_{t}^{\tau } | {\gamma }^{\tau }_{s} |^{2} ds )^{ \frac{p}{2} } ]$ is uniformly bounded for all $\tau \in [ t,T ]$.
Let ${\gamma }^{\tau }_{s} = 0$ for $s > \tau $, leading to 
$\int_{t}^{\upsilon } {\gamma }^{\tau }_{s} d {W}_{s} = \mathbb{E}_{ \upsilon \wedge \tau } [ \frac{\xi }{ \mathcal{E}_{\tau } } ] - \mathbb{E}_{t} [ \frac{\xi }{ \mathcal{E}_{\tau } } ]$ for any $\tau, \upsilon \in [ t,T ]$. 
By the Burkholder-Davis-Gundy (BDG) inequality and Doob's maximal inequality, there exists a constant $K > 0$ that is independent of $( t, \varepsilon, \zeta )$ but relies on $p$ such that
\begin{equation*}
\mathbb{E}_{t} \bigg[ \Big( \int_{t}^{T} | {\gamma }^{\tau }_{s} |^{2} ds \Big)^{ \frac{p}{2} } \bigg] 
\le K \mathbb{E}_{t} \bigg[ \Big| \frac{\xi }{ \mathcal{E}_{\tau } } \Big|^{p} \bigg] 
\le K \mathbb{E}_{t} \bigg[ | \xi |^{p} \sup_{ s \in [ t,T ] } | \mathcal{E}_{s} |^{-p} \bigg] < \infty.
\end{equation*}
Notably, by applying It\^o's rule to $\{ | \mathbb{E}_{s} [ \frac{\xi }{ \mathcal{E}_{\tau } } ] |^{2} \}_{ s \in [ t, \tau ] }$ and using the BDG inequality and H\"older's inequality, we can further show the uniform integrability property.
\begin{equation*}
\mathbb{E}_{t} \bigg[ \sup_{ \tau \in [ t,T ] } \Big( \int_{t}^{T} | {\gamma }^{\tau }_{s} |^{2} ds \Big)^{ \frac{p}{2} } \bigg] 
\le K \mathbb{E}_{t} \bigg[ | \xi |^{p} \sup_{ s \in [ t,T ] } | \mathcal{E}_{s} |^{-p} \bigg] < \infty.
\end{equation*}
On the other hand, we show that $\mathbb{E}_{t} [ ( \int_{t}^{T} | {\gamma }^{\tau }_{s} |^{2} ds )^{ \frac{p}{2} } ]$ is continuous in $\tau \in [ t,T ]$.
By the BDG inequality, Minkowski's inequality and Doob's maximal inequality, we have
\begin{equation*}
\bigg( \mathbb{E}_{t} \bigg[ \Big( \int_{t}^{T} | {\gamma }^{\tau }_{s} - {\gamma }^{\upsilon }_{s} |^{2} ds \Big)^{ \frac{p}{2} } \bigg] \bigg)^{ \frac{1}{p} }
\le K \bigg( \mathbb{E}_{t} \bigg[ \Big| \frac{ \mathbb{E}_{\tau } [ \xi ] }{ \mathcal{E}_{\tau } } - \frac{ \mathbb{E}_{\upsilon } [ \xi ] }{ \mathcal{E}_{\upsilon } } \Big|^{p} \bigg] \bigg)^{ \frac{1}{p} }
\le \frac{2Kp}{p-1} \mathbb{E}_{t} \bigg[ | \xi |^{p} \sup_{ s \in [ t,T ] } | \mathcal{E}_{s} |^{-p} \bigg] < \infty
\end{equation*}
for any $\tau, \upsilon \in [ t,T ]$. Based on Minkowski's inequality and the dominated convergence theorem, we know that
\begin{equation*}
\bigg| \bigg( \mathbb{E}_{t} \bigg[ \Big( \int_{t}^{T} | {\gamma }^{\tau }_{s} |^{2} ds \Big)^{ \frac{p}{2} } \bigg] \bigg)^{ \frac{1}{p} }
     - \bigg( \mathbb{E}_{t} \bigg[ \Big( \int_{t}^{T} | {\gamma }^{\upsilon }_{s} |^{2} ds \Big)^{ \frac{p}{2} } \bigg] \bigg)^{ \frac{1}{p} } \bigg|
\le \bigg( \mathbb{E}_{t} \bigg[ \Big( \int_{t}^{T} | {\gamma }^{\tau }_{s} - {\gamma }^{\upsilon }_{s} |^{2} ds \Big)^{ \frac{p}{2} } \bigg] \bigg)^{ \frac{1}{p} } \to 0
\end{equation*}
as $| \tau - \upsilon | \downarrow 0$.
In summary, in conjunction with using the Weierstrass theorem and the dominated convergence theorem, we obtain
\begin{equation*}
  \lim_{ \varepsilon \downarrow 0 } \sup_{ \tau \in [ t,T ] } \mathbb{E}_{t} \bigg[ \Big( \int_{t}^{ \tau \wedge ( t + \varepsilon ) } | {\gamma }^{\tau }_{s} |^{2} ds \Big)^{ \frac{p}{2} } \bigg]
= \mathbb{E}_{t} \bigg[ \Big( \lim_{ \varepsilon \downarrow 0 } \int_{t}^{ t + \varepsilon } | {\gamma }^{\upsilon }_{s} |^{2} ds \Big)^{ \frac{p}{2} } \bigg] 
= 0, \quad \exists \upsilon \in [ t,T ].
\end{equation*}
Therefore, $\frac{1}{\varepsilon } \sup_{ \tau \in [ t,T ] } | \mathbb{E}_{t} [ \xi {y}^{ t, \varepsilon, \zeta }_{\tau } ] |^{2} \to 0$ as $\varepsilon \downarrow 0$, namely, 
$\sup_{ \tau \in [ t,T ] } | \mathbb{E}_{t} [ \xi {y}^{ t, \varepsilon, \zeta }_{\tau } ] |^{2} = o ( \varepsilon )$.

In terms of the asymptotic expansion of ${J}^{t} ( \bar{u}^{ t, \varepsilon, \zeta } )$,
one can mirror the proof of \cite[Lemma 4.3]{Wang-Liu-Bensoussan-Yiu-Wei-2025} with the use of the derived estimate and the fact that all ${\varphi }^{t,u}_{j}$ with any fixed $t \in [ 0,T )$ and $u \in \mathcal{U}$ are $\mathcal{F}_{t}$-measurable and bounded. 
So we omit the details to shorten the length of this paper.

\subsection{Proof of \texorpdfstring{\Cref{lem: integrabiity and uniqueness of diagonal process}}{Lemma 3.3}}
\label{pf-lem: integrabiity and uniqueness of diagonal process}

Fix $t \in [ 0,T )$. 
Obviously, ${Y}^{t,u}_{T}, {Z}^{t,u}_{T} \in \mathbb{L}_{\mathcal{F}_{T}}^{ 1 + \delta } ( \Omega )$ holds for a.e. $t \in [ 0,T )$ because of \Cref{ass: uniform integrabiity condition with continuity}.
Let us introduce the equivalent martingale measures $\mathbb{P}^{(i)}$ with $i = 1,2$,
under which $\{ {W}^{(i)}_{t} := {W}_{t} - i \int_{0}^{t} {\sigma }_{x} (s) ds \}_{ t \in [ 0,T ] }$ is a one-dimensional standard Brownian motion:
\begin{equation*}
\frac{ d \mathbb{P}^{(i)} }{ d \mathbb{P} } \Big|_{ \mathcal{F}_{T} } = {e}^{ i \int_{0}^{T} {\sigma }_{x} (v) d {W}_{v} - \frac{1}{2} {i}^{2} \int_{0}^{T} | {\sigma }_{x} (v) |^{2} dv }.
\end{equation*}
Let $\mathbb{E}^{(i)}$ be the expectation operator for $\mathbb{P}^{(i)}$, and let $\mathbb{E}^{(i)}_{t} [ \cdot ] := \mathbb{E}^{(i)} [ \cdot | \mathcal{F}_{t} ]$.
Then, 
\begin{equation}\label{eq: BSDEs solution :eq}
\left\{ \begin{aligned}
& {Y}^{t,u}_{s} = \mathbb{E}^{(1)}_{s} \Big[ {e}^{ \int_{s}^{T} {b}_{x} (v) dv } {Y}^{t,u}_{T} \Big], \\ 
& {Z}^{t,u}_{s} = \mathbb{E}^{(2)}_{s} \bigg[ {e}^{ \int_{s}^{T} ( 2 {b}_{x} (v) + | {\sigma }_{x} (v) |^{2} ) dv } {Z}^{t,u}_{T} 
                                                     + \int_{s}^{T} {e}^{ \int_{s}^{\tau } ( 2 {b}_{x} (v) + | {\sigma }_{x} (v) |^{2} ) dv } 
                                                                    \big( {b}_{xx} ( \tau ) {Y}^{t,u}_{\tau } + {\sigma }_{xx} ( \tau ) \mathcal{Y}^{t,u}_{\tau } \big) d \tau \bigg],
\end{aligned} \right.
\end{equation}
up to indistinguishability on $( s, \omega ) \in [ 0,T ] \times \Omega $.
For notational simplicity, we let  
\begin{align*}
{P}^{k,u}_{t} 
& := \bigg( 1 + \gamma \mathbb{E}_{t} [ {X}^{u}_{T} ] + \sum_{j=2}^{2n} j \mathbb{E}_{t} \big[ ( {X}^{u}_{T} - \mathbb{E}_{t} [ {X}^{u}_{T} ] )^{j-1} - ( - \mathbb{E}_{t} [ {X}^{u}_{T} ] )^{j-1} \big] {\varphi }_{j}^{t,u} \bigg) {1}_{\{ k = 0 \}} \\
& \quad + \gamma {1}_{\{ k = 1 \}} 
        - {1}_{\{ k \ge 1 \}}  \sum_{j=k+1}^{2n} j \binom{j-1}{k} ( - \mathbb{E}_{t} [ {X}^{u}_{T} ] )^{j-1-k} {\varphi }_{j}^{t,u}
\end{align*}
and ${M}^{k,u}_{s} := \mathbb{E}^{(1)}_{s} [ {e}^{ \int_{0}^{T} {b}_{x} (v) dv } ( {X}^{u}_{T} )^{k} ]$.
Thus, plugging the terminal value ${Y}^{t,u}_{T}$ given by \cref{eq: BSDEs :eq} into \cref{eq: BSDEs solution :eq} yields
\begin{equation}\label{eq: polynomial expression of Y :eq}
{Y}^{t,u}_{s} = {e}^{ - \int_{0}^{s} {b}_{x} (v) dv } \sum_{k=0}^{2n-1} {P}^{k,u}_{t} {M}^{k,u}_{s},
\end{equation}
up to indistinguishability on $[ t,T ] \times \Omega $.
Moreover, for every integer $k \in [ 0, 2n-1 ]$, 
it follows from \Cref{ass: uniform integrabiity condition with continuity}, or more precisely, 
${\varphi }_{j}^{ \cdot, u } \in \mathbb{L}_{\mathbb{F},loc}^{ \frac{ 2 ( 1 + \delta ) }{ 1 - \delta } } ( 0,T; \mathbb{L}^{ \frac{ 2 n ( 1 + \delta ) }{ 2 n - ( j - 1 - k + k \vee n ) ( 1 + \delta ) } } ( \Omega ) )$ for $j \ge k + 1$, that 
${P}^{k,u} \in \mathbb{L}_{\mathbb{F},loc}^{ \frac{ 2 ( 1 + \delta ) }{ 1 - \delta } } ( 0,T; \mathbb{L}^{ \frac{ 2n ( 1 + \delta ) }{ 2n - ( k \vee n ) ( 1 + \delta ) } } ( \Omega ) )
       \subset \mathbb{L}_{\mathbb{F},loc}^{ \frac{ 2 ( 1 + \delta ) }{ 1 - \delta } } ( 0,T; \mathbb{L}^{ \frac{ 2n ( 1 + \delta ) }{ 2n - k ( 1 + \delta ) } } ( \Omega ) )$.
 Since ${M}^{k,u} \in {C}_{\mathbb{F}} ( 0,T; \mathbb{L}^{ \frac{2n}{k} } ( \Omega ) )$ for $k \in [ 1, 2n-1 ]$ and ${M}^{0,u}$ is essentially bounded,
we know that $\{ {Y}^{t,u}_{t} \}_{ t \in [ 0,T ) } \in \mathbb{L}_{\mathbb{F},loc}^{ \frac{ 2 ( 1 + \delta ) }{ 1 - \delta } } ( 0,T; \mathbb{L}^{ 1 + \delta } ( \Omega ) )$.
Similarly, we can write $\hat{Y}^{t,u}_{s} = {e}^{ - \int_{0}^{s} {b}_{x} (v) dv } \sum_{k=0}^{2n-1} \hat{P}^{k,u}_{t} \hat{M}^{k,u}_{s}$,
where each pair $( \hat{P}^{k,u}, \hat{M}^{k,u} )$ is an indistinguishable version of $( {P}^{k,u}, {M}^{k,u} )$.
Therefore, ${Y}^{t,u}_{t} = \hat{Y}^{t,u}_{t}$, $\mathbb{P}$-a.s., a.e. $t \in [ 0,T )$.

Next, we investigate the integrability and uniqueness of $\{ \mathcal{Y}^{t,u}_{t} \}_{ t \in [ 0,T ) }$.
For $u \in \mathcal{U}$ and $k = 1, 2, \ldots, 2n-1$,
due to the $\mathbb{L}^{p}$-martingale representation obtained under $\mathbb{P}^{(1)}$ (see \cite{Lin-1995}), 
there exists a unique ${\xi }^{k,u} \in \mathbb{L}_{\mathbb{F}}^{2} ( 0,T; \mathbb{L}^{ \frac{2n}{k} } ( \Omega ) )$ such that
\begin{equation}\label{eq: Lp-martingale representation :eq}
{e}^{ \int_{0}^{T} {b}_{x} (v) dv } ( {X}^{u}_{T} )^{k} = \mathbb{E}^{(1)} \Big[ {e}^{ \int_{0}^{T} {b}_{x} (v) dv } ( {X}^{u}_{T} )^{k} \Big] + \int_{0}^{T} {e}^{ \int_{0}^{s} {b}_{x} (v) dv } {\xi }^{k,u} ( d {W}_{s} - {I}_{s} ds ).
\end{equation}
In particular, for $k=0$, the above representation \cref{eq: Lp-martingale representation :eq} also holds, with ${\xi }^{0,u} \in \mathbb{L}_{\mathbb{F}}^{2} ( 0,T; \mathbb{L}^{p} ( \Omega ) )$ for any $p > 1$.
 Owing to \cref{eq: BSDEs :eq}, \cref{eq: Lp-martingale representation :eq} and \cref{eq: polynomial expression of Y :eq},
\begin{equation}\label{eq: polynomial expression of calY :eq}
\mathcal{Y}^{t,u}_{s} = {e}^{ - \int_{0}^{s} {b}_{x} (v) dv } \sum_{k=0}^{2n-1} {P}^{k,u}_{t} {\xi }^{k,u}_{s}, \quad a.e. ~ s \in [ 0,T ], ~ \mathbb{P}\text{-}a.s.
\end{equation}
Then, $\{ \mathcal{Y}^{s,u}_{s} \}_{ s \in [ 0,T ) } \in \mathbb{L}_{\mathbb{F},loc}^{ 1 + \delta } ( 0,T; \mathbb{L}^{ 1 + \delta - \rho } ( \Omega ) )$ follows from
\begin{align*}
& \mathbb{E} \bigg[ \int_{0}^{\tau } | {P}^{k,u}_{t} {\xi }^{k,u}_{t} |^{ 1 + \delta } dt \bigg] \\
& \le \mathbb{E} \bigg[ \Big( \int_{0}^{\tau } | {P}^{k,u}_{t} |^{ \frac{ 2 ( 1 + \delta ) }{ 1 - \delta } } dt \Big)^{ \frac{ 1 - \delta }{2} }  
                        \Big( \int_{0}^{\tau } | {\xi }^{k,u}_{t} |^{2} dt \Big)^{ \frac{ 1 + \delta }{2} } \bigg] \\
& \le \bigg( \mathbb{E} \bigg[ \Big( \int_{0}^{\tau } | {P}^{k,u}_{t} |^{ \frac{ 2 ( 1 + \delta ) }{ 1 - \delta } } dt \Big)^{ \frac{ n ( 1 - \delta ) }{ 2n - k ( 1 + \delta ) } } \bigg] \bigg)^{ \frac{ 2n - k ( 1 + \delta ) }{2n} }
      \bigg( \mathbb{E} \bigg[ \Big( \int_{0}^{\tau } | {\xi }^{k,u}_{t} |^{2} dt \Big)^{ \frac{n}{k} } \bigg] \bigg)^{ \frac{ k ( 1 + \delta ) }{2n} }
\end{align*}
for $k \ge 1$ and
\begin{align*}
& \mathbb{E} \bigg[ \bigg( \int_{0}^{\tau } | {P}^{0,u}_{t} {\xi }^{0,u}_{t} |^{ 1 + \delta } dt \bigg)^{ \frac{ 1 + \delta - \rho }{ 1 + \delta } } \bigg] \\
& \le \mathbb{E} \bigg[ \Big( \int_{0}^{\tau } | {P}^{0,u}_{t} |^{ \frac{ 2 ( 1 + \delta ) }{ 1 - \delta } } dt \Big)^{ \frac{ ( 1 - \delta ) ( 1 + \delta - \rho ) }{ 2 ( 1 + \delta ) } }  
                        \Big( \int_{0}^{\tau } | {\xi }^{0,u}_{t} |^{2} dt \Big)^{ \frac{ 1 + \delta - \rho }{2} } \bigg] \\
& \le \bigg( \mathbb{E} \bigg[ \Big( \int_{0}^{\tau } | {P}^{0,u}_{t} |^{ \frac{ 2 ( 1 + \delta ) }{ 1 - \delta } } dt \Big)^{ \frac{ 1 - \delta }{2} } \bigg] \bigg)^{ \frac{ 1 + \delta - \rho }{ 1 + \delta } }
      \bigg( \mathbb{E} \bigg[ \Big( \int_{0}^{\tau } | {\xi }^{0,u}_{t} |^{2} dt \Big)^{ \frac{ ( 1 + \delta - \rho ) ( 1 + \delta ) }{ 2 \rho } } \bigg] \bigg)^{ \frac{\rho }{ 1 + \delta } }.
\end{align*}
Similarly, we can write $\hat{\mathcal{Y}}^{t,u}_{s} = {e}^{ \int_{s}^{T} {A}_{v} dv } \sum_{k=0}^{2n-1} \hat{P}^{k,u}_{t} \hat{\xi }^{k,u}_{s}$,
where $\hat{P}^{k,u}$ is an indistinguishable version of ${P}^{k,u}$,
and $\mathbb{E} [ ( \int_{0}^{T} | {\xi}^{k,u}_{s} - \hat{\xi }^{k,u}_{s} |^{2} ds )^{ \frac{n}{k} } ] = 0$ for $k \ge 1$ and
$\mathbb{E} [ ( \int_{0}^{T} | {\xi}^{0,u}_{s} - \hat{\xi }^{0,u}_{s} |^{2} ds )^{ \frac{p}{2} } ] = 0$ for any $p > 1$ because of the uniqueness of the solution for \cref{eq: Lp-martingale representation :eq} as a BSDE.
Then, for any fixed $\tau \in [ 0,T )$ and some constant $K > 0$, by H\"older's inequality,
\begin{align*}
& \mathbb{E} \bigg[ \int_{0}^{\tau } | \mathcal{Y}^{s,u}_{t} - \hat{\mathcal{Y}}^{s,u}_{t} | dt \bigg] \\
& \le K \mathbb{E} \bigg[ \Big( \int_{0}^{\tau } | {P}^{0,u}_{t} |^{2} dt \cdot \int_{0}^{\tau } | {\xi }^{0,u}_{t} - \hat{\xi }^{0,u}_{t} |^{2} dt \Big)^{ \frac{1}{2} } \bigg] \\
& \quad + K \sum_{k=1}^{2n-1} \bigg( \mathbb{E} \bigg[ \Big( \int_{0}^{\tau } | {P}^{k,u}_{t} |^{2} dt \cdot 
                                                         \int_{0}^{\tau } | {\xi }^{k,u}_{t} - \hat{\xi }^{k,u}_{t} |^{2} dt \Big)^{ \frac{ 1 + \delta }{2} } \bigg] \bigg)^{ \frac{1}{ 1 + \delta } } \\
& \le K \bigg( \mathbb{E} \bigg[ \Big( \int_{0}^{\tau } | {P}^{0,u}_{t} |^{2} dt \Big)^{ \frac{ 1 + \delta }{2} } \bigg] \bigg)^{ \frac{1}{ 1 + \delta } }
        \bigg( \mathbb{E} \bigg[ \Big( \int_{0}^{\tau } | {\xi }^{0,u}_{t} - \hat{\xi }^{0,u}_{t} |^{2} dt \Big)^{ \frac{ 1 + \delta }{ 2 \delta } } \bigg] \bigg)^{ \frac{\delta }{ 1 + \delta }  } \\
& \quad + K \sum_{k=1}^{2n-1} \bigg( \mathbb{E} \bigg[ \Big( \int_{0}^{\tau } | {P}^{k,u}_{t} |^{2} dt \Big)^{ \frac{ n ( 1 + \delta ) }{ 2n - k ( 1 + \delta ) } } \bigg] \bigg)^{ \frac{ 2n - k ( 1 + \delta ) }{ 2n ( 1 + \delta ) } }
                              \bigg( \mathbb{E} \bigg[ \Big( \int_{0}^{T} | {\xi }^{k,u}_{t} - \hat{\xi }^{k,u}_{t} |^{2} dt \Big)^{ \frac{n}{k} } \bigg] \bigg)^{ \frac{k}{2n} } \\
& = 0.
\end{align*}
Therefore, $\mathcal{Y}^{t,u}_{t} = \hat{\mathcal{Y}}^{t,u}_{t}$, $\mathbb{P}$-a.s., a.e. $t \in [ 0,T )$.

In the same manner, we can show the desired results for ${Z}^{t,u}$
by ${\varphi }^{ \cdot,u}_{j} \in \mathbb{L}_{\mathbb{F},loc}^{ \frac{ 2 ( 1 + \delta ) }{ 1 - \delta } } ( 0,T; \mathbb{L}^{ \frac{ 2 n ( 1 + \delta ) }{ 2 n - ( j - 2 ) ( 1 + \delta ) } } ( \Omega ) )$ 
due to \Cref{ass: uniform integrabiity condition with continuity}.
In fact, we can write
\begin{align}\label{eq: polynomial expression of Z :eq}
{Z}^{t,u}_{s} & = {e}^{ - \int_{0}^{s} ( 2 {b}_{x} (v) + | {\sigma }_{x} (v) |^{2} ) dv } 
                  \sum_{k=0}^{2n-2} {Q}^{k,u}_{t} {N}^{k,u}_{s} \\
              & + {e}^{ - \int_{0}^{s} ( 2 {b}_{x} (v) + | {\sigma }_{x} (v) |^{2} ) dv } 
                  \sum_{k=0}^{2n-1} {P}^{k,u}_{t} \mathbb{E}^{(2)}_{s} \bigg[ \int_{s}^{T} {e}^{ \int_{0}^{\tau } ( {b}_{x} (v) + | {\sigma }_{x} (v) |^{2} ) dv } {b}_{xx} ( \tau ) {M}^{k,u}_{\tau } d \tau \bigg] \notag \\
              & + {e}^{ - \int_{0}^{s} ( 2 {b}_{x} (v) + | {\sigma }_{x} (v) |^{2} ) dv } 
                  \sum_{k=0}^{2n-1} {P}^{k,u}_{t} \mathbb{E}^{(2)}_{s} \bigg[ \int_{s}^{T} {e}^{ \int_{0}^{\tau } ( {b}_{x} (v) + | {\sigma }_{x} (v) |^{2} ) dv } {\sigma }_{xx} ( \tau ) {\xi }^{k,u}_{\tau } d \tau \bigg] \notag
\end{align}
up to indistinguishability on $[ t,T ] \times \Omega $,
where ${N}^{k,u}_{s} := \mathbb{E}^{(2)}_{s} [ {e}^{ \int_{0}^{T} ( 2 {b}_{x} (v) + | {\sigma }_{x} (v) |^{2} ) dv } ( {X}^{u}_{T} )^{k} ]$ and 
\begin{equation*}
{Q}^{k,u}_{t} := - \gamma {1}_{\{ k = 1 \}} 
                 - \sum_{j=k+2}^{2n} j (j-1) \binom{j-2}{k} ( - \mathbb{E}_{t} [ {X}^{u}_{T} ] )^{j-2-k} {\varphi }_{j}^{t,u}. 
\end{equation*}
Since ${b}_{x}, {\sigma }_{x}, {b}_{xx}, {\sigma }_{xx}$ are bounded, 
we can obtain $\{ {Z}^{t,u}_{t} \}_{ t \in [ 0,T ) } \in \mathbb{L}_{\mathbb{F},loc}^{ 1 + \delta } ( 0,T; \mathbb{L}^{ 1 + \delta - \rho } ( \Omega ) )$ by combining the following integrability properties:
\begin{itemize}
\item $\{ {Q}^{k,u}_{t} {N}^{k,u}_{t} \}_{ t \in [ 0,T ) } \in \mathbb{L}_{\mathbb{F},loc}^{ \frac{ 2 ( 1 + \delta ) }{ 1 - \delta } } ( 0,T; \mathbb{L}^{ 1 + \delta } ( \Omega ) )$ arising from
      ${N}^{k,u} \in {C}_{\mathbb{F}} ( 0,T; \mathbb{L}^{ \frac{2n}{k} } ( \Omega ) )$ for $k \in [ 1, 2n-2 ]$, the essential boundedness of ${N}^{0,u}$
      and ${Q}^{k,u} \in \mathbb{L}_{\mathbb{F},loc}^{ \frac{ 2 ( 1 + \delta ) }{ 1 - \delta } } ( 0,T; \mathbb{L}^{ \frac{ 2n ( 1 + \delta ) }{ 2n - k ( 1 + \delta ) } } ( \Omega ) )$;
\item $\{ {P}^{k,u}_{t} \mathbb{E}_{t} [ \int_{t}^{T} | {M}^{k,u}_{s} | ds ] \}_{ t \in [ 0,T ) } \in \mathbb{L}_{\mathbb{F},loc}^{ \frac{ 2 ( 1 + \delta ) }{ 1 - \delta } } ( 0,T; \mathbb{L}^{ 1 + \delta } ( \Omega ) )$ arising from
      ${M}^{k,u} \in {C}_{\mathbb{F}} ( 0,T; \mathbb{L}^{ \frac{2n}{k} } ( \Omega ) )$ for $k \in [ 1, 2n-1 ]$, the essential boundedness of ${M}^{0,u}$
      and ${P}^{k,u} \in \mathbb{L}_{\mathbb{F},loc}^{ \frac{ 2 ( 1 + \delta ) }{ 1 - \delta } } ( 0,T; \mathbb{L}^{ \frac{ 2n ( 1 + \delta ) }{ 2n - k ( 1 + \delta ) } } ( \Omega ) )$;
\item $\{ {P}^{k,u}_{t} \mathbb{E}_{t} [ \int_{t}^{T} | {\xi }^{k,u}_{s} | ds ] \}_{ t \in [ 0,T ) } \in \mathbb{L}_{\mathbb{F},loc}^{ 1 + \delta } ( 0,T; \mathbb{L}^{ 1 + \delta } ( \Omega ) )$ arising from
      the abovementioned integrability of ${P}^{k,u}$ 
      and $\{ \int_{t}^{T} | {\xi }^{k,u}_{s} | ds \}_{ t \in [ 0,T ] } \in \mathbb{L}_{\mathbb{F}}^{2} ( 0,T; \mathbb{L}^{ \frac{2n}{k} } ( \Omega ) )$, given that $k \ne 0$;
\item $\{ {P}^{0,u}_{t} \mathbb{E}_{t} [ \int_{t}^{T} | {\xi }^{0,u}_{s} | ds ] \}_{ t \in [ 0,T ) } \in \mathbb{L}_{\mathbb{F},loc}^{ 1 + \delta } ( 0,T; \mathbb{L}^{ 1 + \delta - \rho } ( \Omega ) )$ arising from
      the abovementioned integrability of ${P}^{0,u}$ 
      and $\{ \int_{t}^{T} | {\xi }^{0,u}_{s} | ds \}_{ t \in [ 0,T ] } \in \mathbb{L}_{\mathbb{F}}^{2} ( 0,T; \mathbb{L}^{p} ( \Omega ) )$ for any $p > 1$.
\end{itemize}
In addition, one can conclude that $\{ \hat{Z}^{t,u}_{t} \}_{ t \in [ 0,T ) }$ and $\{ {Z}^{t,u}_{t} \}_{ t \in [ 0,T ) }$ have the same representation up to indistinguishability on $[ t,T ] \times \Omega $; 
therefore, ${Z}^{t,u}_{t} = \hat{Z}^{t,u}_{t}$, $\mathbb{P}$-a.s., a.e. $t \in [ 0,T )$.
Hence, this proof is complete.

\subsection{Proof of \texorpdfstring{\Cref{thm: maximum principle}}{Theorem 4.1}}
\label{pf-thm: maximum principle}

By plugging the expressions \cref{eq: polynomial expression of Y :eq}, \cref{eq: polynomial expression of calY :eq} and \cref{eq: polynomial expression of Z :eq} of $( {Y}^{ t, \bar{u} }_{s}, \mathcal{Y}^{ t, \bar{u} }_{s}, {Z}^{ t, \bar{u} }_{s} )$ 
into \cref{eq: spike perturbation :eq}, one obtains
\begin{align*}
& {J}^{t} ( \bar{u}^{ t, \varepsilon, \zeta } ) - {J}^{t} ( \bar{u} ) \\
& = \mathbb{E}_{t} \bigg[ \int_{t}^{ t + \varepsilon } \Big( {Y}^{ s, \bar{u} }_{s} {\delta }^{\zeta } b (s) + \mathcal{Y}^{ s, \bar{u} }_{s} {\delta }^{\zeta } \sigma (s) 
                                                           + \frac{1}{2} {Z}^{ s, \bar{u} }_{s} | {\delta }^{\zeta } \sigma (s) |^{2} \Big) ds \bigg] \\
& \quad + \sum_{k=0}^{2n-1} \mathbb{E}_{t} \bigg[ \int_{t}^{ t + \varepsilon } {e}^{ - \int_{0}^{s} {b}_{x} (v) dv } ( {P}^{ k, \bar{u} }_{t} - {P}^{ k, \bar{u} }_{s} ) 
                                                                               \big( {M}^{ k, \bar{u} }_{s} {\delta }^{\zeta } b (s) + {\xi }^{ k, \bar{u} }_{s} {\delta }^{\zeta } \sigma (s) \big) ds \bigg] \\
& \quad + \frac{1}{2} \sum_{k=0}^{2n-2} \mathbb{E}_{t} \bigg[ \int_{t}^{ t + \varepsilon } {e}^{ - \int_{0}^{s} ( 2 {b}_{x} (v) + | {\sigma }_{x} (v) |^{2} ) dv } 
                                                                                           ( {Q}^{ k, \bar{u} }_{t} - {Q}^{ k, \bar{u} }_{s} ) {N}^{ k, \bar{u} }_{s} | {\delta }^{\zeta } \sigma (s) |^{2} ds \bigg] \\
& \quad + \frac{1}{2} \sum_{k=0}^{2n-1} \mathbb{E}_{t} \bigg[ \int_{t}^{ t + \varepsilon } {e}^{ - \int_{0}^{s} ( 2 {b}_{x} (v) + | {\sigma }_{x} (v) |^{2} ) dv } 
                                                                                           ( {P}^{ k, \bar{u} }_{t} - {P}^{ k, \bar{u} }_{s} ) ( {\lambda }^{ k, \bar{u} }_{s} + {\mu }^{ k, \bar{u} }_{s} ) | {\delta }^{\zeta } \sigma (s) |^{2} ds \bigg] 
        + o ( \varepsilon ),
\end{align*}
where
\begin{align*}
& {\lambda }^{ k, \bar{u} }_{s} := \mathbb{E}^{(2)}_{s} \bigg[ \int_{s}^{T} {e}^{ \int_{0}^{\tau } ( {b}_{x} (v) + | {\sigma }_{x} (v) |^{2} ) dv } {b}_{xx} ( \tau ) {M}^{ k, \bar{u} }_{\tau } d \tau \bigg], \\
& {\mu }^{ k, \bar{u} }_{s} := \mathbb{E}^{(2)}_{s} \bigg[ \int_{s}^{T} {e}^{ \int_{0}^{\tau } ( {b}_{x} (v) + | {\sigma }_{x} (v) |^{2} ) dv } {\sigma }_{xx} ( \tau ) {\xi }^{ k, \bar{u} }_{\tau } d \tau \bigg].
\end{align*}
Moreover, according to the proof of \Cref{lem: integrabiity and uniqueness of diagonal process},
${\lambda }^{ k, \bar{u} } \in {C}_{\mathbb{F}} ( 0,T; \mathbb{L}^{ \frac{2n}{k} } ( \Omega ) )$ and 
${\mu }^{ k, \bar{u} } \in \mathbb{L}_{\mathbb{F}}^{2} ( 0,T; \mathbb{L}^{ \frac{2n}{k} } ( \Omega ) )$ for $k \ge 1$, 
${\mu }^{ 0, \bar{u} } \in \mathbb{L}_{\mathbb{F}}^{2} ( 0,T; \mathbb{L}^{p} ( \Omega ) )$ for any $p > 1$,
and ${\lambda }^{ 0, \bar{u} }$ is essentially bounded.
Thus, $( {\lambda }^{ k, \bar{u} }, {\mu }^{ k, \bar{u} } )$ has the same continuity and integrability properties as $( {M}^{ k, \bar{u} }, {\xi }^{ k, \bar{u} } )$ does.
Similarly, $( {Q}^{ k, \bar{u} }, {N}^{ k, \bar{u} } )$ also has the same continuity and integrability properties as $( {P}^{ k, \bar{u} }, {M}^{ k, \bar{u} } )$ does,
while ${M}^{ k, \bar{u} }$ has stronger integrability properties than ${\xi }^{ k, \bar{u} }$ does.
Therefore, to obtain \cref{eq: spike perturbation modified :eq} under the Lipschitz continuity condition \cref{eq: Lipschitz continuity :eq} (i.e., $| {b}_{x} | \le K$, $| {\delta }^{\zeta } b (s) | \le K | \zeta |$, etc.), 
it indeed suffices to show that
\begin{equation*}
\mathbb{E}_{t} \bigg[ \int_{t}^{ t + \varepsilon } | ( {P}^{ k, \bar{u} }_{t} - {P}^{ k, \bar{u} }_{s} ) {\xi }^{ k, \bar{u} }_{s} | ds \bigg] = o ( \varepsilon ).
\end{equation*}

Since $\frac{2n}{ k \vee n } \in ( 1,2 ]$, by the generalized version of Minkowski's inequality, we know that
\begin{equation*}
\bigg( \int_{0}^{T} \big( \mathbb{E} [ | {\xi }^{ k, \bar{u} }_{s} |^{ \frac{2n}{ k \vee n } } ] \big)^{ \frac{ k \vee n }{n} } ds \bigg)^{ \frac{1}{2} }
\le \bigg( \mathbb{E} \bigg[ \Big( \int_{0}^{T} | {\xi }^{ k, \bar{u} }_{s} |^{2} ds \Big)^{ \frac{n}{ k \vee n } } \bigg] \bigg)^{ \frac{ k \vee n }{2n} } < \infty,
\end{equation*}
which implies that $\mathbb{E} [ | {\xi }^{ k, \bar{u} }_{t} |^{ \frac{2n}{ k \vee n } } ] < \infty $ for a.e. $t \in [ 0,T ]$.
Then, according to H\"older's inequality, by rearranging the terms, one obtains
\begin{align*}
& \mathbb{E} \bigg[ \limsup_{ \varepsilon \downarrow 0 } \frac{1}{\varepsilon } \int_{t}^{ t + \varepsilon } | ( {P}^{ k, \bar{u} }_{t} - {P}^{ k, \bar{u} }_{s} ) {\xi }^{ k, \bar{u} }_{s} | ds \bigg] \\
& \le \bigg( \mathbb{E} \bigg[ \Big( \limsup_{ \varepsilon \downarrow 0 } 
                                     \frac{1}{\varepsilon } \int_{t}^{ t + \varepsilon } | {P}^{ k, \bar{u} }_{t} - {P}^{ k, \bar{u} }_{s} |^{2} ds \Big)^{ \frac{n}{ 2n - k \vee n } } \bigg] \bigg)^{ \frac{ 2n - k \vee n }{2n} } \\
& \quad \times K \bigg( \mathbb{E} \bigg[ | {\xi }^{ k, \bar{u} }_{t} |^{ \frac{2n}{ k \vee n } }
                                        + \Big( \limsup_{ \varepsilon \downarrow 0 } 
                                                \frac{1}{\varepsilon } \int_{t}^{ t + \varepsilon } \big| | {\xi }^{ k, \bar{u} }_{s} |^{2} - | {\xi }^{ k, \bar{u} }_{t} |^{2} \big| ds \Big)^{ \frac{n}{ k \vee n } } \bigg] \bigg)^{ \frac{ k \vee n }{2n} }
\end{align*}
for some constant $K > 0$ that only relies on $n$.
On the one hand, $\int_{0}^{T} | {\xi }^{ k, \bar{u} }_{s} |^{2} ds < \infty $, $\mathbb{P}$-a.s.,
follows from ${\xi }^{ k, \bar{u} } \in \mathbb{L}_{\mathbb{F}}^{2} ( 0,T; \mathbb{L}^{ \frac{2n}{k} } ( \Omega ) )$ for $k \ge 1$ 
and ${\xi }^{ 0, \bar{u} } \in \mathbb{L}_{\mathbb{F}}^{2} ( 0,T; \mathbb{L}^{p} ( \Omega ) )$ with any fixed $p > 1$.
Therefore, applying the Lebesgue differentiation theorem to $| {\xi }^{ k, \bar{u} } |^{2}$ yields
\begin{equation*}
\limsup_{ \varepsilon \downarrow 0 } \frac{1}{\varepsilon } \int_{t}^{ t + \varepsilon } \big| | {\xi }^{ k, \bar{u} }_{s} |^{2} - | {\xi }^{ k, \bar{u} }_{t} |^{2} \big| ds = 0, \quad a.e. ~ t \in [ 0,T ], ~ \mathbb{P}\text{-}a.s.
\end{equation*}
On the other hand, ${P}^{k,u} \in \mathbb{L}_{\mathbb{F},loc}^{ \frac{ 2 ( 1 + \delta ) }{ 1 - \delta } } ( 0,T; \mathbb{L}^{ \frac{ 2n ( 1 + \delta ) }{ 2n - ( k \vee n ) ( 1 + \delta ) } } ( \Omega ) )
                          \subset \mathbb{L}_{\mathbb{F},loc}^{2} ( 0,T; \mathbb{L}^{ \frac{ 2n }{ 2n - k \vee n } } ( \Omega ) )$
is $\mathbb{P}$-a.s. sample continuous; see the proof of \Cref{lem: integrabiity and uniqueness of diagonal process}.
Consequently,
\begin{equation*}
\limsup_{ \varepsilon \downarrow 0 } \frac{1}{\varepsilon } \int_{t}^{ t + \varepsilon } | {P}^{ k, \bar{u} }_{t} - {P}^{ k, \bar{u} }_{s} |^{2} ds = 0, \quad \forall t \in [ 0,T ], ~ \mathbb{P}\text{-}a.s.
\end{equation*}
Therefore,
\begin{equation*}
\mathbb{E} \bigg[ \limsup_{ \varepsilon \downarrow 0 } \frac{1}{\varepsilon } \int_{t}^{ t + \varepsilon } | ( {P}^{ k, \bar{u} }_{t} - {P}^{ k, \bar{u} }_{s} ) {\xi }^{ k, \bar{u} }_{s} | ds \bigg] = 0,
\end{equation*}
and hence,
\begin{equation*}
\limsup_{ \varepsilon \downarrow 0 } \frac{1}{\varepsilon } \mathbb{E}_{t} \bigg[ \int_{t}^{ t + \varepsilon } | ( {P}^{ k, \bar{u} }_{t} - {P}^{ k, \bar{u} }_{s} ) {\xi }^{ k, \bar{u} }_{s} | ds \bigg]
\le \mathbb{E}_{t} \bigg[ \limsup_{ \varepsilon \downarrow 0 } \frac{1}{\varepsilon } \int_{t}^{ t + \varepsilon } | ( {P}^{ k, \bar{u} }_{t} - {P}^{ k, \bar{u} }_{s} ) {\xi }^{ k, \bar{u} }_{s} | ds \bigg] = 0,
\end{equation*}
which yields our desired result $\mathbb{E}_{t} [ \int_{t}^{ t + \varepsilon } | ( {P}^{ k, \bar{u} }_{t} - {P}^{ k, \bar{u} }_{s} ) {\xi }^{ k, \bar{u} }_{s} | ds ] = o ( \varepsilon )$.
Thus, \cref{eq: spike perturbation modified :eq} holds.

If $\bar{u} \in \mathcal{U}$ satisfies the equilibrium condition \cref{eq: equilibrium condition :eq}, then substituting it into \cref{eq: spike perturbation modified :eq} immediately yields
\begin{equation*}
{J}^{t} ( \bar{u}^{ t, \varepsilon, \zeta } ) - {J}^{t} ( \bar{u} ) \le o ( \varepsilon ), \quad \forall \zeta \in \mathbb{L}_{ \mathcal{F}_{t} }^{2n} ( \Omega ), ~ \mathbb{P}\text{-}a.s., ~ a.e. ~ t \in [ 0,T ),
\end{equation*}
which implies that $\bar{u}$ is an ONEC. 
That is, the equilibrium condition \cref{eq: equilibrium condition :eq} 
is sufficient. Conversely, if $\bar{u}$ is an ONEC, then plugging \cref{eq: spike perturbation modified :eq} back into \Cref{def: open-loop Nash equilibrium control}, 
in conjunction with the integrability of ${Y}^{ s, \bar{u} }_{s} {\delta }^{\zeta } b (s) + \mathcal{Y}^{ s, \bar{u} }_{s} {\delta }^{\zeta } \sigma (s) + \frac{1}{2} {Z}^{ s, \bar{u} }_{s} | {\delta }^{\zeta } \sigma (s) |^{2}$ guaranteed by
\cref{eq: Lipschitz continuity :eq} and \Cref{lem: integrabiity and uniqueness of diagonal process}, yields
\begin{align*}
0 & \ge \mathbb{E} \bigg[ \limsup_{ \varepsilon \downarrow 0 } \frac{1}{\varepsilon } \int_{t}^{ t + \varepsilon } 
                                                               \mathbb{E}_{t} \Big[ {Y}^{ s, \bar{u} }_{s} {\delta }^{\zeta } b (s) + \mathcal{Y}^{ s, \bar{u} }_{s} {\delta }^{\zeta } \sigma (s) 
                                                                                  + \frac{1}{2} {Z}^{ s, \bar{u} }_{s} | {\delta }^{\zeta } \sigma (s) |^{2} \Big] ds \bigg] \\
  & \ge \limsup_{ \varepsilon \downarrow 0 } \frac{1}{\varepsilon } \int_{t}^{ t + \varepsilon } 
                                             \mathbb{E} \Big[ {Y}^{ s, \bar{u} }_{s} {\delta }^{\zeta } b (s) + \mathcal{Y}^{ s, \bar{u} }_{s} {\delta }^{\zeta } \sigma (s) 
                                                            + \frac{1}{2} {Z}^{ s, \bar{u} }_{s} | {\delta }^{\zeta } \sigma (s) |^{2} \Big] ds \\
  &   = \mathbb{E} \Big[ \Big( {Y}^{ t, \bar{u} }_{t} {\delta }^{z} b (t) + \mathcal{Y}^{ t, \bar{u} }_{t} {\delta }^{z} \sigma (t) + \frac{1}{2} {Z}^{ t, \bar{u} }_{t} | {\delta }^{z} \sigma (t) |^{2} \Big) {1}_{S} \Big]
\end{align*}
for a.e. $t \in [ 0,T )$ and $\zeta = z {1}_{S}$ with arbitrarily fixed $z \in \mathbb{R}$ and $S \in \mathcal{F}_{t}$. 
Hence, \cref{eq: equilibrium condition :eq} follows from the arbitrariness of $S \in \mathcal{F}_{t}$.
That is, the necessity of the equilibrium condition \cref{eq: equilibrium condition :eq} arises. 
Therefore, this proof is complete.

\subsection{Proof of \texorpdfstring{\Cref{thm: method of undetermined coefficients}}{Theorem 5.3}}
\label{pf-thm: method of undetermined coefficients}

Let us treat the desired expression for ${Y}^{t}$ as an ansatz so that
\begin{align*}
d {Y}^{t}_{s} & = \bigg( {\Phi }_{s}' - \sum_{j=2}^{2n} \big( {\phi }_{j,s}' ( {\chi }_{s} - {\chi }_{t} )^{j-1} - {\psi }_{j,s}' \mathbb{E}_{t} [ ( {\chi }_{s} - {\chi }_{t} )^{j-1} ] \big) \bigg) ds \\
        & \quad - \sum_{j=2}^{2n} (j-1) {\phi }_{j,s} ( {\chi }_{s} - {\chi }_{t} )^{j-2} \mathfrak{X}_{s} d {W}_{s}\\
        & \quad - \sum_{j=3}^{2n} \binom{j-1}{2} \big( {\phi }_{j,s} ( {\chi }_{s} - {\chi }_{t} )^{j-3} | \mathfrak{X}_{s} |^{2} - {\psi }_{j,s} \mathbb{E}_{t} [ ( {\chi }_{s} - {\chi }_{t} )^{j-3} | \mathfrak{X}_{s} |^{2} ] \big) ds,
\end{align*}
which leads to the desired expression for $\mathcal{Y}^{t}$.
Consequently, ${B}_{s} {\Phi }_{s} = {D}_{s} {\phi }_{2,s} \mathfrak{X}_{s}$, i.e., the desired expression for $\mathfrak{X}$, immediately follows from the equilibrium condition $0 = {B}_{s} {Y}^{s}_{s} + {D}_{s} \mathcal{Y}^{s}_{s}$.
Since \cref{eq: ODEs :eq} has a solution, we have
\begin{align}\label{eq: ds term of BSDE in linear case :eq}
0 & = {A}_{s} \bigg( {\Phi }_{s} - \sum_{j=2}^{2n} \big( {\phi }_{j,s} ( {\chi }_{s} - {\chi }_{t} )^{j-1} - {\psi }_{j,s} \mathbb{E}_{t} [ ( {\chi }_{s} - {\chi }_{t} )^{j-1} ] \big) \bigg) \\
  & \quad - {I}_{s} \frac{ {B}_{s} {\Phi }_{s} }{ {D}_{s} {\phi }_{2,s} } 
            \sum_{j=1}^{2n-1} j {\phi }_{j+1,s} ( {\chi }_{s} - {\chi }_{t} )^{j-1} \notag \\
  & \quad + {\Phi }_{s}' - \sum_{j=2}^{2n} \big( {\phi }_{j,s}' ( {\chi }_{s} - {\chi }_{t} )^{j-1} - {\psi }_{j,s}' \mathbb{E}_{t} [ ( {\chi }_{s} - {\chi }_{t} )^{j-1} ] \big) \notag \\
  & \quad - \Big| \frac{ {B}_{s} {\Phi }_{s} }{ {D}_{s} {\phi }_{2,s} } \Big|^{2}
            \sum_{j=1}^{2n-2} \binom{j+1}{2} \big( {\phi }_{j+2,s} ( {\chi }_{s} - {\chi }_{t} )^{j-1} - {\psi }_{j+2,s} \mathbb{E}_{t} [ ( {\chi }_{s} - {\chi }_{t} )^{j-1} ] \big). \notag
\end{align}
implying that $d {Y}^{t}_{s} = - ( {A}_{s} {Y}^{t}_{s} + {I}_{s} \mathcal{Y}^{t}_{s} ) ds + \mathcal{Y}^{t}_{s} d {W}_{s}$ for our ansatz.
Additionally, plugging the terminal conditions in \cref{eq: ODEs :eq} into our ansatz, we know that the assumed ${Y}^{t}_{T}$ satisfies the terminal condition in \cref{eq: BSDE in linear case :eq}.
Thus far, we have verified that our desired expression for $( \mathfrak{X}, {Y}^{t}, \mathcal{Y}^{t} )$ satisfies \cref{eq: BSDE in linear case :eq}, and thus, this proof is completed.

\subsection{Proof of \texorpdfstring{\Cref{thm: method of polynomial algebraic equation}}{Theorem 5.5}}
\label{pf-thm: method of polynomial algebraic equation}

In fact, \cref{eq: integral equation :eq} and \cref{eq: integral inequality :eq} can be re-expressed as
\begin{equation*}
H ( {p}_{t}, {q}_{t} ) \frac{ {B}_{t} }{ {D}_{t} } {I}_{t} dt = {H}_{p} ( {p}_{t}, {q}_{t} ) d {p}_{t}, \quad
0 \ge {H}_{p} ( 2 {p}_{t}, {q}_{t} ).
\end{equation*}
Notably, $H ( {p}_{t}, \cdot ) \equiv {e}^{ - \int_{t}^{T} \frac{ {B}_{v} }{ {D}_{v} } {I}_{v} dv }$ if ${H}_{q} \equiv 0$.
Therefore, in the rest of this proof, 
it suffices to show the equivalence between the given condition \cref{eq: binomial inversion :eq} and ${H}_{q} \equiv 0$.
Given the expression of $H ( p,q )$, we know that ${H}_{q} \equiv 0$ is equivalent to $\sum_{j=k}^{2n} \binom{j}{k} {\alpha }_{j-k} (q) {\varphi }_{j} ( \vec{\alpha } (q) ) \equiv {c}_{k} \in \mathbb{R}$ 
for $k = 2,3, \ldots, 2n$ with some constant sequence $\{ {c}_{k} \}_{k = 2,3, \ldots, 2n}$, namely,
\begin{equation}\label{eq: binomial transform :eq}
\sum_{j=0}^{k} \binom{N-2j}{N-2k} {\alpha }^{2k-2j} (q) {\varphi }_{N-2j} \big( \vec{\alpha } (q) \big) = {c}_{N-2k}, \quad \forall k \le \frac{N}{2} - 1, ~ N = 2n, 2n-1.
\end{equation}
Given \cref{eq: binomial transform :eq}, based on the method of binomial inversion, one obtains
\begin{align*}
  {\varphi }_{N-2m} \big( \vec{\alpha } (q) \big)
& = \sum_{j=0}^{m} \binom{N-2j}{N-2m} {\alpha }_{2m-2j} (1) {\varphi }_{N-2j} \big( \vec{\alpha } (q) \big) \sum_{k=j}^{m} \binom{m-j}{k-j} {q}^{k-j} (-q)^{m-k} \\
& = \sum_{k=0}^{m} \sum_{j=0}^{k} \binom{N-2j}{N-2m} \binom{2m-2j}{2m-2k} (-1)^{m-k} {\alpha }_{2m-2k} (q) {\alpha }_{2k-2j} (q) {\varphi }_{N-2j} \big( \vec{\alpha } (q) \big) \\
& = \sum_{k=0}^{m} \binom{N-2k}{N-2m} (-1)^{m-k} {\alpha }_{2m-2k} (q) \sum_{j=0}^{k} \binom{N-2j}{N-2k} {\alpha }_{2k-2j} (q) {\varphi }_{N-2j} \big( \vec{\alpha } (q) \big) \\
& = \sum_{k=0}^{m} \binom{N-2k}{N-2m} (-1)^{m-k} {\alpha }_{2m-2k} (q) {c}_{N-2k}, \quad \forall m \le \frac{N}{2} - 1, ~ N = 2n, 2n-1.
\end{align*}
Conversely, with the given condition \cref{eq: binomial inversion :eq}, one can also use the method of binomial inversion to arrive at
\begin{equation*}
(-1)^{m} {c}_{N-2m} = \sum_{k=0}^{m} \binom{N-2k}{N-2m} (-1)^{m-k} {\alpha }_{2m-2k} (q) \cdot (-1)^{k} {\varphi }_{N-2k} \big( \vec{\alpha } (q) \big),
\end{equation*}
namely, \cref{eq: binomial transform :eq}. Hence, \cref{eq: binomial inversion :eq} is equivalent to ${H}_{q} \equiv 0$ and leads to the desired result.

\subsection{Proof of \texorpdfstring{\Cref{thm: MV solution is ONEC}}{Theorem 5.8}}
\label{pf-thm: MV solution is ONEC}

Owing to the uniqueness of the solutions for the BSDEs \cref{eq: BSDE ending at xi :eq} and \cref{eq: BSDE in linear case :eq}, 
the equivalence between \cref{eq: MV time-consistent solution :eq} and $\mathfrak{X}_{t} = \frac{ {B}_{t} }{ \gamma {D}_{t} }$ for a.e. $t \in [ 0,T )$ arises.
Since \Cref{thm: maximum principle} with \Cref{rem: equilibrium condition reduced} and \cref{eq: BSDE in linear case :eq} indeed provides the sufficiency and necessity of \cref{eq: equilibrium condition for I=0 :eq} for the ONECs,
it suffices to show the equivalence between \cref{eq: homogeneity condition :eq} and
\begin{equation}\label{eq: vanishing derivative :eq}
\sum_{j=1}^{n} j (2j-1) {\alpha }_{2j-2} \Big( \int_{t}^{T} | \mathfrak{X}_{s} |^{2} ds \Big) {\varphi }_{2j} \bigg( \vec{\alpha } \Big( \int_{t}^{T} | \mathfrak{X}_{s} |^{2} ds \Big) \bigg) = 0, \quad a.e. ~ t \in \{ s: | {B}_{s} | > 0 \}.
\end{equation}
Given \cref{eq: homogeneity condition :eq}, \cref{eq: vanishing derivative :eq} immediately follows from \cref{eq: derivative of varphi-alpha :eq}.
Conversely, given \cref{eq: vanishing derivative :eq}, as $\int_{t}^{T} | \mathfrak{X}_{s} |^{2} ds = \int_{t}^{T} | \frac{ {B}_{s} }{ \gamma {D}_{s} } |^{2} ds$ is absolutely continuous and decreasing in $t$,
\cref{eq: vanishing derivative :eq} must hold for every $t \in [ 0,T )$, which leads to \cref{eq: homogeneity condition :eq}.

\subsection{Proof of \texorpdfstring{\Cref{thm: BSDEs characterization}}{Theorem 5.10}}
\label{pf-thm: BSDEs characterization}

Let us introduce the process $\{ {\chi }^{ t,x, \mathfrak{X} } \}_{ s \in [ t,T ] }$ via ${\chi }^{ t,x, \mathfrak{X} }_{s} = x + \int_{t}^{s} \mathfrak{X}_{v} d {W}_{v}$.
Then, for $i=0,1,2$, $j = 0, 1, \ldots, 2n-1$ and $m \in \mathbb{R}$, 
the random field ${M}^{i,j} ( t,x ) := \mathbb{E}^{(i)}_{t} [ {e}^{ i \int_{t}^{T} ( {A}_{v} + \frac{i-1}{2} | {I}_{v} |^{2} ) dv } ( {\chi }^{ t,x, \mathfrak{X} }_{T} )^{j} ]$ enables a semi-martingale decomposition.
In fact, ${M}^{i,j}$ has the following stochastic Feynman-Kac representation:
\begin{equation*}
\left\{ \begin{aligned}
- d {M}^{i,j} ( t,x ) 
& = \bigg\{ \frac{1}{2} {M}^{i,j}_{xx} ( t,x ) | \mathfrak{X}_{t} |^{2} + \mathcal{M}^{i,j}_{x} ( t,x ) \mathfrak{X}_{t} + i {M}^{i,j}_{x} ( t,x ) \mathfrak{X}_{t} {I}_{t}  \\
& \qquad  + i {M}^{i,j} ( t,x ) \Big( {A}_{t} + \frac{i-1}{2} | {I}_{t} |^{2} \Big) \bigg\} dt
  - \mathcal{M}^{i,j} ( t,x ) ( d {W}_{t} - i {I}_{t} dt ), \\
    {M}^{i,j} ( T,x ) & = {x}^{j}.
\end{aligned} \right.
\end{equation*}
Interested readers can refer to \cite{Ma-Yong-1997,Yong-Zhou-1999} for the related theories.
Notably, as the pair $( {M}^{i,j}, \mathcal{M}^{i,j} )$ with $j \ge 1$
has been given by the semi-martingale decomposition of the given conditional expectations,
one can eliminate the questions about the existence, uniqueness and moment estimates of the solution for the corresponding BSPDE.
Moreover, since ${M}^{i,j}_{x} = j {M}^{i,j-1}$, the stochastic Feynman-Kac representation can be re-expressed as
\begin{equation*}
\left\{ \begin{aligned}
- d {M}^{i,j} ( t,x ) 
& = \bigg\{ \frac{1}{2} j (j-1) {M}^{i,j-2} ( t,x ) | \mathfrak{X}_{t} |^{2} + j \mathcal{M}^{i,j-1} ( t,x ) \mathfrak{X}_{t} + i j {M}^{i,j-1} ( t,x ) \mathfrak{X}_{t} {I}_{t}  \\
& \qquad  + i \mathcal{M}^{i,j} ( t,x ) {I}_{t} + i {M}^{i,j} ( t,x ) \Big( {A}_{t} + \frac{i-1}{2} | {I}_{t} |^{2} \Big) \bigg\} dt
  - \mathcal{M}^{i,j} ( t,x ) d {W}_{t}, \\
    {M}^{i,j} ( T,x ) & = {x}^{j}.
\end{aligned} \right.
\end{equation*}
Plugging $x = 0$ into these equations and letting $( \mathbb{M}^{i,j}, \mathfrak{M}^{i,j} ) = ( {M}^{i,j} ( \cdot, 0 ), \mathcal{M}^{i,j} ( \cdot, 0 ) )$ immediately yields the BSDE \cref{eq: BSDEs auxiliary :eq}.
Conversely, with a fixed $i$, starting at  
\begin{equation*}
- d \mathbb{M}^{i,0}_{t} 
= \bigg\{ i \mathfrak{M}^{i,0}_{t} {I}_{t} + i \mathbb{M}^{i,0}_{t} \Big( {A}_{t} + \frac{i-1}{2} | {I}_{t} |^{2} \Big) \bigg\} dt
- \mathfrak{M}^{i,0}_{t} d {W}_{t}, \quad
\mathbb{M}^{i,0}_{T} = 1,
\end{equation*}
the BDSE \cref{eq: BSDEs auxiliary :eq} for every $j = 1, \ldots, 2n-i$ has a unique solution:
\begin{equation*}
( \mathbb{M}^{i,j}, \mathfrak{M}^{i,j} ) \in {C}_{\mathbb{F}} \big( 0,T; \mathbb{L}^{ \frac{2n}{j} } ( \Omega ) \big) \times \mathbb{L}_{\mathbb{F}}^{2} \big( 0,T; \mathbb{L}^{ \frac{2n}{j} } ( \Omega ) \big).
\end{equation*}
This proves the first assertion of \Cref{thm: BSDEs characterization}.
Furthermore, by the It\^o-Kunita-Ventzel formula (see \cite[Theorem 1.5.3.2]{Jeanblanc-Yor-Chesney-2009}), we obtain
\begin{equation*}
  d {M}^{i,j} ( s, {\chi }_{s} ) 
= - i {M}^{i,j} ( s, {\chi }_{s} ) {A}_{s} ds 
    + \big( j {M}^{i,j-1} ( s, {\chi }_{s} ) \mathfrak{X}_{s} + \mathcal{M}^{i,j} ( s, {\chi }_{s} ) \big) ( d {W}_{s} - i {I}_{s} ds ).
\end{equation*}
Hence, owing to \cref{eq: BSDE in linear case :eq}, we have the following expression for $( {Y}^{t}_{s}, \mathcal{Y}^{t}_{s} )$:
\begin{align*}
{Y}^{t}_{s} & = {M}^{ 1,0 } ( s, {\chi }_{s} - {\chi }_{t} ) 
                \bigg( 1 + \sum_{j=2}^{2n} j {M}^{ 0, j-1 } ( t,0 ) {\varphi }_{j} \Big( {M}^{ 0,2 } ( t,0 ), \ldots, {M}^{ 0,2n } ( t,0 ) \Big) \bigg) \\
      & \quad - \gamma {M}^{ 1,1 } ( s, {\chi }_{s} - {\chi }_{t} )
              - \sum_{j=2}^{2n} j {M}^{ 1, j-1 } ( s, {\chi }_{s} - {\chi }_{t} ) {\varphi }_{j} \Big( {M}^{ 0,2 } ( t,0 ), \ldots, {M}^{ 0, 2n } ( t,0 ) \Big)
\end{align*}
and
\begin{align*}
\mathcal{Y}^{t}_{s} & = \mathcal{M}^{ 1,0 } ( s, {\chi }_{s} - {\chi }_{t} ) 
                        \bigg( 1 + \sum_{j=2}^{2n} j {M}^{ 0, j-1 } ( t,0 ) {\varphi }_{j} \Big( {M}^{ 0,2 } ( t,0 ), \ldots, {M}^{ 0,2n } ( t,0 ) \Big) \bigg) \\
              & \quad - \gamma \big( {M}^{ 1,0 } ( s, {\chi }_{s} - {\chi }_{t} ) \mathfrak{X}_{s} + \mathcal{M}^{ 1,1 } ( s, {\chi }_{s} - {\chi }_{t} ) \big) \\
& \quad - \sum_{j=2}^{2n} j \Big( ( j - 1 ) {M}^{ 1, j-2 } ( s, {\chi }_{s} - {\chi }_{t} ) \mathfrak{X}_{s} + \mathcal{M}^{ 1, j-1 } ( s, {\chi }_{s} - {\chi }_{t} ) \Big) 
                          {\varphi }_{j} \Big( {M}^{ 0,2 } ( t,0 ), \ldots, {M}^{ 0,2n } ( t,0 ) \Big),
\end{align*}
implying that the equilibrium condition $0 = {B}_{t} {Y}^{t}_{t} + {D}_{t} \mathcal{Y}^{t}_{t}$ can be re-expressed as \begin{align*}
0 & = \big( \mathbb{M}^{ 1,0 }_{t} {B}_{t} + \mathfrak{M}^{ 1,0 }_{t} {D}_{t} \big)
    - \gamma \big( \mathbb{M}^{ 1,1 }_{t} {B}_{t} + \mathfrak{M}^{ 1,1 }_{t} {D}_{t} + \mathbb{M}^{ 1,0 }_{t} {D}_{t} \mathfrak{X}_{t} \big) \\
  & \quad - \sum_{j=2}^{2n} \Big( \mathbb{M}^{ 1, j-1 }_{t} {B}_{t} - \mathbb{M}^{ 0, j-1 }_{t} {B}_{t} \mathbb{M}^{ 1,0 }_{t}
                                + \mathfrak{M}^{ 1, j-1 }_{t} {D}_{t} - \mathbb{M}^{ 0, j-1 }_{t} {D}_{t} \mathfrak{M}^{ 1,0 }_{t} + ( j - 1 ) \mathbb{M}^{ 1, j-2 }_{t} {D}_{t} \mathfrak{X}_{t} \Big) \\
  & \qquad \qquad \times    j {\varphi }_{j} ( \mathbb{M}^{ 0,2 }_{t}, \ldots, \mathbb{M}^{ 0,2n }_{t} ),
\end{align*}
or equivalently, as \cref{eq: equilibrium condition of first-order :eq}.
Similarly, one can obtain that ${Z}^{ t, \bar{u} }_{t} \le 0$ is equivalent to \cref{eq: equilibrium condition of second-order :eq}.
Thus, the proof is complete.

\section{Standard moment estimate results}
\label{app: Standard moment estimate results}

Let us consider the SDE $d {x}_{t} = ( {A}_{t} {x}_{t} + {\alpha }_{t} ) dt + ( {B}_{t} {x}_{t} + {\beta }_{t} ) d {W}_{t}$ with $\mathbb{F}$-adapted parameters,
where $( A,B )$ are essentially bounded and $( \alpha, \beta )$ satisfies $\mathbb{E} [ ( \int_{0}^{T} | {\alpha }_{s} | ds )^{2k} + ( \int_{0}^{T} | {\beta }_{s} |^{2} ds )^{k} ] < \infty $ with some $k \ge 1$.
Afterward, through straightforward calculations, with the use of Burkholder's inequality and H\"older's inequality,
\begin{align*}
  \mathbb{E}_{t} \bigg[ \sup_{ \tau \in [ t, t + \varepsilon ] } | {x}_{\tau } |^{2k} \bigg]
& \le {5}^{2k-1} \bigg( | {x}_{t} |^{2k}
                      + \mathbb{E}_{t} \bigg[ \sup_{ \tau \in [ t, t + \varepsilon ] } \Big| \int_{t}^{\tau } {A}_{s} {x}_{s} ds \Big|^{2k} \bigg]
                      + \mathbb{E}_{t} \bigg[ \sup_{ \tau \in [ t, t + \varepsilon ] } \Big| \int_{t}^{\tau } {\alpha }_{s} ds \Big|^{2k} \bigg] \\
& \qquad \qquad       + \mathbb{E}_{t} \bigg[ \sup_{ \tau \in [ t, t + \varepsilon ] } \Big| \int_{t}^{\tau } {B}_{s} {x}_{s} d {W}_{s} \Big|^{2k} \bigg]
                      + \mathbb{E}_{t} \bigg[ \sup_{ \tau \in [ t, t + \varepsilon ] } \Big| \int_{t}^{\tau } {\beta }_{s} d {W}_{s} \Big|^{2k} \bigg] \bigg) \\
& \le K \bigg( | {x}_{t} |^{2k}
             + \mathbb{E}_{t} \bigg[ \int_{t}^{ t + \varepsilon } | {A}_{s} {x}_{s} |^{2k} ds \bigg]
             + \mathbb{E}_{t} \bigg[ \Big( \int_{t}^{ t + \varepsilon } | {\alpha }_{s} | ds \Big)^{2k} \bigg] \\
& \qquad ~~  + \mathbb{E}_{t} \bigg[ \Big( \int_{t}^{ t + \varepsilon } | {B}_{s} {x}_{s} |^{2} ds \Big)^{k} \bigg]
             + \mathbb{E}_{t} \bigg[ \Big( \int_{t}^{ t + \varepsilon } | {\beta }_{s} |^{2} ds \Big)^{k} \bigg] \bigg) \\
& \le K \bigg( | {x}_{t} |^{2k}
             + \mathbb{E}_{t} \bigg[ \Big( \int_{t}^{T} | {\alpha }_{s} | ds \Big)^{2k} \bigg]
             + \mathbb{E}_{t} \bigg[ \Big( \int_{t}^{T} | {\beta }_{s} |^{2} ds \Big)^{k} \bigg] \\
& \qquad ~~  + \Big( \esssup_{ [ 0,T ] \times \Omega } |A|^{2k} + \esssup_{ [ 0,T ] \times \Omega } |B|^{2k} \Big) \int_{0}^{\varepsilon } \mathbb{E}_{t} \bigg[ \sup_{ \tau \in [ t, t + s ] } | {x}_{\tau } |^{2k} \bigg] ds \bigg)
\end{align*}
for some constant $K > 0$ that relies only on $k$.
By applying Gr\"onwall's inequality to $\mathbb{E}_{t} [ \sup_{ \tau \in [ t, t + \cdot ] } | {x}_{\tau } |^{2k} ]$, one can conclude that there exists a constant $K > 0$ that relies only on $( k,T )$ and the essential supremum of $( A,B )$ such that
\begin{equation*}
\mathbb{E}_{t} \bigg[ \sup_{ \tau \in [ t,T ] } | {x}_{\tau } |^{2k} \bigg] 
\le K \bigg( | {x}_{t} |^{2k} + \mathbb{E}_{t} \bigg[ \Big( \int_{t}^{T} | {\alpha }_{s} | ds \Big)^{2k} \bigg] + \mathbb{E}_{t} \bigg[ \Big( \int_{t}^{T} | {\beta }_{s} |^{2} ds \Big)^{k} \bigg] \bigg),
\end{equation*}
which produces a slightly stronger result than \cite[Lemma 3.4.2]{Yong-Zhou-1999} through the generalized Minkowski inequality, i.e.,
\begin{equation*}
\mathbb{E}_{t} \bigg[ \Big( \int_{t}^{T} | {\alpha }_{s} | ds \Big)^{2k} \bigg] \le \Big( \int_{t}^{T} ( \mathbb{E}_{t} [ | {\alpha }_{s} |^{2k} ] )^{ \frac{1}{2k} } ds \Big)^{2k}, \quad
\mathbb{E}_{t} \bigg[ \Big( \int_{t}^{T} | {\beta }_{s} |^{2} ds \Big)^{k} \bigg] \le \Big( \int_{t}^{T} ( \mathbb{E}_{t} [ | {\beta }_{s} |^{2k} ] )^{ \frac{1}{k} } ds \Big)^{k}.
\end{equation*}
For brevity, hereafter, we let $K$ be a generic $( t, \varepsilon, \zeta )$-independent constant that may differ in different places.
As $( {b}_{x}, {b}_{xx}, {\sigma }_{x}, {\sigma }_{xx} )$ are uniformly bounded, we have
\begin{equation*}
\mathbb{E}_{t} \bigg[ \sup_{ s \in [ t,T ] } | {y}^{ t, \varepsilon, \zeta }_{s} |^{2k} \bigg] 
\le K \mathbb{E}_{t} \bigg[ \Big( \int_{t}^{ t + \varepsilon } | {\delta }^{\zeta } \sigma (s) |^{2} ds \Big)^{k} \bigg] 
\le K | \zeta |^{2k} {\varepsilon }^{k}
\end{equation*}
and
\begin{align*}
  \mathbb{E}_{t} \bigg[ \sup_{ s \in [ t,T ] } | {z}^{ t, \varepsilon, \zeta }_{s} |^{2k} \bigg]
& \le K \bigg( \mathbb{E}_{t} \bigg[ \Big( \int_{t}^{T} \Big| \frac{1}{2} {b}_{xx} (s) ( {y}^{ t, \varepsilon, \zeta }_{s} )^{2} 
                                                            + {1}_{\{ s \in [ t, t + \varepsilon ) \}} {\delta }^{\zeta } b (s) \Big| ds \Big)^{2k} \bigg] \\
& \qquad ~~  + \mathbb{E}_{t} \bigg[ \Big( \int_{t}^{T} \Big| \frac{1}{2} {\sigma }_{xx} (s) ( {y}^{ t, \varepsilon, \zeta }_{s} )^{2} 
                                                         + {1}_{\{ s \in [ t, t + \varepsilon ) \}} {\delta }^{\zeta } {\sigma }_{x} (s) {y}^{ t, \varepsilon, \zeta }_{s} \Big|^{2} ds \Big)^{k} \bigg] \bigg) \\
& \le K \bigg( | \zeta |^{4k} {\varepsilon }^{2k} 
             + \mathbb{E}_{t} \bigg[ \Big( \int_{t}^{ t + \varepsilon } | {\delta }^{\zeta } b (s) | ds \Big)^{2k} \bigg] 
             + \mathbb{E}_{t} \bigg[ \Big( \int_{t}^{ t + \varepsilon } | {y}^{ t, \varepsilon, \zeta }_{s} |^{2} ds \Big)^{k} \bigg] \bigg) \\
& \le K | \zeta \vee 1 |^{4k} {\varepsilon }^{2k}.
\end{align*}
Furthermore, for the following linearized SDE of $\tilde{y}^{ t, \varepsilon, \zeta }:= {X}^{ \bar{u}^{ t, \varepsilon, \zeta } } - {X}^{ \bar{u} }$, i.e.,
\begin{align*}
  d \tilde{y}^{ t, \varepsilon, \zeta }_{s}
& = \bigg( \tilde{y}^{ t, \varepsilon, \zeta }_{s} \int_{0}^{1} {b}_{x} \big( s, \theta {X}^{ \bar{u} }_{s} + ( 1 - \theta ) {X}^{ \bar{u}^{ t, \varepsilon, \zeta } }_{s}, \bar{u}^{ t, \varepsilon, \zeta }_{s} \big) d \theta
         + {1}_{\{ s \in [ t, t + \varepsilon ) \}} {\delta }^{\zeta } b (s) \bigg) ds \\
& \quad + \bigg( \tilde{y}^{ t, \varepsilon, \zeta }_{s} \int_{0}^{1} {\sigma }_{x} \big( s, \theta {X}^{ \bar{u} }_{s} + ( 1 - \theta ) {X}^{ \bar{u}^{ t, \varepsilon, \zeta } }_{s}, \bar{u}^{ t, \varepsilon, \zeta }_{s} \big) d \theta
               + {1}_{\{ s \in [ t, t + \varepsilon ) \}} {\delta }^{\zeta } \sigma (s) \bigg) d {W}_{s},
\end{align*}
we have
\begin{align*}
  \mathbb{E}_{t} \bigg[ \sup_{ s \in [ t,T ] } | \tilde{y}^{ t, \varepsilon, \zeta }_{s} |^{2k} \bigg] 
& \le K \bigg( \mathbb{E}_{t} \bigg[ \Big( \int_{t}^{ t + \varepsilon } | {\delta }^{\zeta } b (s) | ds \Big)^{2k} \bigg] 
             + \mathbb{E}_{t} \bigg[ \Big( \int_{t}^{ t + \varepsilon } | {\delta }^{\zeta } \sigma (s) |^{2} ds \Big)^{k} \bigg] \bigg) \\
& \le K | \zeta |^{2k} ( {\varepsilon }^{2k} + {\varepsilon }^{k} ).
\end{align*}
In the same manner, since
\begin{align*}
  f ( s, {X}^{ \bar{u}^{ t, \varepsilon, \zeta } }_{s}, \bar{u}^{ t, \varepsilon, \zeta }_{s} ) 
& = f(s) + {1}_{\{ s \in [ t, t + \varepsilon ) \}} {\delta }^{\zeta } f (s) 
  + ( {X}^{ \bar{u}^{ t, \varepsilon, \zeta } }_{s} - {X}^{ \bar{u} }_{s} ) \big( {f}_{x} (s) + {1}_{\{ s \in [ t, t + \varepsilon ) \}} {\delta }^{\zeta } {f}_{x} (s) \big) \\
& \quad + ( {X}^{ \bar{u}^{ t, \varepsilon, \zeta } }_{s} - {X}^{ \bar{u} }_{s} )^{2} 
          \int_{0}^{1} \theta {f}_{xx} \big( s, \theta {X}^{ \bar{u} }_{s} + ( 1 - \theta ) {X}^{ \bar{u}^{ t, \varepsilon, \zeta } }_{s}, \bar{u}^{ t, \varepsilon, \zeta }_{s} \big) d \theta
\end{align*}
for $f = b, \sigma $, one obtains the following linearized SDE of $\tilde{z}^{ t, \varepsilon, \zeta } := {X}^{ \bar{u}^{ t, \varepsilon, \zeta } } - {X}^{ \bar{u} } - {y}^{ t, \varepsilon, \zeta }$:
\begin{align*}
  d \tilde{z}^{ t, \varepsilon, \zeta }_{s}
& = \bigg( \tilde{z}^{ t, \varepsilon, \zeta }_{s} \big( {b}_{x} (s) + {1}_{\{ s \in [ t, t + \varepsilon ) \}} {\delta }^{\zeta } {b}_{x} (s) \big)
         + {1}_{\{ s \in [ t, t + \varepsilon ) \}} {\delta }^{\zeta } b (s) 
         + {1}_{\{ s \in [ t, t + \varepsilon ) \}} {\delta }^{\zeta } {b}_{x} (s) {y}^{ t, \varepsilon, \zeta }_{s} \\
& \quad \qquad + ( \tilde{y}^{ t, \varepsilon, \zeta }_{s} )^{2} 
                 \int_{0}^{1} \theta {b}_{xx} \big( s, \theta {X}^{ \bar{u} }_{s} + ( 1 - \theta ) {X}^{ \bar{u}^{ t, \varepsilon, \zeta } }_{s}, \bar{u}^{ t, \varepsilon, \zeta }_{s} \big) d \theta \bigg) ds \\
& \quad + \bigg( \tilde{z}^{ t, \varepsilon, \zeta }_{s} \big( {\sigma }_{x} (s) + {1}_{\{ s \in [ t, t + \varepsilon ) \}} {\delta }^{\zeta } {\sigma }_{x} (s) \big)
               + {1}_{\{ s \in [ t, t + \varepsilon ) \}} {\delta }^{\zeta } {\sigma }_{x} (s) {y}^{ t, \varepsilon, \zeta }_{s} \\
& \quad \qquad + ( \tilde{y}^{ t, \varepsilon, \zeta }_{s} )^{2} 
                 \int_{0}^{1} \theta {\sigma }_{xx} \big( s, \theta {X}^{ \bar{u} }_{s} + ( 1 - \theta ) {X}^{ \bar{u}^{ t, \varepsilon, \zeta } }_{s}, \bar{u}^{ t, \varepsilon, \zeta }_{s} \big) d \theta \bigg) d {W}_{s},
\end{align*}
and hence,
\begin{align*}
  \mathbb{E}_{t} \bigg[ \sup_{ s \in [ t,T ] } | \tilde{z}^{ t, \varepsilon, \zeta }_{s} |^{2k} \bigg]
& \le K \bigg( \mathbb{E}_{t} \bigg[ \Big( \int_{t}^{ t + \varepsilon } | {\delta }^{\zeta } b (s) | ds \Big)^{2k} \bigg]
             + \mathbb{E}_{t} \bigg[ \Big( \int_{t}^{ t + \varepsilon } | {\delta }^{\zeta } {b}_{x} (s) {y}^{ t, \varepsilon, \zeta }_{s} | ds \Big)^{2k} \bigg] \\
& \qquad ~~  + \mathbb{E}_{t} \bigg[ \Big( \int_{t}^{ t + \varepsilon } | {\delta }^{\zeta } {\sigma }_{x} (s) {y}^{ t, \varepsilon, \zeta }_{s} |^{2} ds \Big)^{k} \bigg]
             + \mathbb{E}_{t} \bigg[ \sup_{ s \in [ t,T ] } | \tilde{y}^{ t, \varepsilon, \zeta }_{s} |^{4k} \bigg] \bigg) \\
& \le K \big( | \zeta |^{2k} {\varepsilon }^{2k} + | \zeta |^{2k} {\varepsilon }^{3k} + | \zeta |^{2k} {\varepsilon }^{2k} + | \zeta |^{4k} ( {\varepsilon }^{4k} + {\varepsilon }^{2k} ) \big) \\
& \le K ( | \zeta \vee 1 |^{4k} {\varepsilon }^{2k} + | \zeta |^{2k} {\varepsilon }^{3k} + | \zeta |^{4k} {\varepsilon }^{4k} ).
\end{align*}
Finally, we consider the following linearized SDE of ${\eta }^{ t, \varepsilon, \zeta } := {X}^{ \bar{u}^{ t, \varepsilon, \zeta } } - {X}^{ \bar{u} } - {y}^{ t, \varepsilon, \zeta } - {z}^{ t, \varepsilon, \zeta }$:
\begin{align*}
  d {\eta }^{ t, \varepsilon, \zeta }_{s} 
& = \bigg( \big( {b}_{x} (s) + {1}_{\{ s \in [ t, t + \varepsilon ) \}} {\delta }^{\zeta } {f}_{x} (s) \big) {\eta }^{ t, \varepsilon, \zeta }_{s} 
         + {1}_{\{ s \in [ t, t + \varepsilon ) \}} {\delta }^{\zeta } {b}_{x} (s) ( {y}^{ t, \varepsilon, \zeta }_{s} + {z}^{ t, \varepsilon, \zeta }_{s} ) \\
& \qquad + ( {y}^{ t, \varepsilon, \zeta }_{s} + \tilde{z}^{ t, \varepsilon, \zeta }_{s} )^{2} 
           \int_{0}^{1} \theta {b}_{xx} \big( s, \theta {X}^{ \bar{u} }_{s} + ( 1 - \theta ) {X}^{ \bar{u}^{ t, \varepsilon, \zeta } }_{s}, \bar{u}^{ t, \varepsilon, \zeta }_{s} \big) d \theta
         - \frac{1}{2} {b}_{xx} (s) ( {y}^{ t, \varepsilon, \zeta }_{s} )^{2} \bigg) ds \\
& \quad + \bigg( \big( {\sigma }_{x} (s) + {1}_{\{ s \in [ t, t + \varepsilon ) \}} {\delta }^{\zeta } {\sigma }_{x} (s) \big) {\eta }^{ t, \varepsilon, \zeta }_{s} 
               + {1}_{\{ s \in [ t, t + \varepsilon ) \}} {\delta }^{\zeta } {\sigma }_{x} (s) {z}^{ t, \varepsilon, \zeta }_{s} \\
& \quad \qquad + ( {y}^{ t, \varepsilon, \zeta }_{s} + \tilde{z}^{ t, \varepsilon, \zeta }_{s} )^{2} 
                  \int_{0}^{1} \theta {\sigma }_{xx} \big( s, \theta {X}^{ \bar{u} }_{s} + ( 1 - \theta ) {X}^{ \bar{u}^{ t, \varepsilon, \zeta } }_{s}, \bar{u}^{ t, \varepsilon, \zeta }_{s} \big) d \theta
               - \frac{1}{2} {\sigma }_{xx} (s) ( {y}^{ t, \varepsilon, \zeta }_{s} )^{2} \bigg) d {W}_{s}.
\end{align*}
Suppose that $\varepsilon < 1$. Combining
\begin{equation*}
\mathbb{E}_{t} \bigg[ \Big( \int_{t}^{T} ( {y}^{ t, \varepsilon, \zeta }_{s} )^{2} \big| {b}_{xx} \big( s, {X}^{ \bar{u} }_{s}, \bar{u}^{ t, \varepsilon, \zeta }_{s} \big) - {b}_{xx} (s) \big| ds \Big)^{2k} \bigg] \\
\le K {\varepsilon }^{2k} \mathbb{E}_{t} \bigg[ \sup_{ s \in [ t,T ] } | {y}^{ t, \varepsilon, \zeta } |^{4k} \bigg]
\le K | \zeta |^{4k} {\varepsilon }^{4k}
\end{equation*}
and the following moment estimates derived by H\"older's inequality, i.e.,
\begin{align*}
& \mathbb{E}_{t} \bigg[ \Big( \int_{t}^{T} \big| ( {y}^{ t, \varepsilon, \zeta }_{s} + \tilde{z}^{ t, \varepsilon, \zeta }_{s} )^{2} - ( {y}^{ t, \varepsilon, \zeta }_{s} )^{2} \big| ds
                              \int_{0}^{1} \theta {b}_{xx} \big( s, \theta {X}^{ \bar{u} }_{s} + ( 1 - \theta ) {X}^{ \bar{u}^{ t, \varepsilon, \zeta } }_{s}, \bar{u}^{ t, \varepsilon, \zeta }_{s} \big) d \theta  \Big)^{2k} \bigg] \\
& \le K \bigg( \mathbb{E}_{t} \bigg[ \sup_{ s \in [ t,T ] } | {y}^{ t, \varepsilon, \zeta } \tilde{z}^{ t, \varepsilon, \zeta } |^{2k} \bigg]
             + \mathbb{E}_{t} \bigg[ \sup_{ s \in [ t,T ] } | \tilde{z}^{ t, \varepsilon, \zeta } |^{4k} \bigg] \bigg) 
  \le K | \zeta \vee 1 |^{6k} {\varepsilon }^{3k}
\end{align*}
and
\begin{align*}
& \frac{1}{ {\varepsilon }^{2k} }
  \mathbb{E}_{t} \bigg[ \Big( \int_{t}^{T} ( {y}^{ t, \varepsilon, \zeta }_{s} )^{2} ds 
                              \int_{0}^{1} \theta \big| {b}_{xx} \big( s, \theta {X}^{ \bar{u} }_{s} + ( 1 - \theta ) {X}^{ \bar{u}^{ t, \varepsilon, \zeta } }_{s}, \bar{u}^{ t, \varepsilon, \zeta }_{s} \big)
                                                      - {b}_{xx} \big( s, {X}^{ \bar{u} }_{s}, \bar{u}^{ t, \varepsilon, \zeta }_{s} \big) \big| d \theta \Big)^{2k} \bigg] \\
& \le K | \zeta |^{4k}
      \bigg| \mathbb{E}_{t} \bigg[ \Big( \int_{t}^{T} ds \int_{0}^{1} \max_{ v \in \{ 0,1 \} } 
                                         \big| {b}_{xx} \big( s, \theta {X}^{ \bar{u} }_{s} + ( 1 - \theta ) {X}^{ \bar{u}^{ t, \varepsilon, \zeta } }_{s}, \bar{u}_{s} + v \zeta \big) 
                                             - {b}_{xx} \big( s, {X}^{ \bar{u} }_{s}, \bar{u}_{s} + v \zeta \big) \big| d \theta \Big)^{4k} \bigg] \bigg|^{ \frac{1}{2} },
\end{align*}
which tends to $0$ as $\varepsilon \downarrow 0$ because of the dominated convergence theorem, one obtains
\begin{align*}
\lim_{ \varepsilon \downarrow 0 } \frac{1}{ {\varepsilon }^{2k} } 
\mathbb{E}_{t} \bigg[ \Big( \int_{t}^{T} \Big| ( {y}^{ t, \varepsilon, \zeta }_{s} + \tilde{z}^{ t, \varepsilon, \zeta }_{s} )^{2} 
                                               \int_{0}^{1} \theta {b}_{xx} \big( s, \theta {X}^{ \bar{u} }_{s} + ( 1 - \theta ) {X}^{ \bar{u}^{ t, \varepsilon, \zeta } }_{s}, \bar{u}^{ t, \varepsilon, \zeta }_{s} \big) d \theta 
\qquad \qquad \quad \\                       - \frac{1}{2} {b}_{xx} (s) ( {y}^{ t, \varepsilon, \zeta }_{s} )^{2} \Big| ds \Big)^{2k} \bigg]
= 0.
\end{align*}
Similarly, one also obtains
\begin{align*}
\lim_{ \varepsilon \downarrow 0 } \frac{1}{ {\varepsilon }^{2k} } 
\mathbb{E}_{t} \bigg[ \Big( \int_{t}^{T} \Big| ( {y}^{ t, \varepsilon, \zeta }_{s} + \tilde{z}^{ t, \varepsilon, \zeta }_{s} )^{2} 
                                               \int_{0}^{1} \theta {\sigma }_{xx} \big( s, \theta {X}^{ \bar{u} }_{s} + ( 1 - \theta ) {X}^{ \bar{u}^{ t, \varepsilon, \zeta } }_{s}, \bar{u}^{ t, \varepsilon, \zeta }_{s} \big) d \theta 
\qquad \qquad \quad \\                       - \frac{1}{2} {\sigma }_{xx} (s) ( {y}^{ t, \varepsilon, \zeta }_{s} )^{2} \Big|^{2} ds \Big)^{k} \bigg]
= 0.
\end{align*}
In summary, in conjunction with 
\begin{align*}
  \mathbb{E}_{t} \bigg[ \Big( \int_{t}^{ t + \varepsilon } | {\delta }^{\zeta } {b}_{x} (s) ( {y}^{ t, \varepsilon, \zeta }_{s} + {z}^{ t, \varepsilon, \zeta }_{s} ) | ds \Big)^{2k} \bigg]
& \le K {\varepsilon }^{2k} \bigg( \mathbb{E}_{t} \bigg[ \sup_{ s \in [ t,T ] } | {y}^{ t, \varepsilon, \zeta } |^{2k} \bigg] + \mathbb{E}_{t} \bigg[ \sup_{ s \in [ t,T ] } | {z}^{ t, \varepsilon, \zeta } |^{2k} \bigg] \bigg) \\
  \text{and} \quad
  \mathbb{E}_{t} \bigg[ \Big( \int_{t}^{ t + \varepsilon } | {\delta }^{\zeta } {\sigma }_{x} (s) {z}^{ t, \varepsilon, \zeta }_{s} |^{2} ds \Big)^{k} \bigg]
& \le K {\varepsilon }^{k} \mathbb{E}_{t} \bigg[ \sup_{ s \in [ t,T ] } | {z}^{ t, \varepsilon, \zeta } |^{2k} \bigg]
  \le K | \zeta \vee 1 |^{4k} {\varepsilon }^{3k},
\end{align*}
we can obtain $\mathbb{E}_{t} [ \sup_{ s \in [ t,T ] } | {\eta }^{ t, \varepsilon, \zeta }_{s} |^{2k} ] = o ( {\varepsilon }^{2k} )$.

\section{Stochastic Lebesgue differentiation theorem}
\label{app: Stochastic Lebesgue differentiation theorem}

In this paper, the diagonal process triplet $\{ ( {Y}^{t,u}_{t}, \mathcal{Y}^{t,u}_{t}, {Z}^{t,u}_{t} ) \}_{ t \in [ 0,T ) }$ belongs to
\begin{equation*}
\mathbb{L}_{\mathbb{F},loc}^{p} \big( 0,T; \mathbb{L}^{p} ( \Omega ) \big) \times \mathbb{L}_{\mathbb{F},loc}^{p} \big( 0,T; \mathbb{L}^{p} ( \Omega ) \big)
\times \mathbb{L}_{\mathbb{F},loc}^{p} \big( 0,T; \mathbb{L}^{p} ( \Omega ) \big)
\end{equation*}
merely for some $p \in ( 1, 1 + \delta ) \subset ( 1,2 )$, provided that \Cref{ass: uniform integrabiity condition with continuity} holds.
Therefore, we extend the stochastic Lebesgue differentiation theorem presented in \cite{Hu-Jin-Zhou-2017} for $\mathbb{L}_{\mathbb{F}}^{2} ( 0,T; \mathbb{L}^{2} ( \Omega ) )$ to fit our weaker integrability property.
As $\mathbb{L}_{\mathcal{F}_{T}}^{p} ( \Omega )$ is separable for all finite $p \ge 1$,
we can mirror the proof of \cite[Lemma 3.4]{Hu-Jin-Zhou-2017} to arrive at the following theorem.

\begin{theorem}\label{thm: stochastic Lebesgue differentiation theorem}
Suppose that $p > 1$ and $Y \in \mathbb{L}_{\mathbb{F},loc}^{p} ( 0,T; \mathbb{L}^{p} ( \Omega ) )$.
If 
\begin{equation*}
\lim_{ \varepsilon \downarrow 0 } \frac{1}{\varepsilon } \int_{t}^{ t + \varepsilon } \mathbb{E}_{t} [ {Y}_{s} ] ds = 0, \quad a.e. ~ t \in [ 0,T ), ~ \mathbb{P}\text{-}a.s.,
\end{equation*}
then ${Y}_{t} = 0$, a.e. $t \in [ 0,T )$, $\mathbb{P}$-a.s.
\end{theorem}

\begin{proof}
Applying the classic Lebesgue differentiation theorem produces the following:
\begin{equation*}
\lim_{\varepsilon \downarrow 0} \frac{1}{\varepsilon } \int_{t}^{ t + \varepsilon } \mathbb{E} [ | {Y}_{s} |^{p} ] ds = \mathbb{E} [ | {Y}_{t} |^{p} ], \quad
\lim_{\varepsilon \downarrow 0} \frac{1}{\varepsilon } \int_{t}^{ t + \varepsilon } \mathbb{E} [ {Y}_{s} \eta ] ds = \mathbb{E} [ {Y}_{t} \eta ], \quad a.e. ~ t \in [ 0,T ),
\end{equation*}
where $\eta $ is arbitrarily chosen from a countable dense subset $\mathcal{D} \subset \mathbb{L}_{\mathcal{F}_{T}}^{ \frac{p}{p-1} } ( \Omega ) \cap \mathbb{L}_{\mathcal{F}_{T}}^{\infty } ( \Omega )$.
Write ${\eta }_{s} = \mathbb{E}_{s} [ \eta ]$, which leads to $\mathbb{E} [ {Y}_{s} \eta ] = \mathbb{E} [ {Y}_{s} {\eta }_{s} ]$.
By Holder's inequality, Doob's maximal inequality and the dominated convergence theorem, one obtains
\begin{align*}
      \bigg| \lim_{\varepsilon \downarrow 0} \frac{1}{\varepsilon }
             \int_{t}^{t + \varepsilon } \mathbb{E} [ {Y}_{s} ( {\eta }_{s} - {\eta }_{t} ) ] ds \bigg|
& \le \lim_{\varepsilon \downarrow 0} \frac{1}{\varepsilon }
      \bigg| \int_{t}^{ t + \varepsilon } \mathbb{E} [ | {Y}_{s} |^{p} ] ds \bigg|^{ \frac{1}{p} }
      \bigg| \mathbb{E} \bigg[ \int_{t}^{t + \varepsilon } | {\eta }_{s} - {\eta }_{t} |^{ \frac{p}{p-1} } ds \bigg] \bigg|^{ \frac{p-1}{p} } \\
& \le \lim_{\varepsilon \downarrow 0} \bigg| \frac{1}{\varepsilon } \int_{t}^{ t + \varepsilon } \mathbb{E} [ | {Y}_{s} |^{p} ] ds \bigg|^{ \frac{1}{p} }
      \bigg| \mathbb{E} \bigg[ \sup_{ s \in [ t, t + \varepsilon ] } | {\eta }_{s} - {\eta }_{t} |^{ \frac{p}{p-1} } \bigg] \bigg|^{ \frac{p-1}{p} } \\
& \le p ( \mathbb{E} [ | {Y}_{t} |^{p} ] )^{ \frac{1}{p} }
      \bigg| \lim_{\varepsilon \downarrow 0} \mathbb{E} \big[ | {\eta }_{ t + \varepsilon } - {\eta }_{t} |^{ \frac{p}{p-1} } \big] \bigg|^{ \frac{p-1}{p} } \\
&   = 0,
\end{align*}
and hence,
\begin{equation*}
  \mathbb{E} [ {Y}_{t} \eta ]
= \lim_{\varepsilon \downarrow 0} \frac{1}{\varepsilon } \int_{t}^{ t + \varepsilon } \mathbb{E} [ {Y}_{s} {\eta }_{s} ] ds
= \lim_{\varepsilon \downarrow 0} \frac{1}{\varepsilon } \int_{t}^{ t + \varepsilon } \mathbb{E} [ {Y}_{s} {\eta }_{t} ] ds
= \lim_{\varepsilon \downarrow 0} \mathbb{E} \bigg[ \frac{ {\eta }_{t} }{\varepsilon } \int_{t}^{ t + \varepsilon } \mathbb{E}_{t} [ {Y}_{s} ] ds \bigg].
\end{equation*}
Since applying H\"older's inequality and Doob's maximal inequality yields
\begin{align*}
\mathbb{E} \bigg[ \bigg| \frac{1}{\varepsilon } \int_{t}^{ t + \varepsilon } \mathbb{E}_{t} [ {Y}_{s} ] ds \bigg|^{p} \bigg]
\le \frac{1}{\varepsilon } \int_{t}^{ t + \varepsilon } \mathbb{E} \big[ | \mathbb{E}_{t} [ {Y}_{s} ] |^{p} \big] ds
\le \Big( \frac{p}{p-1} \Big)^{p} \frac{1}{\varepsilon } \int_{t}^{ t + \varepsilon } \mathbb{E} [ | {Y}_{s} |^{p} ] ds
\to \Big( \frac{p}{p-1} \Big)^{p} \mathbb{E} [ | {Y}_{t} |^{p} ]
\end{align*}
as $\varepsilon \downarrow 0$, 
one can conclude that there exists a sufficiently small ${\delta }_{t} > 0$ such that $\frac{1}{\varepsilon } \int_{t}^{ t + \varepsilon } \mathbb{E}_{t} [ {Y}_{s} ] ds$ is uniformly integrable in $\varepsilon \in ( 0, {\delta }_{t} )$,
which implies that
\begin{align*}
  \lim_{\varepsilon \downarrow 0} \bigg| \mathbb{E} \bigg[ \frac{ {\eta }_{t} }{\varepsilon } \int_{t}^{ t + \varepsilon } \mathbb{E}_{t} [ {Y}_{s} ] ds \bigg] \bigg|
& \le ( \esssup | \eta | ) \lim_{\varepsilon \downarrow 0} \mathbb{E} \bigg[ \bigg| \frac{1}{\varepsilon } \int_{t}^{ t + \varepsilon } \mathbb{E}_{t} [ {Y}_{s} ] ds \bigg| \bigg] \\
& = ( \esssup | \eta | ) \mathbb{E} \bigg[ \lim_{\varepsilon \downarrow 0} \bigg| \frac{1}{\varepsilon } \int_{t}^{ t + \varepsilon } \mathbb{E}_{t} [ {Y}_{s} ] ds \bigg| \bigg] 
  = 0.
\end{align*}
In summary, for a.e. $t \in [ 0,T )$, we have that $\mathbb{E} [ {Y}_{t} \eta ] = 0$ with an arbitrarily chosen $\eta \in \mathcal{D}$. 
Therefore, ${Y}_{t} = 0$ $\mathbb{P}$-a.s. for a.e. $t \in [ 0,T )$. 
\end{proof}

Consider the linear controlled SDE \cref{eq: linear controlled SDE :eq}.
Based on ${Z}^{ t, \bar{u} } \in {C}_{\mathbb{F}} ( 0,T; \mathbb{L}^{ 1 + \delta } ( \Omega ) )$ and the dominated convergence theorem, 
we know that $\mathbb{E}_{t} [ | {Z}^{ t, \bar{u} }_{s} - {Z}^{ t, \bar{u} }_{t} | ]$ is continuous in $s$ and vanishes at $s = t$, implying that
\begin{equation*}
\mathbb{E}_{t} \bigg[ \int_{t}^{ t + \varepsilon } | {Z}^{ t, \bar{u} }_{s} - {Z}^{ t, \bar{u} }_{t} | | {D}_{s} |^{2} ds \bigg]
\le ( \esssup |D| )^{2} \int_{t}^{ t + \varepsilon } \mathbb{E}_{t} [ | {Z}^{ t, \bar{u} }_{s} - {Z}^{ t, \bar{u} }_{t} | ] ds
  = o ( \varepsilon ). 
\end{equation*}
Then, it follows from \cref{eq: spike perturbation :eq}, mirroring the proof of \Cref{thm: maximum principle} for handling the integrand ${Y}^{ t, \bar{u} }_{s} {B}_{s} + \mathcal{Y}^{ t, \bar{u} }_{s} {D}_{s}$, that
\begin{equation*}
  {J}^{t} ( \bar{u}^{ t, \varepsilon, \zeta } ) - {J}^{t} ( \bar{u} )
= \zeta \mathbb{E}_{t} \bigg[ \int_{t}^{ t + \varepsilon } ( {Y}^{ s, \bar{u} }_{s} {B}_{s} + \mathcal{Y}^{ s, \bar{u} }_{s} {D}_{s} ) ds \bigg]
+ \frac{1}{2} | \zeta |^{2} {Z}^{ t, \bar{u} }_{t} \int_{t}^{ t + \varepsilon } | {D}_{s} |^{2} ds
+ o ( \varepsilon ).
\end{equation*}
Notably, $\bar{u} \in \mathcal{U}$ is an ONEC if
\begin{equation}\label{eq: equilibrium condition for linear control problems :eq}
{Y}^{ t, \bar{u} }_{t} {B}_{t} + \mathcal{Y}^{ t, \bar{u} }_{t} {D}_{t} = 0, \quad {Z}^{ t, \bar{u} }_{t} \le 0, \quad a.e. ~ t \in [ 0,T ), ~ \mathbb{P}\text{-}a.s.
\end{equation}
Provided that $\bar{u} \in \mathcal{U}$ is an ONEC, in view of the arbitrariness of $\zeta \in \mathbb{R}$, one obtains ${Z}^{ t, \bar{u} }_{t} \le 0$ and
\begin{equation*}
\lim_{\varepsilon \downarrow 0} \frac{1}{\varepsilon }\int_{t}^{ t + \varepsilon } \mathbb{E}_{t} [ {Y}^{ s, \bar{u} }_{s} {B}_{s} + \mathcal{Y}^{ s, \bar{u} }_{s} {D}_{s} ] ds = 0, \quad a.e. ~ t \in [ 0,T ), ~ \mathbb{P}\text{-}a.s.
\end{equation*}
Applying \Cref{lem: integrabiity and uniqueness of diagonal process} and the stochastic Lebesgue differentiation theorem immediately yields 
${Y}^{ t, \bar{u} }_{t} {B}_{t} + \mathcal{Y}^{ t, \bar{u} }_{t} {D}_{t} = 0$, a.e. $t \in [ 0,T )$, $\mathbb{P}$-a.s.
That is, the necessity of the equilibrium condition \cref{eq: equilibrium condition for linear control problems :eq} holds.

\bibliographystyle{apacite}
\bibliography{references}
\end{document}